\pdfoutput=1
\documentclass[11pt]{amsart}
\usepackage{geometry}               
\usepackage[parfill]{parskip}    
\usepackage{graphicx,color,amsthm,amsmath}
\usepackage{amssymb,enumitem,multicol}
\usepackage{epstopdf}
\usepackage{pdfsync}
\usepackage{fullpage}

\theoremstyle{plain} 

\newtheorem{thm}{Theorem}

\theoremstyle{definition} 
\newtheorem{defn}[thm]{Definition}
\newtheorem{ex}{Example}

\usepackage{tikz}
\newcommand*{\circled}[2][]{\tikz[baseline=(C.base)]{
    \node[inner sep=0pt] (C) {\vphantom{1g}#2};
    \node[draw, circle, inner sep=3pt, yshift=1pt] 
        at (C.center) {\vphantom{1g}};}}

\newtheorem{innercustomgeneric}{\customgenericname}
\providecommand{\customgenericname}{}
\newcommand{\newcustomtheorem}[2]{%
  \newenvironment{#1}[1]
  {%
   \renewcommand\customgenericname{#2}%
   \renewcommand\theinnercustomgeneric{##1}%
   \innercustomgeneric
  }
  {\endinnercustomgeneric}
}

\newcustomtheorem{customthm}{Theorem}
\newcustomtheorem{customlemma}{Lemma}
\newcustomtheorem{customprop}{Proposition}
\title{Dihedral linking invariants}

\author[P.\ Cahn]{Patricia Cahn}
\address{Department of Mathematics and Statistics, Smith College, USA}
\email{pcahn@smith.edu}
\author[E.\ Catania]{Elise Catania}
\address{School of Mathematics, College of Science and Engineering, University of Minnesota, USA}
\email{catan042@umn.edu}
\author[S.\ Chimgee]{Sarangoo Chimgee}
\address{Department of Mathematics and Statistics, Smith College, USA}
\email{sara.chimgee@gmail.com}
\author[O.\ Del Guercio]{Olivia Del Guercio}
\address{Department of Mathematics, Rice University, USA}
\email{od6@rice.edu}
\author[J.\ Kendrick]{Jack Kendrick}
\address{Department of Mathematics, University of Washington, USA}
\email{jackgk@uw.edu}
\begin{document}
\maketitle

\begin{abstract} A Fox p-colored knot $K$ in $S^3$ gives rise to a corresponding $p$-fold dihedral branched cover $M$ of $S^3$ along $K$.  The pre-image of the knot $K$ under the covering map is a $\dfrac{p+1}{2}$-component link $L$ in $M$, and the set of pairwise linking numbers of the components of $L$ is an invariant of $K$. This powerful invariant played a key role in the development of early knot tables, and appears in formulas for many other important knot and manifold invariants. We give an algorithm for computing this invariant for all odd $p$, generalizing an algorithm of Perko. We then extend this  algorithm to compute linking numbers of arbitrary curves in a $p$-fold dihedral branched cover of $S^3$ along $K$. As an application, we compute Kjuchukova's ribbon obstruction $\Xi_p$ using a method of the first author and Kjuchukova.  We also tabulate the dihedral linking invariant for all $p$-colorings of prime knots of crossing number less than or equal to 13, with $p\geq 3$ prime.  Finally, we demonstrate the strength of the dihedral linking invariant by comparing it to several polynomial invariants. For example, the dihedral linking invariant distinguishes more than 98\% of the 1183 prime non-mutant knot pairs with the same Fox coloring invariant and the same HOMFLY-PT polynomial through 13 crossings.

\end{abstract}

\section{Introduction}

Linking numbers of curves in branched covers of the 3-sphere along a knot $K$ are a rich source of invariants of $K$.  One such invariant is the {\it $p$-dihedral linking invariant}, which is defined for Fox $p$-colorable knots, where $p$ is odd.  A Fox $p$-coloring of $K$ gives rise to a $p$-fold branched covering $f: M\rightarrow S^3$ along $K$.  The knot $K$ lifts to a $(p+1)/2$-component link in $M$, and the set of pairwise linking numbers of its components is the dihedral linking invariant of the $p$-colored knot $K$.  

We refer the reader to ~\cite{perko2016historical} for details on the rich history of non-cyclic branched covering spaces in knot theory and their associated linking invariants, particularly the key role they played in the development of early knot tables.  For example, using computations of Bankwitz and Schumann \cite{bankwitz1934viergeflechte}, Reidemeister used the dihedral linking invariant to distinguish two knots with the same Alexander polynomial ~\cite{reidemeister1938knot}.  

Linking invariants derived from branched covers, and dihedral covers in particular, also appear in formulas for many other knot and manifold invariants, for example: Cappell and Shaneson's formula for the Rokhlin $\mu$-invariant~\cite{CS1975invariants}, Litherland's formula for the Casson-Gordon invariants~\cite{litherland1980formula}, and Kjuchukova's $\Xi_p$-invariant ~\cite{kjuchukova2018dihedral,cahnkjuchukova2018computing}, which gives an obstruction to a knot being homotopy-ribbon~\cite{cahnkjuchukova2017singbranchedcovers, geske2021signatures}, as well as bounds on 4-genera associated to a $p$-colored knot ~\cite{cahnkju2018genus}.

Perko gave an explicit geometric algorithm for computing the dihedral linking invariant for all 3-colorable knots in his 1964 thesis, and tabulated the invariant for all such knots up to 11 crossings ~\cite{perko1964thesis}.  The primary challenge in generalizing Perko's algorithm to $p>3$ is to find a simple way to combinatorially encode a cell structure on the covering space, which also lends itself to fast computation.  We develop such a method and give a fully general geometric algorithm for computing the dihedral linking invariant for all $p$-colorable knots.  We then tabulate the invariant for $p$-colorable prime knots of crossing number $\leq 13$. Code and additional data can be found at ~\cite{cahngithubdihedrallinking}. 

Hartley and Murasugi gave general, algebraic formulas for linking numbers of branch curves via the Reidemeister-Schreier algorithm, contrasting them with Perko's method in ~\cite{perko1964thesis}, which they say requires ``considerable geometric intuition'' ~\cite{hartley1977covering}.  Hartley used the method in \cite{hartley1977covering} to show that many knots are non-invertible using linking numbers derived from other metacyclic covering spaces \cite{hartley1983identifying}; in that paper, he also discusses his implementation of the Hartley-Murasugi method.  A number of additional methods have been introduced for computing the dihedral linking invariant for particular knots and knot families.  For example, in ~\cite[Theorem 4]{perko1976dihedral}, Perko gave an elegant method for computing linking numbers for 2-bridge knots from the Schubert normal form.  Perko also introduced a visually appealing way of computing the linking invariant from a {\it Perko surface}, a Seifert-like surface constructed from the diagram of a $p$-colored knot~\cite{perko2016historical}; see Figure ~\ref{perko.fig}. However, such surfaces may not exist for a given diagram of the knot, and it is not known whether every colorable knot admits a diagram with a Perko surface.

\begin{figure}[htbp]
	\includegraphics[width=2in]{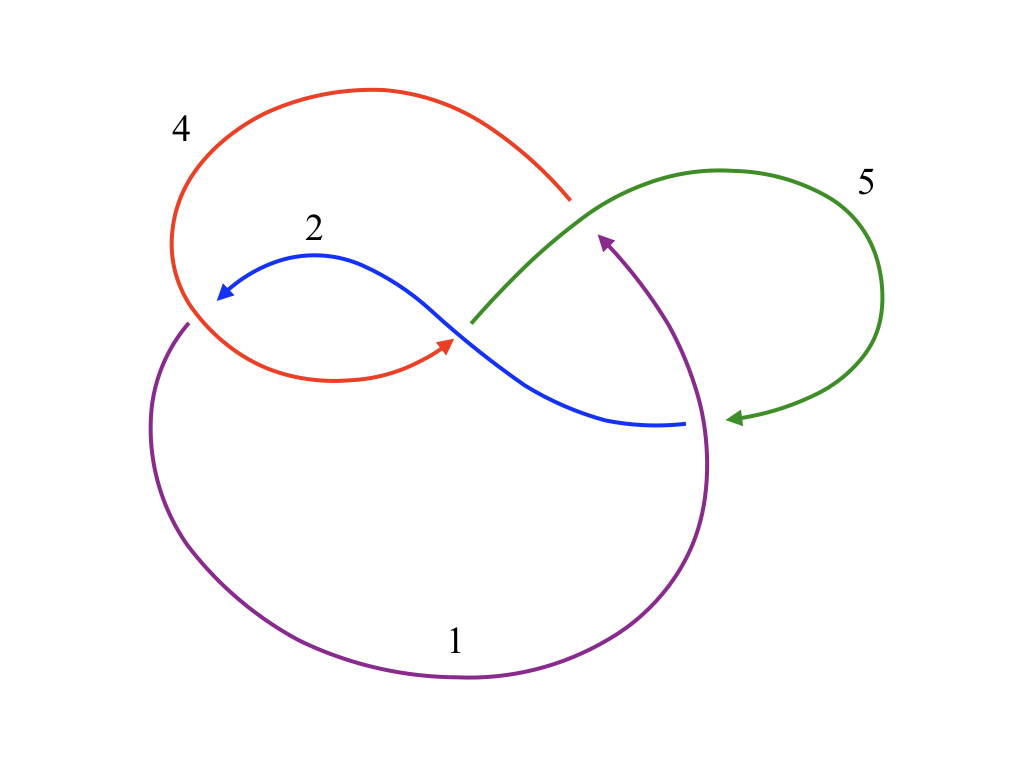}\includegraphics[width=2in]{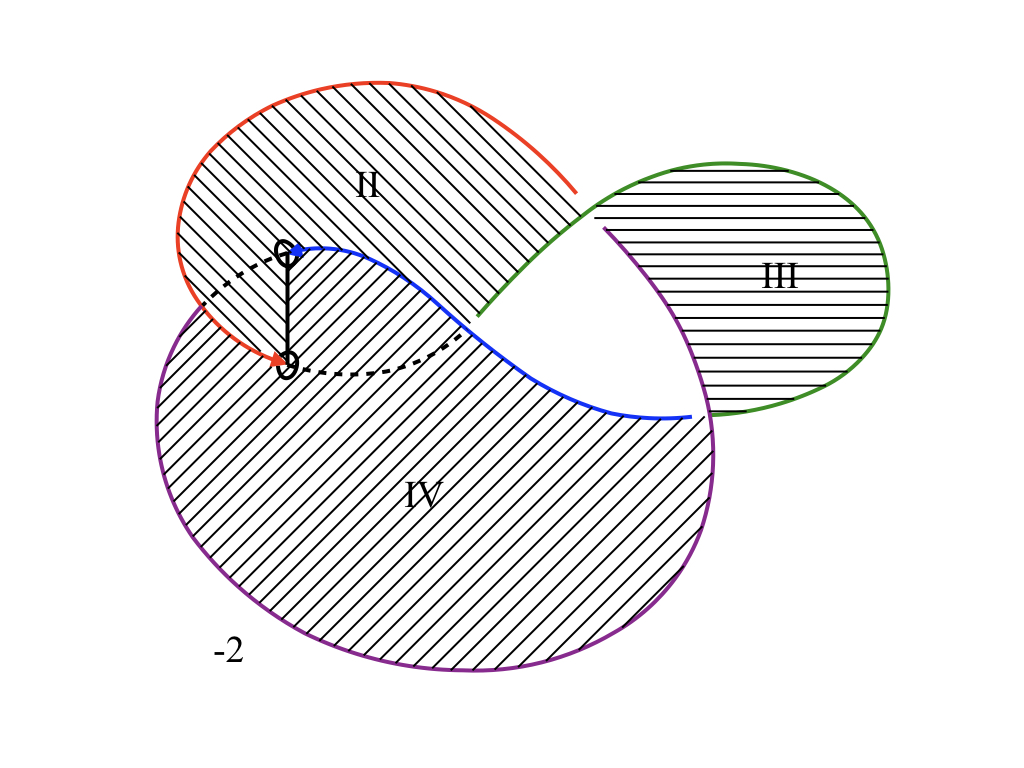}\includegraphics[width=2in]{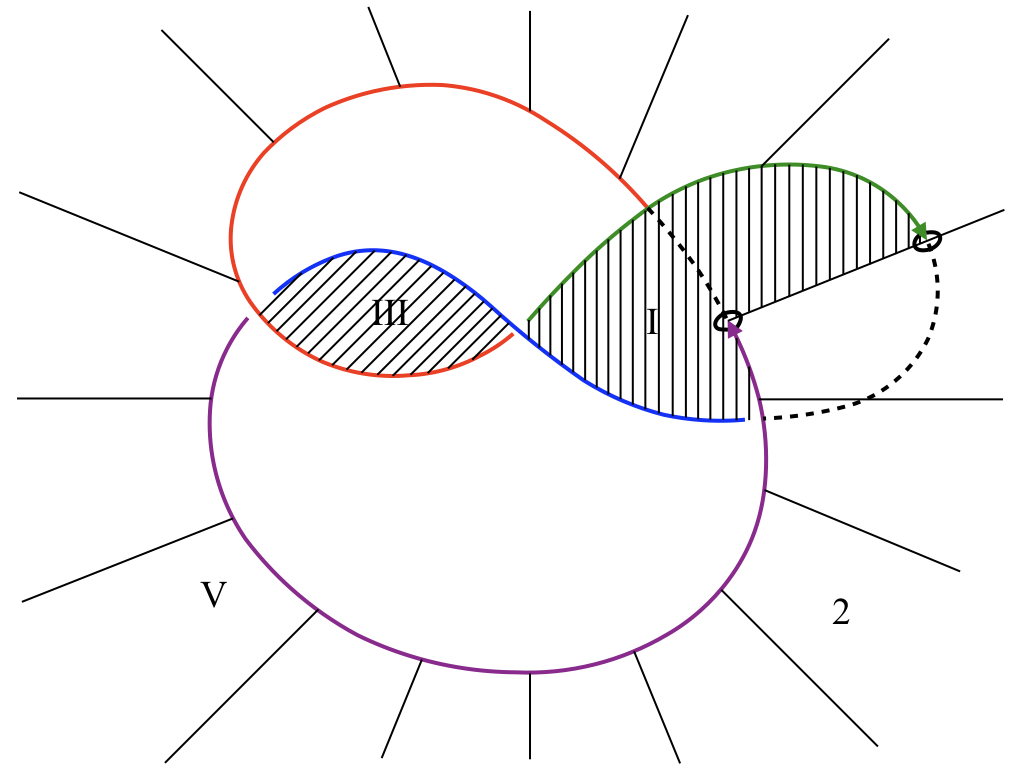}
	\caption{Perko's surfaces bounding the branch curves of the Figure-8 knot, which has 5-dihedral linking invariant $\{0,2,-2\}$. Figure adapted from Figure 1 of ~\cite{perko2016historical}.}
	\label{perko.fig}
\end{figure}

We now review the necessary background information.  Consider a diagram $D_K$ of $K$, and let $p$ be an odd number. A Fox $p$-coloring of $D_K$ is an assignment of the values $\{0,1,\dots,p-1\}$ to the arcs of $D_K$ such that at each crossing, $a+b \equiv 2c \mod p$, where $a$ and $c$ are the values on the understands and $c$ is the value on the overstrand.  Equivalently, a Fox $p$-coloring is described by a surjection $\rho:\pi_1(S^3-K)\twoheadrightarrow D_p$, where $D_p$ denotes the dihedral group of order $2p$.   Label the vertices of a regular $p$-gon $\{0,\dots,p-1\}$. Denote by $\text{Ref}_n\in D_p$ the reflection over a line through the $n^{th}$ vertex of a regular $p$-gon, and denote by $\text{Ref}_n(x)\in\{0,\dots,p-1\}$ the image of vertex $x$ under this reflection.  The standard meridian of an arc in the diagram $D_K$ colored $n$ is mapped to the reflection $\text{Ref}_n$, and a crossing with overstrand colored $c$ and understrands colored $a$ and $b$ satisfies $\text{Ref}_c(a)=b.$ 
 
 Given a $p$-coloring $\rho:\pi_1(S^3-K)\twoheadrightarrow D_p$, the corresponding {\it irregular $p$-fold dihedral cover of $S^3$}, denoted $M_\rho$, is the branched cover of $S^3$ along $K$ corresponding to a subgroup $\rho^{-1}(\mathbb{Z}_2)$ of $\pi_1(S^3-K)$. The knot $K$ has one lift $K^0\subset M_\rho$ of branching index 1, and $(p-1)/2$ lifts $K^1,\dots, K^{(p-1)/2}\subset M_\rho$ of branching index 2.  
 
 The linking number $\text{lk}(\alpha,\beta)\in \mathbb{Q}$ of two knots $\alpha$, $\beta$ in a closed, connected, oriented 3-manifold $M$ is defined whenever both $\alpha$ and $\beta$ are rationally null-homologous in $M$.  Suppose $\Sigma_\alpha$ is a 2-chain such that $\partial \Sigma_\alpha=k\cdot \alpha$, where $k\in \mathbb{Z}$.  Then $\text{lk}(\alpha,\beta)=\dfrac{1}{k}\left(\beta\cdot \Sigma_\alpha\right)$, where $\cdot$ denotes the algebraic intersection number.  This linking number is well-defined and symmetric ~\cite{birman1980seifert}.   If either $\alpha$ or $\beta$ represent a nonzero class in $H_1(M;\mathbb{Q}),$ we set $\text{lk}(\alpha,\beta)=\infty$. The {\it $p$-dihedral linking invariant} of $K$ together with a choice of $p$-coloring $\rho$ of $K$ is the multiset
 
 $$DLN(K,\rho) = \left\{ \text{lk}(K^i,K^j) | i\neq j \in \{0,1,\dots,(p-1)/2\}\right\}.$$

 An outline of our algorithm is as follows. Neuwirth gave a combinatorial method for constructing a branched cover of a 3-manifold $M$ along a 1-subcomplex $K$ from a choice of {\it splitting complex} for $K$, and a permutation representation of the knot group $\pi_1(M-K)\twoheadrightarrow S_n$ corresponding to the desired branched cover \cite{neuwirth2016chapter}.   We carry out this construction where $M=S^3$, the splitting complex $C$ is the cone on $K$, and the permutation representation is the coloring $\rho:\pi_1(S^3-K)\twoheadrightarrow D_p \leqslant S_{p}$. The resulting construction gives the irregular $p$-fold cover $M_\rho$ of $S^3$ along $K$, equipped with a cell structure determined by the lift $\tilde{C}$ of $C$ to $M_\rho$.  As noted in ~\cite{perko2016historical}, this cell structure on $M$ actually dates back to Wirtinger.

 We then find a rational 2-chain $\Sigma^j$ bounding each connected component $K^j$ of the singular set, by solving a $qn\times qn$ system of linear equations, where $n$ is the number of crossings in a chosen diagram of $K$ and $q=\dfrac{p-1}{2}$.  If a solution exists, we conclude $K^j$ is rationally null-homologous, and $\text{lk}(K^i,K^j)$ exists provided $K^i$ is rationally null-homologous as well.  If no solution exists, $K^j$ is not rationally null-homologous.  To compute $\text{lk}(K^i,K^j)$ we choose a push-off of $K^i$ transverse to the 2-chain $\Sigma^j$, and compute the signed intersection number $K^i\cdot \Sigma^j$.

 In Section ~\ref{fig8example.sec}, we carry out this algorithm for the Figure-8 knot.  We then turn to computing the invariant for an arbitrary $p$-colored knot $(K,\rho)$.  In Section ~\ref{setup.sec}, we introduce {\it configuration diagrams}, a fully combinatorial method for encoding the cell structure on the corresponding branched cover $M_\rho$. In Section ~\ref{2chain.sec}, we find a rational 2-chain $\Sigma^j$ with boundary $K^j$ using the configuration diagram of $(K,\rho)$, and in Section ~\ref{intersectionnumber.sec}, we compute the linking number of $K^j$ and $K^i$ using the configuration diagram of $(K,\rho)$.    Additional worked examples are in Section ~\ref{examples.sec}.

 In Section~\ref{xi.sec}, we explain how to extend the algorithm above to compute linking numbers between lifts of curves $\gamma,\delta\subset S^3-K$ to the $p$-fold dihedral cover $M_\rho$ corresponding to a $p$-coloring $\rho$ of $K$.  We call such curves {\it pseudo-branch curves}. An algorithm for computing linking numbers of lifts of pseudo-branch curves was given by the first author and Kjuchukova in \cite{cahnkjuchukova2016linking} in the case $p=3$, for dihedral covers, and for all $p$ in~\cite{cahn2023linking} in the case of cyclic covers.  Such linking numbers appear in a formula for Kjuchukova's $\Xi_p$ invariant~\cite{kjuchukova2018dihedral}, which is of interest because it gives rise to a homotopy ribbon obstruction \cite{cahnkjuchukova2017singbranchedcovers},~\cite{geske2021signatures}, and bounds on 4-genera associated to a $p$-colorable knot~\cite{cahnkju2018genus}. The first author and Kjuchukova gave a diagrammatic algorithm for computing $\Xi_p$~\cite{cahnkjuchukova2018computing}. In part because the algorithm in ~\cite{cahnkjuchukova2016linking} for computing linking numbers of pseudo-branch curves was specific to the case $p=3$, the examples in~\cite{cahnkjuchukova2018computing} are carried out only for $p=3$.  The above extension of the algorithm to the computation of linking numbers of pseudo-branch curves for arbitrary odd $p$ allows us to carry out the algorithm in~\cite{cahnkjuchukova2018computing} when $p>3$. As an application, we carry out one such example in Section~\ref{xiexample.sec}.
 
 In Section~\ref{tabulation.sec}, we give an overview of our methods for tabulating the dihedral linking invariant for prime knots of crossing number $\leq 13$ that are $p$-colorable, for prime $p\geq 3$.  The code used to produce the tabulation and the tabulation itself are available at ~\cite{cahngithubdihedrallinking}. We then compare the dihedral linking invariant to the strength of the Fox coloring invariant alone (that is, the number of non-trivial $p$-colorings of $K$ for each odd prime $p$), as well as to the Alexander, Jones, HOMFLY-PT, Kauffman and Khovanov polynomials, using data tabulated in KnotInfo~\cite{knotinfo}; the Khovanov data on KnotInfo is computed using ~\cite{KnotJob}.  The dihedral linking invariant proves to be particularly strong, especially when restricted to non-mutant knot pairs.  For each given polynomial invariant above, the dihedral linking invariant distinguishes between 94.3\% and 98.3\% of non-mutant prime knot pairs of crossing number $\leq 13$ with the same Fox coloring invariant and the same given polynomial invariant.

 {\it Acknowledgements.} This work was supported by NSF Grants DMS-1821212 and 2145384 to P. Cahn.  We thank Charles Livingston for valuable discussions regarding the tabulation of the linking invariant; Nathan Dunfield for showing us an alternative method of computing the linking invariant using SnapPy; and Sebastian Baader, who first asked us whether the linking invariant can distinguish mutant knots.

  \section{Extended Example}\label{fig8example.sec}
  
  In this section, let $K$ be the Figure-8 knot, with diagram $D_K$ and 5-coloring $\rho$ as shown in Figure ~\ref{cone.fig}. We will carry out our algorithm on $K$.

 The 3-manifold $M_\rho$ can be constructed combinatorially from the diagram $D_K$ by means of a splitting complex for $K$, namely the cone on $D_K$ \cite{neuwirth2016chapter}.  First, the cone on $D_K$ gives rise to an associated cell structure on $S^3$, as shown in Figure ~\ref{cone.fig}, with cone point at infinity.

 \begin{figure}[htbp]
 \includegraphics[width=5in]{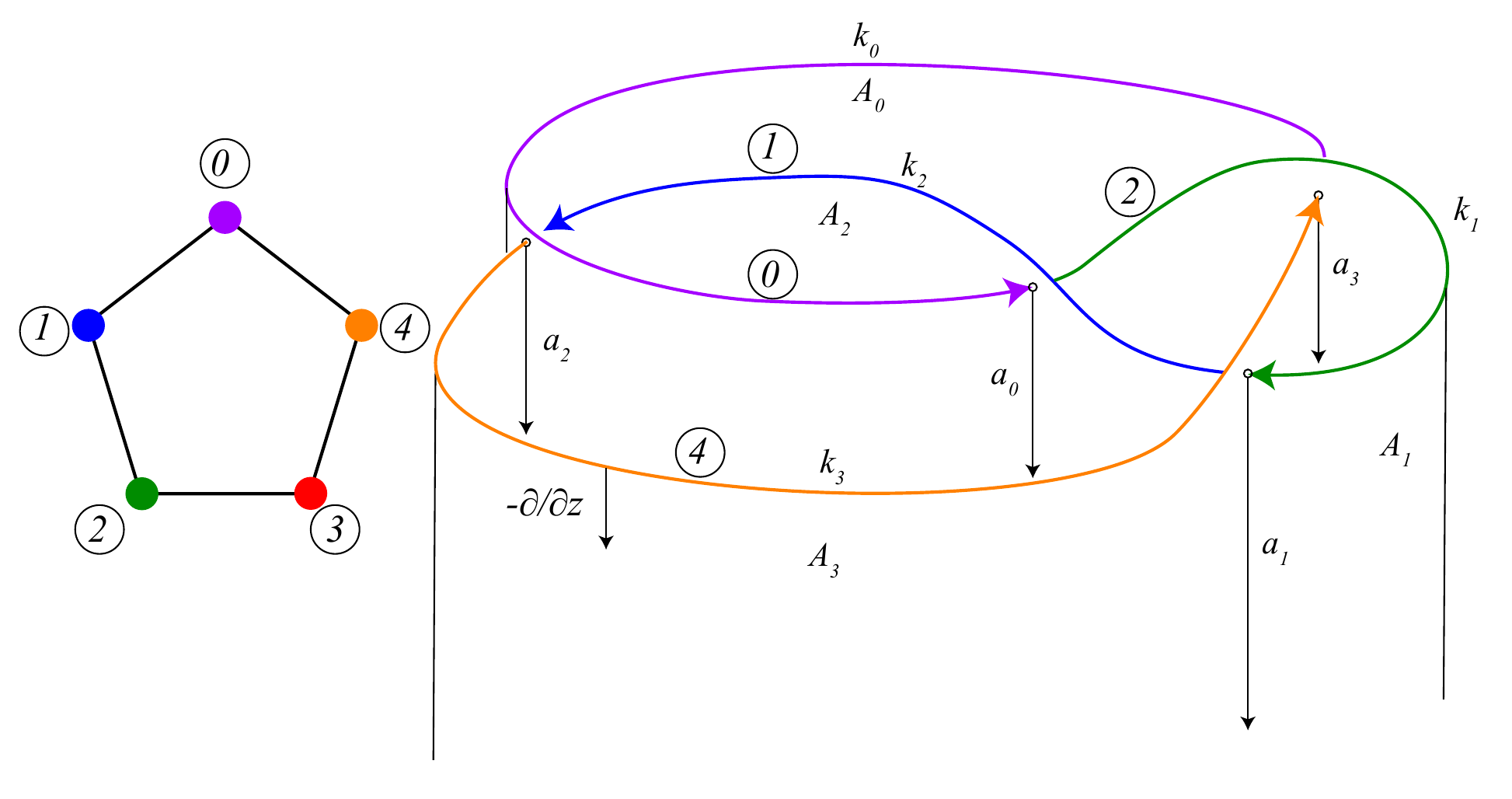}
 \caption{Cell structure on $S^3$ determined by the cone on $D_K$.}\label{cone.fig}
 \end{figure}

We label the cells as follows:
\begin{enumerate}
	\item ``horizontal" 1-cells $k_0$, $k_1$, $k_2$, $k_3$,  which are arcs in the diagram $D_K$ of $K$;
	\item ``vertical'' 1-cells $a_0$, $a_1$, $a_2$, $a_3$, below each crossing of $D_K$;
	\item ``vertical'' 2-cells $A_0$, $A_1$, $A_2$, $A_3$, below each arc of $D_K$
	\item one 3-cell, $E$, the complement of the cone
	\item one 0-cell $e$ at the cone point
	\item  a 0-cell $e_i$ at the head of each arc $k_i$
\end{enumerate}

Denote the $k$-skeleton of this cell structure by $C^k$.

Next we lift this cell structure to the dihedral cover $M_\rho$ of $S^3$, branched along $K$ determined by the 5-coloring $\rho$ of $K$ shown  in Figure ~\ref{cone.fig}.  Consider the open cover of the 2-skeleton $C^2$ which consists of: a small open 3-ball $U_i$ centered at $e_i$ and containing a segment of the over-arc at crossing $i$; a small tubular neighborhood $V_j$ of each arc $k_i$ of $D_K$ disjoint from $D_K-k_i$, and a (topological) 3-ball neighborhood $W$ of $e$ disjoint from $D_K$ and containing the remainder of the 2-skeleton $S^2$; see Figure ~\ref{skeletonnbd.fig}.  Together with the 3-cell $E$, this forms an open cover of $S^3$.  It suffices to understand the lift of the cell structure on its intersection with each of the $U_i$, $V_j$, $W$, and $E$.

\begin{figure}\includegraphics[width=3in]{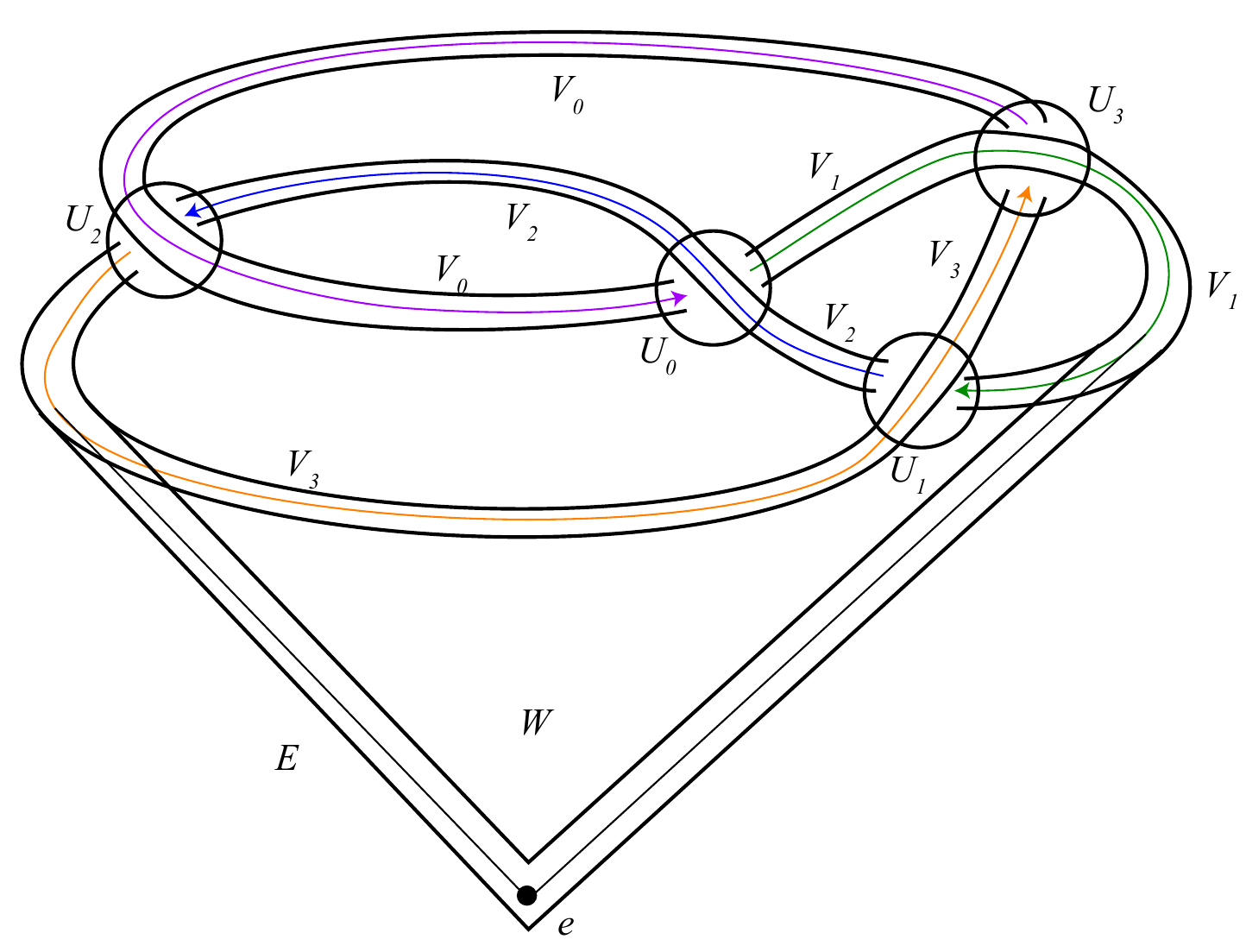}
	\caption{Open cover of the 2-skeleton determined by $D_K$.}
	\label{skeletonnbd.fig}
\end{figure}

Since $M_\rho$ is a 5-fold dihedral cover of $(S^3,K)$, $K$ has one index-1 preimage $K_0$ and two index-2 preimages $K_1$ and $K_2$. Denote the three corresponding lifts of each 1-cell $k_i$ by $k_i^0$, $k_i^1$,  and $k_i^2$.  

There are five lifts of the 3-cell $E$, denoted $E^0\dots E^4$.  We label them as follows.  Let $a$ and $b$ be two points in $E$, and $g:[0,1]\rightarrow S^3$ a path from $a$ to $b$ which intersects exactly one of the two-cells $A_i$ once, transversely, where $A_i$ is colored $n$.  Denote by $\tilde{g}_m$ the lift of $g$ such that $\tilde{g}_m(0)\in E^n$.  Then $\tilde{g}_m(1)\in E^{\text{Ref}_n(m)}$. This labelling is well-defined since the monodromy homomorphism $\rho$ is trivial on the complement $E$ of the cone. 

The neighborhood $W$ of the $0$-cell $e$ has 5 lifts to $M$.  The pair $(W,C^2)$ is homeomorphic to the cone on $(S^2,G_K)$ where $G_K$ is the 4-valent graph corresponding to the knot diagram $D_K$.  Its five lifts intersect the 3-cells $E_i$ as shown in Figure ~\ref{dehncoloring.fig}.  The subscripts on the $E_i$ are labeled by the 5 Dehn colorings of $K$ corresponding to the Fox coloring $\rho$, as defined in \cite{carter2014three}.

\begin{figure}\includegraphics[width=6in]{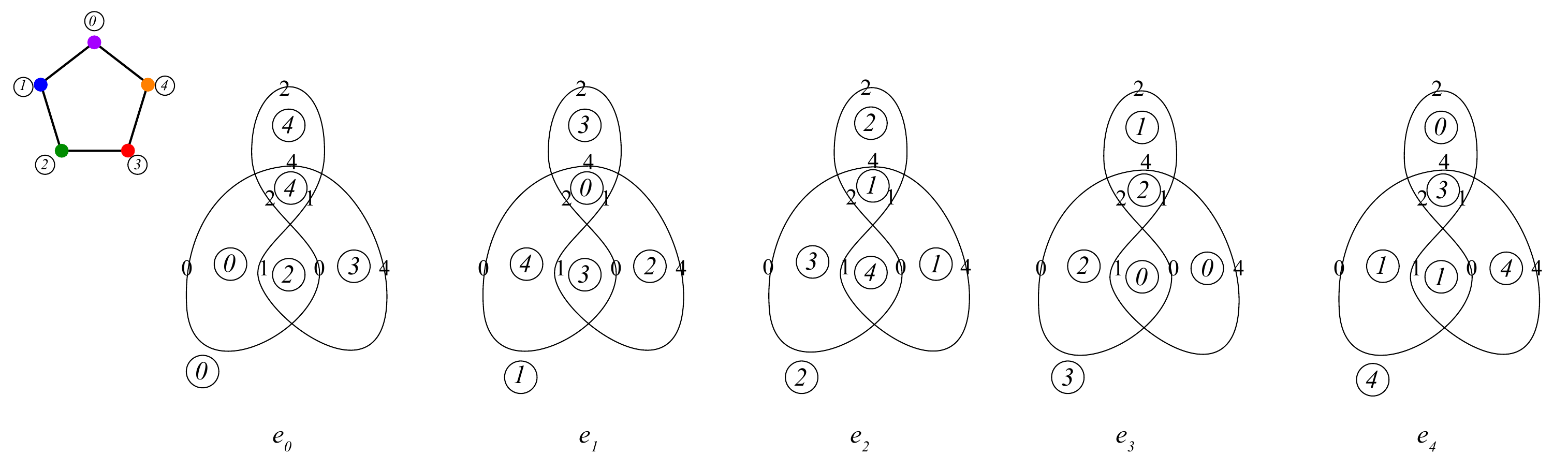}
	\caption{Position of the lifts $E^j$ of $E$ in a neighborhood of each lift $e_k$ of $e$. The position of $E^j$ is marked by \circled{j}.}
	\label{dehncoloring.fig}
\end{figure}

The other cells $a_i$ and $A_i$ above each have 5 lifts to $M_\rho$. We introduce a systematic way of labelling these cells.

Since $D_K$ has an even number of crossings, the vector field $-\partial/\partial z$ along $K$ has two lifts along each of the index-2 curves $K_1$ and $K_2$.  We choose one such lift $V^1$ along $K^1$ and one lift $V^2$ along $K^2$.  Now denote by $R^1_i$ and $L^1_i$ the two lifts of $A_i$ adjacent to $k_i^1$, with $V^1$ tangent to $R^1_i$.  Similarly denote by $R^2_i$ and $L^2_i$ the two lifts of $A_i$ adjacent to $k_i^2$, with $V^2$ tangent to $R^2_i$.  See Figure ~\ref{lifts5fold.fig}, which also serves to label the remaining lifts of the cells $A_i$ and $a_i$.  In particular, the lift of $A_i$ adjacent to $k_i^0$ is denoted $B_i$.  The five lifts of $a_i$ are $b_i$, $r_i^1$, $l_i^1$, $r_i^2$, and $l_i^2$, as shown in Figure  ~\ref{lifts5fold.fig}.
\begin{figure}
	\includegraphics[width=6.5in]{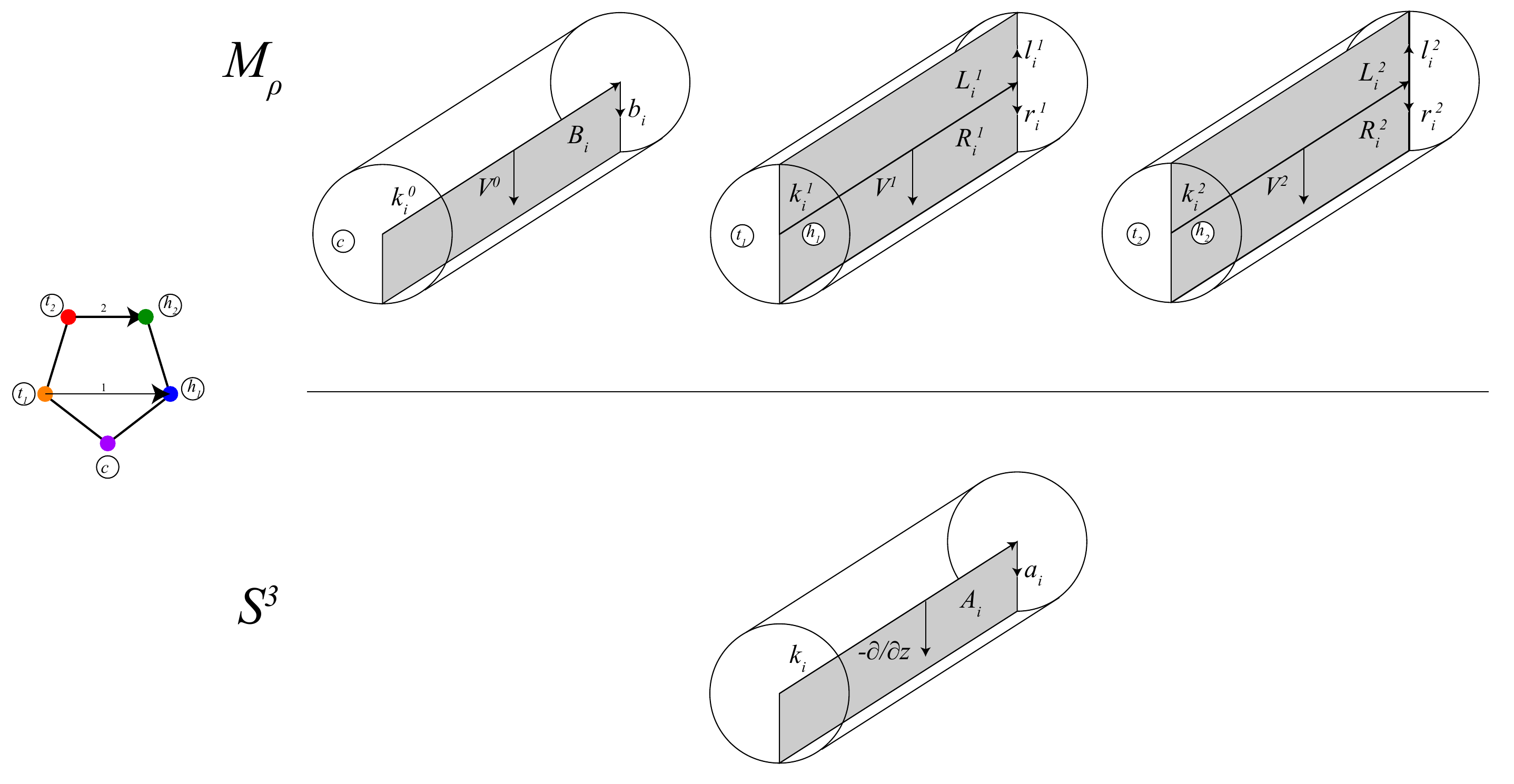}
	\caption{Lifts of a tubular neighborhood of the arc $k_i$ to the 5-fold dihedral cover.}\label{lifts5fold.fig}
	\label{lifts5fold.fig}
\end{figure}

To fill in Figure ~\ref{fig8crossinglifts.fig}, which depicts the lift of the cell structure in $U_i$ with all cells labeled, we carry out the following steps:

\begin{enumerate}
	\item Lift the 2-skeleton $C^2\cap U_i$ in a neighborhood of each crossing, and label the positions of the 3-cells $E_j$ according to the coloring.  
	\item At crossing 0, arbitrarily label one index-2 lift of $k_0$ $k_0^1$,  and label the other $k_0^2$.
	\item At crossing 0, arbitrarily label one of the 2-cells adjacent to $k_0^1$ $R_0^1$, and label the other $L_0^1$.
	\item At crossing 0, arbitrarily label one of the 2-cells adjacent to $k_0^2$ $R_0^2$, and label the other $L_0^2$.
	
\end{enumerate}
\begin{figure}
	\includegraphics[width=5.5in]{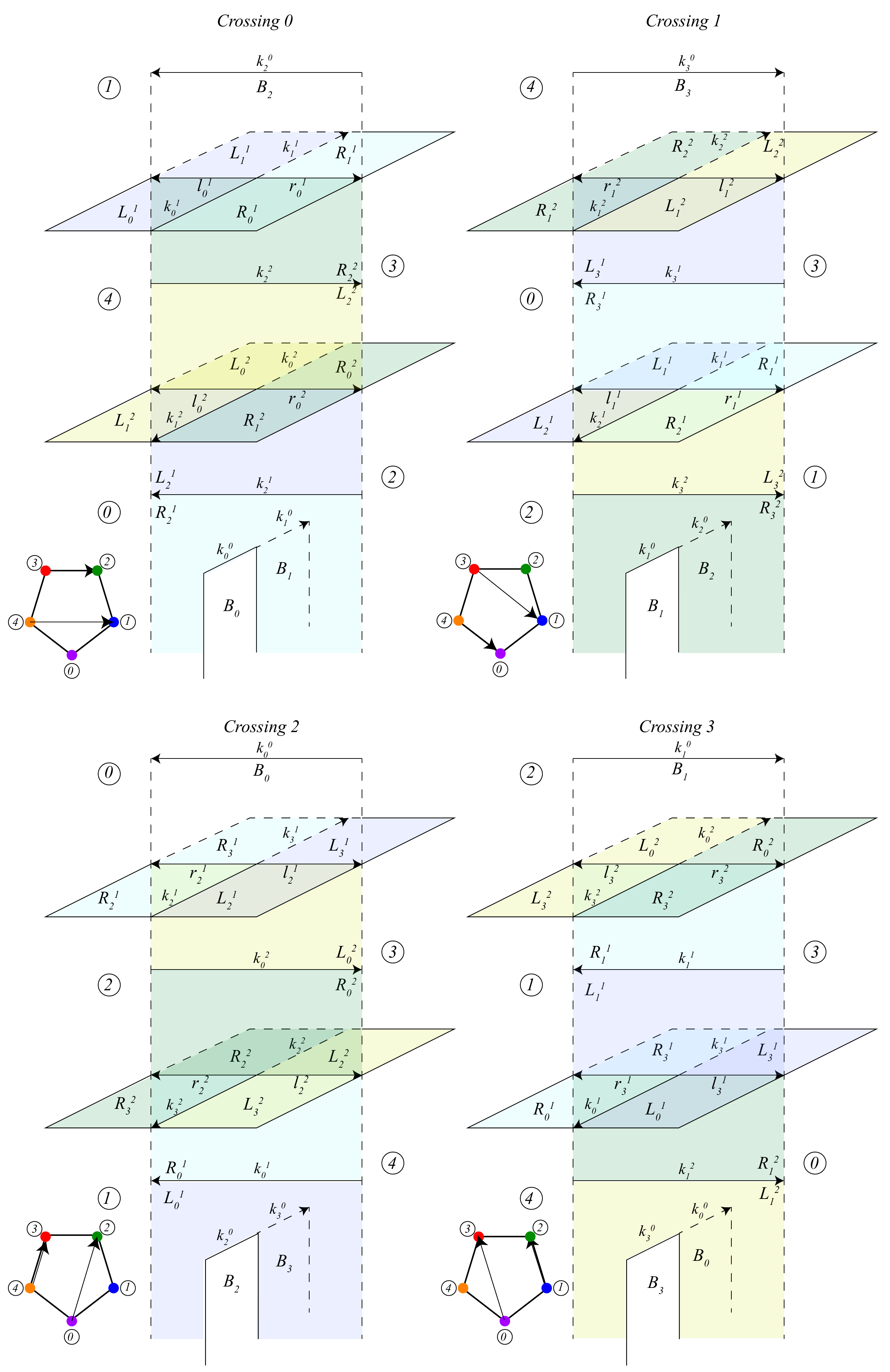}
	\caption{Cell structure on the 5-fold branched cover of $S^3$ along the Figure-8 knot.}\label{fig8crossinglifts.fig}
\end{figure}

Next we will find a two-chain $\Sigma^j$ such that $\partial K_j=\Sigma^j$ for $j=0,1,2$.  A priori, these two-chains take the following forms:

$$\Sigma^0= \sum_{i=0}^3 B_i + \sum_{i=0}^3 x_i^1(R_i^1-L_i^1)+ \sum_{i=0}^3 x_i^2(R_i^2-L_i^2)$$

$$\Sigma^1= \sum_{i=0}^3 y_i^1R_i^1+(1-y_i^1)L_i^1+\sum_{i=0}^3 y_i^2(R_i^2-L_i^2)$$

$$\Sigma^2=\sum_{i=0}^3 z_i^1(R_i^1-L_i^1)+\sum_{i=0}^3 z_i^2R_i^2+(1-z_i^2)L_i^2$$

Indeed, each $\Sigma^j$ is a linear combination of the 2-cells $B_i$, $R^k_i$ and $L^k_i$, and each $k_i^j$ must appear exactly once as a summand in $\partial \Sigma^j$.   The linear combinations take the form above if and only if:
$$\partial \Sigma^0-K^0=\sum_{i=0}^3\sum_{k=1}^2 \alpha_i^{k} (r_i^k-l_i^k),$$
$$\partial \Sigma^1-K^1=\sum_{i=0}^3\sum_{k=1}^2 \beta_i^{k} (r_i^k-l_i^k), \text{ and }$$
$$\partial \Sigma^2-K^2=\sum_{i=0}^3\sum_{k=1}^2 \gamma_i^{k} (r_i^k-l_i^k).$$
Since $\partial \Sigma^j = K^j$, we determine each $\alpha_i^{k}$,  $\beta_i^{k}$, and $\gamma_i^{k}$ and set them equal to 0. This results in three linear systems in the variables $x_i^j$, $y_i^j$, and $z_i^j$.  If the linear system corresponding to $\Sigma^j$ has a solution, it determines $\Sigma^j$ explicitly; if not, $K^j$ cannot be rationally nullhomologous.

 To compute  $\alpha_i^{k}$,  $\beta_i^{k}$, and $\gamma_i^{k}$, it suffices to look at the lift of the cell structure in $f^{-1}(U_i)$,  where $U_i$ is the neighborhood of crossing $i$ of $D_K$ described above.

To compute  $\alpha_0^{1}$,  $\beta_0^{1}$, and $\gamma_0^{1}$, we compute $\partial \Sigma^j$, but only record the boundaries of the 2-cells incident to $r_0^1$ or $l_0^1$, and only list the 1-cells which intersect $f^{-1}(U_0)$.  Referring to Figure ~\ref{fig8crossinglifts.fig} at Crossing 0, this gives :
\begin{center}
\begin{tabular}{ccc}
    $\partial\Sigma^0$ & $\partial\Sigma^1$ & $\partial\Sigma^2$ \\
   $\begin{cases}
\partial{B}_2= k^0_2-l^1_0+r^1_0+\dots\\
x^1_0 \partial R^1_0= x^1_0(k^1_0+r^1_0)+\dots \\
-x^1_0\partial L^1_0= -x^1_0(k^1_0+l^1_0)+\dots \\
x^1_1\partial R^1_1= x^1_1(k^1_1-r^1_0)+\dots \\
-x^1_1\partial L^1_1= -x^1_1(k^1_1-l^1_0)+\dots\\
x^2_2\partial R^2_2= x^2_2(k^2_2-r^1_0+l^1_0)+\dots\\
\end{cases}$&
$\begin{cases}

y^1_0\partial{R^1_0}= y^1_0(k^1_0+r^1_0)\dots \\
(1-y^1_0)\partial{L^1_0}= (1-y^1_0)(k^1_0+l^1_0)\dots \\
y^1_1\partial{R^1_1}=y^1_1(k^1_1-r^1_0)\dots \\
(1-y^1_1)\partial{L^1_1}= (1-y^1_1)(k^1_1-l^1_0)\dots \\
y^2_2\partial{R^2_2}= y^2_2(k^2_2-r^1_0+l^1_0)\dots
\end{cases} $ &
$\begin{cases}
z^1_0\partial{R^1_0}=z^1_0(k^1_0+r^1_0)\dots  \\
-z^1_0\partial{L^1_0}=-z^1_0(k^1_0+l^1_0) \dots \\
z^1_1\partial{R^1_1}=z^1_1(k^1_1-r^1_0) \dots \\
-z^1_1\partial{L^1_1}=-z^1_1(k^1_1-l^1_0) \dots \\
z_2^2\partial{R^2_2}=z_2^2(k^2_2-r^1_0+l^1_0) 
\end{cases}$
\end{tabular}
\end{center}

Therefore $$\alpha_0^{1}=x^1_0-x_1^1-x^2_2+1=0$$
$$\beta_0^{1}=y_0^1-y_1^1-y^2_2=0$$
$$\gamma_0^{1}=z_0^1-z_1^1-z^2_2=0$$

Similarly, recording the boundaries of 2-cells incident to $r_0^2$ or $l_0^2$ at Crossing 0 gives 
$$\alpha_0^2=x_0^2-x_1^2+x_2^2-x_2^1=0$$
$$\beta_0^2=y_0^2-y_1^2+y_2^2+1-y_2^1=0$$
$$\gamma_0^2=z_0^2-z_1^2-1+z_2^2-z_1^2=0$$

We can follow this process for each of the remaining crossings, generating three systems of eight inhomogeneous linear equations in eight variables, one system per 2-chain. We can then solve these systems of equations to determine values of each $x^j_i$, $y_i^j$, and $z_i^j$ in the 2-chains. These values yield a 2-chain whose boundary is the relevant lift of the knot.

\begin{center}
	\begin{tabular}{ccc}$\Sigma^0$ & $\Sigma^1$ & $\Sigma^2$\\
		$\begin{cases}
		\alpha_0^{1}=x^1_0-x_1^1-x^2_2+1=0\\
		\alpha_0^2=x_0^2-x_1^2+x_2^2-x_2^1=0\\
		\alpha_1^1=x_1^1-x_2^1+x_3^1+x_3^2=0\\
		\alpha_1^2=x_1^2-x_2^2+x_3^1+1=0\\
		\alpha_2^1=x_2^1-x_3^1-x_0^2-1=0\\
		\alpha_2^2=x_2^2-x_3^2+x_0^2-x_0^1=0\\
		\alpha_3^1=x_3^1-x_0^1+x_1^1+x_1^2=0\\
		\alpha_3^2=x_3^2-x_0^2+x_1^1-1=0
		\end{cases}$&
		$\begin{cases}\beta_0^{1}=y_0^1-y_1^1-y^2_2=0\\
		\beta_0^2=y_0^2-y_1^2+y_2^2+1-y_2^1=0\\
		\beta_1^1=y_1^1-y_2^1+y_3^1+y_3^2=0\\ 
		\beta_1^2=y_1^2-y_2^2-1+y_3^1=0\\
		\beta_2^1=y_2^1-y_3^1-y_0^2=0\\
		\beta_2^2=y_2^2-y_3^2+y_0^2-y_0^1=0\\
		\beta_3^1=y_3^1-y_0^1-1+y_1^1+y_1^2=0\\
		\beta_3^2=y_3^2-y_0^2+y_1^1=0
		\end{cases}$&
		$\begin{cases}
		\gamma_0^1=z_0^1-z_1^1-z^2_2=0\\
		\gamma_0^2=z_0^2-z_1^2-1+z_2^2-z_2^1=0\\
		\gamma_1^1=z_1^1-z_2^1+z_3^1-1+z_3^2=0\\
		\gamma_1^2=z_1^2-z_2^2+z_3^1=0\\
		\gamma_2^1=z_2^1-z_3^1+1-z_0^2=0\\
		\gamma_2^2=z_2^2-z_3^2+z_0^2-z_0^1=0\\
		\gamma_3^1=z_3^1-z_0^1+z_1^1+z_1^2=0\\
		\gamma_3^2=z_3^2-z_0^2+z_1^1=0
		\end{cases}$
	\end{tabular}
\end{center}
A choice of solution for the first system is $x_0^1=0$, $x_1^1=1$, $x_2^1=0$, and $x_3^1=-1$, and all $x_i^2=0$. The corresponding 2-chain is 
$$\Sigma^0=\sum_{i=0}^3 B_i+R_1^1-L_1^1-R_3^1+L_3^1.$$
We now compute the intersection numbers of $K^1$ and $K^2$ with $\Sigma^0$. We first isotop each $K^i$ in the complement of $K^0$ so that it meets $\Sigma^0$ transversely.  To make a consistant choice of push-off, let $h=h_i^j$ be the superscript of the 3-cell $E^h$, such that informally, if one stands on the arc $k_i^j$ facing in the direction of the orientation of $K^j$, the 2-cell $R_i^j$ is on the right; see Figure ~\ref{lifts5fold.fig}. Then push $K^j$ into the 3-cells $E^{h_i^j}$, and denote this push-off by $(K^j)'$.  

Now $$(K^1)'\cdot \Sigma^0=(K^1)'\cdot \sum_{i=0}^3 B_i+ (K^1)'\cdot (R_1^1-L_1^1-R_3^1+L_3^1)=(0+0-1+0)+(1-1-1+0)=-2$$
and
$$(K^2)'\cdot \Sigma^0=(K^2)'\cdot \sum_{i=0}^3 B_i+ (K^2)'\cdot (R_1^1-L_1^1-R_3^1+L_3^1)=(0+1+0+0)+(0
+0+0+1)=2$$
Similarly, a solution to the second system is $y_2^1=1$,  $y_3^1=1$, and all other $y_i^j=0$.  This corresponds to the 2-chain
$$\Sigma_1=L_0^1+L_1^1+R_2^1+R_3^1.$$
This gives $(K^0)'\cdot \Sigma^1=-2$ and $(K^2)'\cdot\Sigma^1=0$.

Finally, a solution to the third system is $z_2^1=-1$, and all other $z_i^j=0$. We take
$$\Sigma_2=-R_2^1+L_2^1+L_0^2+L_1^2+L_2^2+L_3^2$$

This gives $(K^0)'\cdot \Sigma^2=2$ and $(K^1)'\cdot \Sigma^2=0$.

The dihedral linking invariant of the Figure-8 is therefore

$$\text{DLN}(K,\rho)=\{\text{lk}(K^0,K^1),\text{lk}(K^0,K^2), \text{lk}(K^1,K^2)\}=\{-2,2,0\}.$$

Note that the algorithm double-checks itself, by computing both $\text{lk}(K^i,K^j)$ and $\text{lk}(K^j,K^i)$ and confirming that they agree.

\section{General Setup}\label{setup.sec}

Let $K$ be a knot and $\rho:\pi_1(S^3-K)\twoheadrightarrow D_p$ a  $p$-coloring of $K$. Let $M_\rho$ denote the corresponding $p$-fold irregular branched cover of $(S^3,K)$.  Let $D_K$ be a diagram of $K$ with an even number of crossings $n$; we can always arrange this by adding a crossing with a first Reidemeister move.

As in Section ~\ref{fig8example.sec}, we consider a cell structure on $S^3$ determined by the cone on $D_K$.

	\begin{enumerate}
		\item ``horizontal" 1-cells $k_0$, $k_1$, $\dots$, $k_{n-1}$,  which are arcs in the diagram $D_K$ of $K$;
		\item ``vertical'' 1-cells $a_0$, $a_1$, $\dots$, $a_{n-1}$, below each crossing of $D_K$;
		\item ``vertical'' 2-cells $A_0$, $A_1$, $\dots$, $A_{n-1}$, below each arc of $D_K$
		\item one 3-cell, $E$, the complement of the cone
		\item one 0-cell $e$ at the cone point
		\item  one 0-cell $e_i$ at the head of each arc $k_i$
	\end{enumerate}
Denote the $k$-skeleton by $C^k$.

Next we lift this cell structure to the dihedral cover $M_\rho$ of $S^3$, branched along $K$ determined by the $p$-coloring $\rho$ of $K$. We let $U_i$ and $V_i$ represent neighborhoods of crossing $i$ and arc $k_i$ as before; see Figure ~\ref{skeletonnbd.fig}.

Since $M_\rho$ is a $p$-fold dihedral cover of $(S^3,K)$, $K$ has one index-1 preimage $K^0$, and $q=\dfrac{p-1}{2}$ index-2 preimages $K^1$, $K^2$, $\dots$, $K^q$. Denote the corresponding lifts of each 1-cell $k_i$ by $k_i^0$, $k_i^1$, $\dots$, $k_i^q$.

  There are $p$ lifts of the 3-cell $E$, denoted $E^0\dots E^p$.  We again label them by the following rule: Let $a$ and $b$ be two points in $E$, and $g:[0,1]\rightarrow S^3$ a path from $a$ to $b$ which intersects exactly one of the two-cells $A_i$ once, transversely, where $k_i$ is colored $n$.  Denote by $\tilde{g}_m$ the lift of $g$ such that $\tilde{g}_m(0)\in E^m$.  Then $\tilde{g}_m(1)\in E^{\text{Ref}_n(m)}$.

The other cells $a_i$ and $A_i$ above each have $p$ lifts to $M_\rho$.  Let $B_i$ denote the lift of $A_i$ incident to $k_i^0$.  There are two lifts of $A_i$ incident to each $k_i^j$ for $1\leq j \leq q$, which we denote by $R_i^j$ and $L_i^j$, labeled as in Section ~\ref{fig8example.sec}. First, arbitrarily label the two lifts of $A_0$ incident to $k_0^j$ $R_0^j$ and $L_0^j$.  Then lift the vertical vector field $-\partial/\partial z$ along $K$ (tangent to $A_i$) to $M_\rho$. Because the number of crossings of $D_K$ is even, $-\partial/\partial z$ has two lifts to $M_\rho$.  Let $V^j$ denote the lift of $V$ tangent to $R_0^j$.  Then for $i\geq 1$, let $R_i^j$ be the lift of $A_i$ incident to $k_i^j$ and tangent to $V^j$, and let $L_i^j$ be the other lift of $A_i$ incident to $k_i^j$.   Finally, let $r_i^j$ denote the lift of $a_i$ on the boundary of $R_i^j$, let $l_i^j$ denote the lift of $a_i$ on the boundary of $L_i^j$, and let $b_i$ be the lift of $a_i$ on the boundary of $B_i$.  See Figure ~\ref{lifts5fold.fig} for an example when $p=5$.

It will be useful to record the positions of the 2-cells $R_i^j$ and $L_i^j$ relative to the 3-cells $E^0$, $\dots$, $E^{p-1}$.  The 2-cells $R_i^j$ and $L_i^j$ are incident to two 3-cells.  Informally, let $h=h_i^j$ denote the superscript of the 3-cell $E^h$ such that, if one stands on the arc $k_i^j$ facing in the direction of its orientation, $R_i^j$ is on the right and $L_i^j$ is on the left.  Let $t=t_i^j$ denote the superscript of the 3-cell $E^t$ such that, if one stands on the arc $k_i^j$ facing in the direction of its orientation, $L_i^j$ is on the right and $R_i^j$ is on the left.

\begin{defn} The {\it configuration diagram} for arc $i$ is a set of $q=\dfrac{p-1}{2}$ arrows $\text{arr}_i^1$, $\dots$, $\text{arr}_i^q$ between vertices $\{0,\dots,p-1\}$ of a regular $p$-gon (necessarily with disjoint endpoints), such that the head and tail of $\text{arr}_i^j$ are $h_i^j$ and $t_i^j$ $\in \{0,\dots,p-1\}$, respectively.  Note that if $c(i)$ is the color of the arc $k_i$, none of the arrows $\text{arr}_i^1$, $\dots$, $\text{arr}_i^q$ have an endpoint on vertex $c(i)$.  We let $\text{arr}_i^0$ be an arrow with head and tail $h_i^j=t_i^j=c(i)$, but omit this arrow from the diagram. See Figure ~\ref{configdiagram.fig}.
	\end{defn}
	
	 Observe that the configuration diagram for arc $i+1$ is obtained from the configuration diagram for arc $i$ by a reflection over the vertex $c(o(i))$, where $s=o(i)$ is the subscript of the arc $k_s$ passing over $k_i$ at its head.  The configuration diagrams for each crossing of the Figure-8 knot are shown in Figure ~\ref{fig8crossinglifts.fig}, together with the lifts of the relevant 2-cells.
 \begin{figure}[htbp]  
\includegraphics[width=5in]{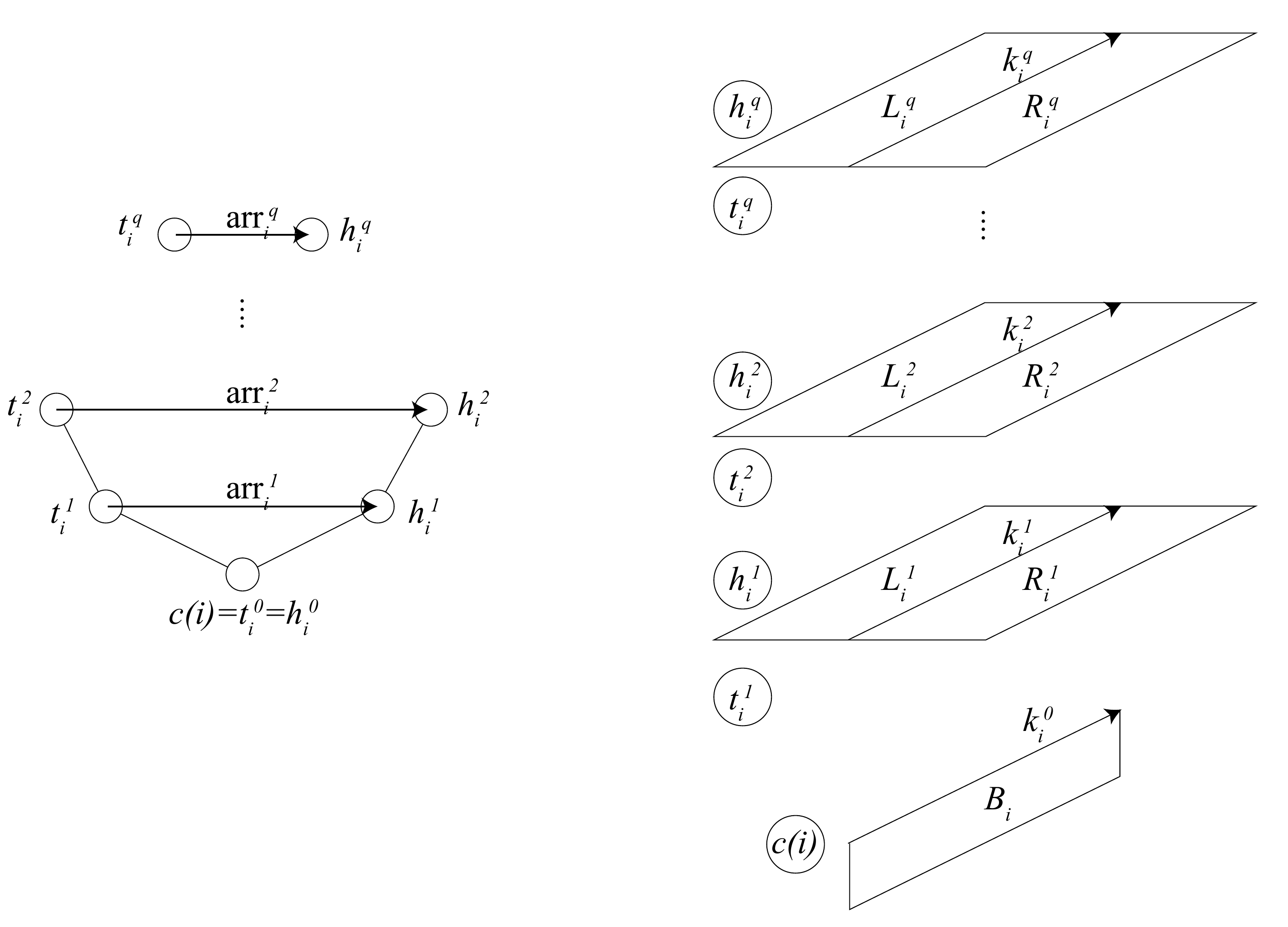}
\caption{A possible configuration diagram for the $i^{th}$ arc $k_i$ of $D_K$.}
\label{configdiagram.fig}
\end{figure}

\section{Constructing 2-chains bounded by the singular set}\label{2chain.sec}

In this section, we construct a rational 2-chain $\Sigma^j$ bounding each $K^j$, for $j\in\{0,\dots,q\}$ or determine no such 2-chain exists.  A priori, these 2-chains take the following forms:

$$\Sigma^0= \sum_{i=0}^{n-1} B_i + \sum_{j=1}^q \sum_{i=0}^{n-1} x_i^{0,j}(R_i^j-L_i^j)$$

$$\Sigma^k=  \sum_{i=0}^{n-1}x_i^{k,k}R_i^k+(1-x_i^{k,k})L_i^k+ \sum_{j\in\{0,\dots,q\}\setminus\{k\} }\quad \sum_{i=0}^{n-1}x_i^{k,j}(R_i^j-L_i^j)$$

Now for all $k\in\{0,\dots,q\}$,

$$\partial\Sigma^k-K^k=\sum_{j=1}^q \sum_{i=0}^{n-1} \alpha_i^{k,j} (r_i^j-l_i^j).$$

We compute each $\alpha_i^{k.j}$ using the configuration diagram at crossing $i$ and set the result equal to 0, giving us a system of equations for which a solution determines the 2-chain $\Sigma^k$.  

Figure ~\ref{rijlijcases.fig} shows the 8 possible configurations of 2-cells incident to $r_i^j-l_i^j$ when the local writhe number of the crossing is positive. We rotate the picture so that $R_i^j$ is to the right of $k_i^j$ and $L_i^j$ is to the left of $k_i^j$. Informally, let $a(i,j)$ and $b(i,j)\in\{0,1,\dots, q\}$ be the index-2 lifts of the over arc $k_{o(i)}$ which sit above and below $r_i^j-l_i^j$ in the picture, respectively.  Note that if $a(i,j)=0$, $B_{o(i)}$ is incident to $r_i^j-l_i^j$ from above, and if $b(i,j)=0$, $B_{o(i)}$ is incident to $r_i^j-l_i^j$ from below. A formal definition, which also works for the case $j=0$, is as follows.

\begin{defn} Let $i\in\{0,\dots,n-1\}$ and let $j\in\{0,\dots,q\}$ where  $q=\dfrac{p-1}{2}$.  Then set
	
	\[ a(i,j)=
	s \text{ if } h_i^j=h_{o(i)}^s \text{ or } h_i^j=t_{o(i)}^s,	\]
	and set
	\[ b(i,j)=
	s \text{ if } t_i^j=h_{o(i)}^s \text{ or } t_i^j=t_{o(i)}^s.	\]
\label{abdefinition.def}	
\end{defn}

\begin{ex}We can read off the values of $a(i,j)$ and $b(i,j)$ from the configuration diagrams at crossings $i$ and $o(i)$. First observe from Figure ~\ref{cone.fig} that
	
	$$o(0)=2,o(1)=3, o(2)=0, \text{ and } o(3)=1.$$
	
	Next, observe $a(i,j)$ is the arrow in configuration diagram $o(i)$ such that its head or tail is incident to $h_i^j$; that is, either $h_i^j=h_{o(i)}^{a(i,j)}$ or $h_i^j=t_{o(i)}^{a(i,j)}$.  Similarly  $b(i,j)$ is the arrow in configuration diagram $o(i)$ such that its head or tail is incident to $t_i^j$; that is, either $t_i^j=h_{o(i)}^{a(i,j)}$ or $t_i^j=t_{o(i)}^{a(i,j)}$.
	
	Since $o(0)=2$, we can read off each $a(0,j)$ and $b(0,j)$ from the first and third diagrams in Figure \ref{configfig8.fig}.  For example, $h_0^1=1=h_2^0=t_2^0$, and $t_0^1= 4=t_2^2$, so $a(0,1)=0$ and $b(0,1)=2$.  Similarly, $h_0^2= 2=h_2^1$ and $t_0^2=3=h_2^2$, so $a(0,2)=1$ and $b(0,2)=2$. We can read off each $a(1,j)$ and $b(1,j)$ from the first and third diagrams, and so on.
	
	The complete set of values of $a(i,j)$ and $b(i,j)$ are below.
	\begin{center}
	\begin{tabular}{|c|c|c|c|}
	\hline
	$a(i,j)$&$j=0$&$j=1$&$j=2$\\ \hline
	$i=0$&$1$&$0$&$1$ \\ \hline
	$i=1$&$2$&$2$&$1$ \\ \hline
	$i=2$&$1$&$2$&$2$ \\ \hline
	$i=3$&$2$&$1$&$0$ \\ \hline
	\end{tabular} \hspace{1in}	\begin{tabular}{|c|c|c|c|}
	\hline
	$b(i,j)$&$j=0$&$j=1$&$j=2$\\ \hline
	$i=0$&$1$&$2$&$2$ \\ \hline
	$i=1$&$2$&$1$&$0$ \\ \hline
	$i=2$&$1$&$0$&$1$ \\ \hline
	$i=3$&$2$&$2$&$1$ \\ \hline
	\end{tabular}
	\end{center}
\label{abfunctionsexample.ex}	
\end{ex}
\begin{figure}
	\includegraphics[width=\textwidth]{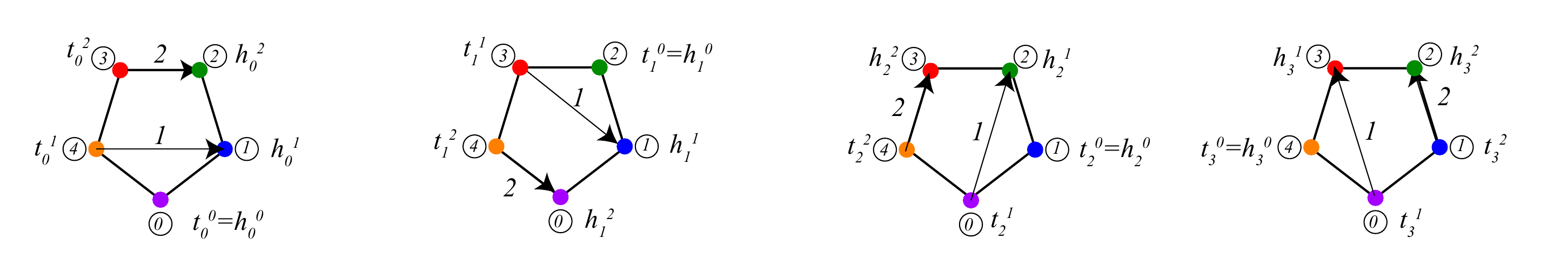}
	\caption{The configuration diagrams for the Figure-8 knot.}
	\label{configfig8.fig}
\end{figure}
We record these configurations combinatorially using two $\{0,1,-1\}$-valued functions, $\epsilon_a$ and $\epsilon_b$. Together with the local writhe number $\epsilon(i)$ at crossing $i$, $\epsilon_a$ and $\epsilon_b$ determine the coefficients of the variables $x_{o(i)}^{a(i)}$ and $x_{o(i)}^{b(i)}$ in $\alpha_i^{k,j}$.

\begin{defn} Let $i\in \{0,\dots,n-1\}$, $j\in\{0,\dots,q\}$ with $a(i,j)$, $b(i,j)$, and $o(i)$ as defined above.  Then set
	\[ \epsilon_a(i,j)=\begin{cases} 
	      1 & h_i^j= h_{o(i)}^{a(i,j)} \text{ and } a(i,j)\neq 0\\
	       -1 & h_i^j= t_{o(i)}^{a(i,j)} \text{ and } a(i,j)\neq 0\\
		   0 & a(i,j)=0
	   \end{cases}
	\]
	
	and
	\[ \epsilon_b(i,j)=\begin{cases} 
	1 & t_i^j= t_{o(i)}^{b(i,j)} \text{ and } b(i,j)\neq 0\\
	       -1 & t_i^j= h_{o(i)}^{b(i,j)} \text{ and } b(i,j)\neq 0\\
	       
		   0 & b(i,j)=0
	   \end{cases}
	\]
	Note that if $j=0$, $a(i,j)=b(i,j)$, so $\epsilon_a(i,j)=\epsilon_b(i,j)$.
	\label{abfunctions.def}
\end{defn}

\begin{ex} We list the $\epsilon_a(i,j)$ and $\epsilon_b(i,j)$ for  the Figure-8 in Section ~\ref{fig8example.sec} in the table below.
	\begin{center}
\begin{tabular}{|c|c|c|c|}
\hline
$\epsilon_a(i,j)$&$j=0$&$j=1$&$j=2$\\ \hline
$i=0$&$-1$&$0$&$1$ \\ \hline
$i=1$&$1$&$-1$&$-1$ \\ \hline
$i=2$&$1$&$1$&$-1$ \\ \hline
$i=3$&$-1$&$-1$&$0$ \\ \hline
\end{tabular}\hspace{1in}
\begin{tabular}{|c|c|c|c|}
\hline
$\epsilon_b(i,j)$&$j=0$&$j=1$&$j=2$\\ \hline
$i=0$&$1$&$1$&$-1$ \\ \hline
$i=1$&$-1$&$-1$&$0$ \\ \hline
$i=2$&$-1$&$0$&$1$ \\ \hline
$i=3$&$1$&$-1$&$-1$ \\ \hline
\end{tabular}
\end{center}
\end{ex}

By considering the 8 possible configurations of cells in Figures ~\ref{rijlijcases.fig} and ~\ref{rijlijcasesneg.fig}, we see that $\alpha_i^{k,j}=x_i^j-x_{i+1}^{j}-\epsilon_a(i,j)x_{o(i)}^{k,a(i,j)}-\epsilon_b(i,j)x_{o(i)}^{k,b(i,j)}+C(i,j,k)$, where $C(i,j,k)$ is a constant.  This constant $C(i,j,k)$ is the coefficient of $r_i^j-l_i^j$ in the boundary of $\sum_{i=0}^{n-1}B_i$ when $k=0$, and in the boundary of $\sum_{i=0}^{n-1} L_i^k$ when $k\in \{1,\dots,q\}$.

We write $C(i,j,k)=C_a(i,j,k)+C_b(i,j,k)$ with $C_a$ and $C_b$ defined as follows.  Note that $C_a$ is nonzero precisely when $L_i^k$ (when $k\neq 0$) or $B_i$ (when $k=0$) is incident to $r_i^j-l_i^j$ from above.  Similarly, $C_b$ is nonzero when $L_i^k$ (when $k\neq 0$) or $B_i$ (when $k=0$) is incident to $r_i^j-l_i^j$ from below.

\begin{defn} Let $i\in \{0,\dots,n-1\}$, $j\in\{1,\dots,q\}$, and $k\in \{0,\dots,q\}$.  Then set
	\[ C_a(i,j,k)=\begin{cases} 
	      -\epsilon(i) & a(i,j)=k, k\neq 0, \text{ and } \epsilon(i)\epsilon_a(i,j)=-1\\
		  -\epsilon(i) & a(i,j)=k \text{ and } k=0\\
 	       
		   0 &  \text{else}
	   \end{cases}
	\]
	\[ C_b(i,j,k)=\begin{cases} 
	      \epsilon(i) & b(i,j)=k, k\neq 0,  \text{ and } \epsilon(i)\epsilon_b(i,j)=1\\
		  \epsilon(i) & b(i,j)=k\text{ and } k=0\\
 	       
		   0 & \text{else}
	   \end{cases}
	\]
	\label{abconstants.def}
\end{defn}

\begin{ex}
	The values of $C_a$ and $C_b$ for the Figure-8 knot in Section ~\ref{fig8example.sec} are as follows:
	\begin{center}
	\begin{tabular}{|c|c|c|c|}
	\hline
	$C_a(i,j,0)$&$j=0$&$j=1$&$j=2$\\ \hline
	$i=0$&$0$&$1$&$0$ \\ \hline
	$i=1$&$0$&$0$&$0$ \\ \hline
	$i=2$&$0$&$0$&$0$ \\ \hline
	$i=3$&$0$&$0$&$-1$ \\ \hline
	\end{tabular}\hspace{1in}\begin{tabular}{|c|c|c|c|}
\hline
$C_a(i,j,1)$&$j=0$&$j=1$&$j=2$\\ \hline
$i=0$&$0$&$0$&$1$ \\ \hline
$i=1$&$0$&$0$&$-1$ \\ \hline
$i=2$&$1$&$0$&$0$ \\ \hline
$i=3$&$0$&$-1$&$0$ \\ \hline
\end{tabular}

\begin{tabular}{|c|c|c|c|}
\hline
$C_a(i,j,2)$&$j=0$&$j=1$&$j=2$\\ \hline
$i=0$&$0$&$0$&$0$ \\ \hline
$i=1$&$0$&$-1$&$0$ \\ \hline
$i=2$&$0$&$1$&$0$ \\ \hline
$i=3$&$-1$&$0$&$0$ \\ \hline
\end{tabular}

\begin{tabular}{|c|c|c|c|}
\hline
$C_b(i,j,0)$&$j=0$&$j=1$&$j=2$\\ \hline
$i=0$&$0$&$0$&$0$ \\ \hline
$i=1$&$0$&$0$&$1$ \\ \hline
$i=2$&$0$&$-1$&$0$ \\ \hline
$i=3$&$0$&$0$&$0$ \\ \hline
\end{tabular}\hspace{1in}
\begin{tabular}{|c|c|c|c|}
\hline
$C_b(i,j,1)$&$j=0$&$j=1$&$j=2$\\ \hline
$i=0$&$0$&$0$&$0$ \\ \hline
$i=1$&$0$&$0$&$0$ \\ \hline
$i=2$&$-1$&$0$&$0$ \\ \hline
$i=3$&$0$&$0$&$0$ \\ \hline
\end{tabular}

\begin{tabular}{|c|c|c|c|}
\hline
$C_b(i,j,2)$&$j=0$&$j=1$&$j=2$\\ \hline
$i=0$&$0$&$0$&$-1$ \\ \hline
$i=1$&$0$&$0$&$0$ \\ \hline
$i=2$&$0$&$0$&$0$ \\ \hline
$i=3$&$1$&$0$&$0$ \\ \hline
\end{tabular}
\end{center}

\end{ex}

\begin{ex} We can now compute $\alpha_0^{0,1}$ for the Figure-8 knot in Section ~\ref{fig8example.sec} (where it was denoted $\alpha_0^1$) directly from the combinatorial formula $\alpha_i^{k,j}=x_i^{k,j}-x_{i+1}^{k,j}-\epsilon_a(i,j)x_{o(i)}^{k,a(i,j)}-\epsilon_b(i,j)x_{o(i)}^{k,b(i,j)}+C(i,j,k)$.
	
Set $i=0$, $j=1$, and $k=0$.  Clearly $x_i^{k,j}-x_{i+1}^{k,j}=x_0^{0,1}-x_1^{0,1}$.  

Recall that $o(0)=2$.

We found $a(0,1)=0$ and $b(0,1)=2$ in Example ~\ref{abfunctionsexample.ex}.

Next, $\epsilon_a(0,1)=0$ since $a(0,1)=0$ by Definition ~\ref{abfunctions.def}. Also using Definition ~\ref{abfunctions.def}, and by considering Figure ~\ref{configfig8.fig}, $\epsilon_b(0,1)=1$ since $t_0^1=4=t_2^2$ .

Since crossing 0 has negative local writhe number, $\epsilon(i)=-1$.

Now $C_a(0,1,0)=-(-1)=1$ by Definition ~\ref{abconstants.def}, since $a(0,1)=0=k$.  Similarly, $C_b(0,1,0)=0$ since $0\neq b(0,1)$.

Therefore $\alpha_0^{0,1}=x_0^1-x_1^1-x_{2}^{0,2}+1,$ which agrees with the direct computation in Section~\ref{fig8example.sec}.

\end{ex}

\begin{thm}  Let $q=\dfrac{p-1}{2}$ with $p$ odd.  Let $\rho:\pi_1(S^3\rightarrow K)$ be a $p$-coloring of $K$, with corresponding irregular $p$-fold dihedral cover $f:M_\rho\rightarrow S^3$.  Let $K^0$ be the index-1 connected component of $f^{-1}(K)$ and let $K^1,K^2,\dots, K^q$ be the index-2 connected components of $f^{-1}(K)$.  Then $K^k$ is rationally null-homologous in $M_\rho$ if and only if the system of equations
	
	$$x_i^{k,j}-x_{i+1}^{k,j}-\epsilon_a(i,j)x_{o(i)}^{k,a(i,j)}-\epsilon_b(i,j)x_{o(i)}^{k,b(i,j)}+C_a(i,j,k)+C_b(i,j,k)=0$$
	
	has a solution $(x_0^{k,j}, \dots, x_{n-1}^{k,j})_{j=1}^q \in \mathbb{Q}^{nq}$.

	  Furthermore, if $k=0$, the index-1 lift $K^0$ of $K$ bounds the 2-chain
	$$\Sigma^0= \sum_{i=0}^{n-1} B_i + \sum_{j=1}^q \sum_{i=0}^{n-1} x_i^{0,j}(R_i^j-L_i^j),$$
	
	and if $k\neq 0$, 
    the index-2 lift $K^k$ of $K$ bounds the 2-chain 
   	$$\Sigma^k=  \sum_{i=0}^{n-1}x_i^{k,k}R_i^k+(1-x_i^{k,k})L_i^k+ \sum_{j\in\{0,\dots,q\}\setminus\{k\} }\quad \sum_{i=0}^{n-1}x_i^{k,j}(R_i^j-L_i^j).$$
	
	\label{chains.thm}
\end{thm}

\begin{proof} We systematically check each possible configuration of 2-cells incident to $r_i^j-l_i^j$, and compute the total number of times $r_i^j-l_i^j$ appears in the boundary of 
	$$\Sigma^k=\sum_{i=0}^{n-1}x_i^{k,k}R_i^k+(1-x_i^{k,k})L_i^k+ \sum_{j\in\{0,\dots,q\}\setminus\{k\} }\quad \sum_{i=0}^{n-1}x_i^{k,j}(R_i^j-L_i^j)$$ if $k\neq 0,$ 
	$$\Sigma^0= \sum_{i=0}^{n-1} B_i + \sum_{j=1}^q \sum_{i=0}^{n-1} x_i^{0,j}(R_i^j-L_i^j),$$ if $k=0.$
	
	There are 16 possible configurations.  The 8 configurations corresponding to a crossing with positive local writhe number are shown in Figure ~\ref{rijlijcases.fig}, and the 8 configurations corresponding to a crossing with negative local writhe number are shown in Figure ~\ref{rijlijcasesneg.fig}.  In each case, the coefficients of $x_i^{k,j}$, $x_{i+1}^{k,j}$, $x_{o(i)}^{k,a(i,j)}$, and $x_{o(i)}^{k,b(i,j)}$ in $\alpha_i^{k,j}$ are shown, as are the values of $\epsilon_a$ and $\epsilon_b$.  We observe that the coefficients of $x_i^{k,j}$ and $x_{i+1}^{k,j}$ coming from the boundaries of the ``horizontal" 2-cells $\partial (x_i^j(R_i^j-L_i^j)+x_{i+1}^j(R_{i+1}^j-L_{i+1}^j))$ are 1 and -1 respectively.  We then check that the coefficients of  $x_{o(i)}^{k,a(i,j)}$ and $x_{o(i)}^{k,b(i,j)}$ coming from the boundaries of the two ``vertical" 2-cells are $-\epsilon_a$ and $-\epsilon_b$, respectively.  There are either 0 or 1 additional 2-cells incident to $r_i^j-l_i^j$, depending on the value of $k$.  This contributes the constant term $C=C_a+C_b$ to $\alpha_i^{k,j}$, where $C=0$ or $\pm 1$.  If $k=0$, then $C=\pm 1$ if and only if $B_i$ is incident to $r_i^j-l_i^j$.  If $k\neq 0$, then $C=\pm 1$ if and only if $a(i,j)=k$ or $b(i,j)=k$.  The sign depends on the local writhe number of the crossing.  In each of the configurations in Figures ~\ref{rijlijcases.fig} and ~\ref{rijlijcasesneg.fig}, we verify that $C=C_a(i,j,k)+C_b(i,j,k)$ is the coefficient of $r_i^j-l_i^j$ in the relevant additional 2-cell, completing the proof.

\end{proof}

\begin{figure}[htbp]
\includegraphics[width=5in]{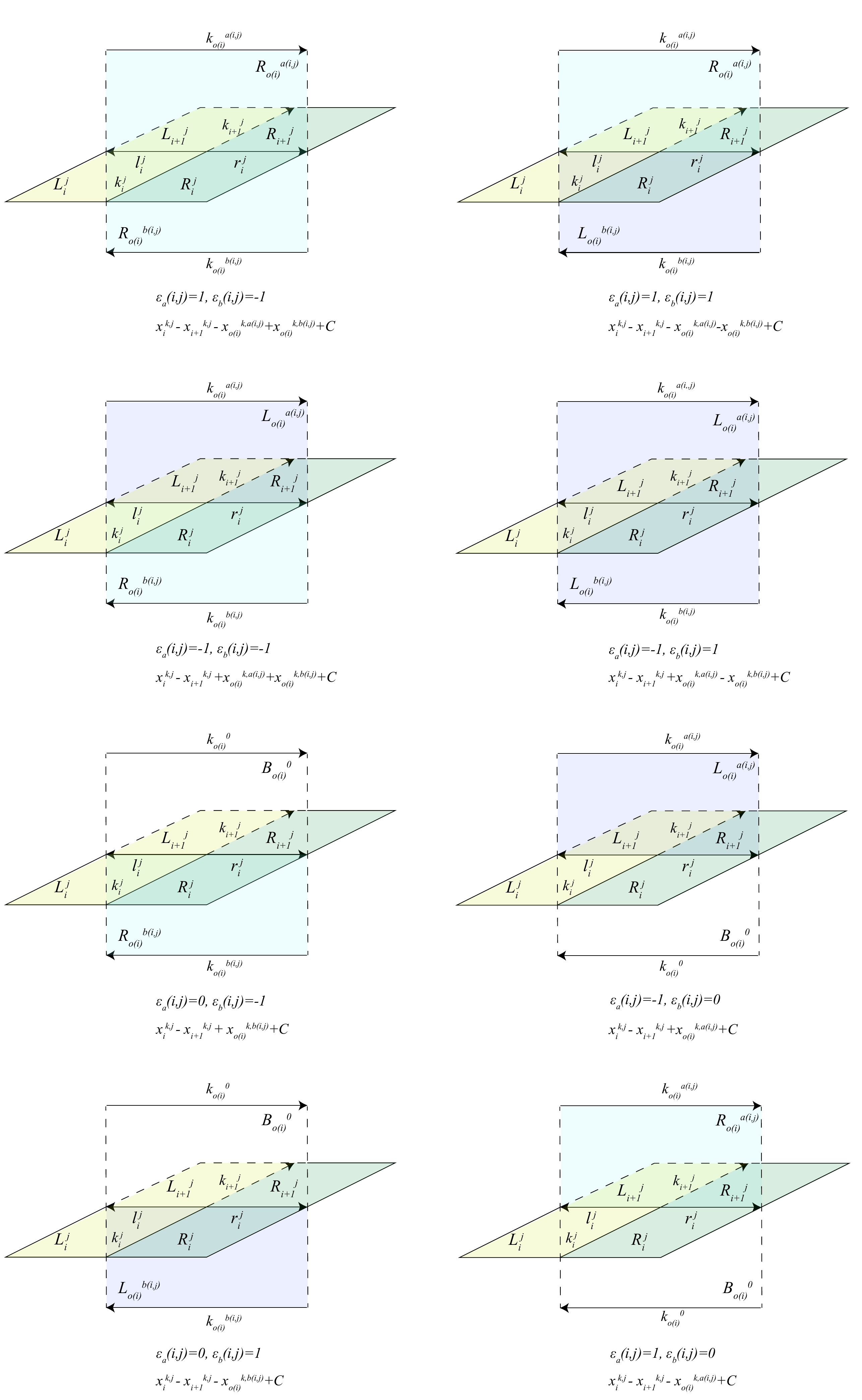}
\caption{Configurations of cells incident to $r_i^j-l_i^j$ at a positive crossing.   }
\label{rijlijcases.fig}
\end{figure}
\begin{figure}[htbp]
\includegraphics[width=5in]{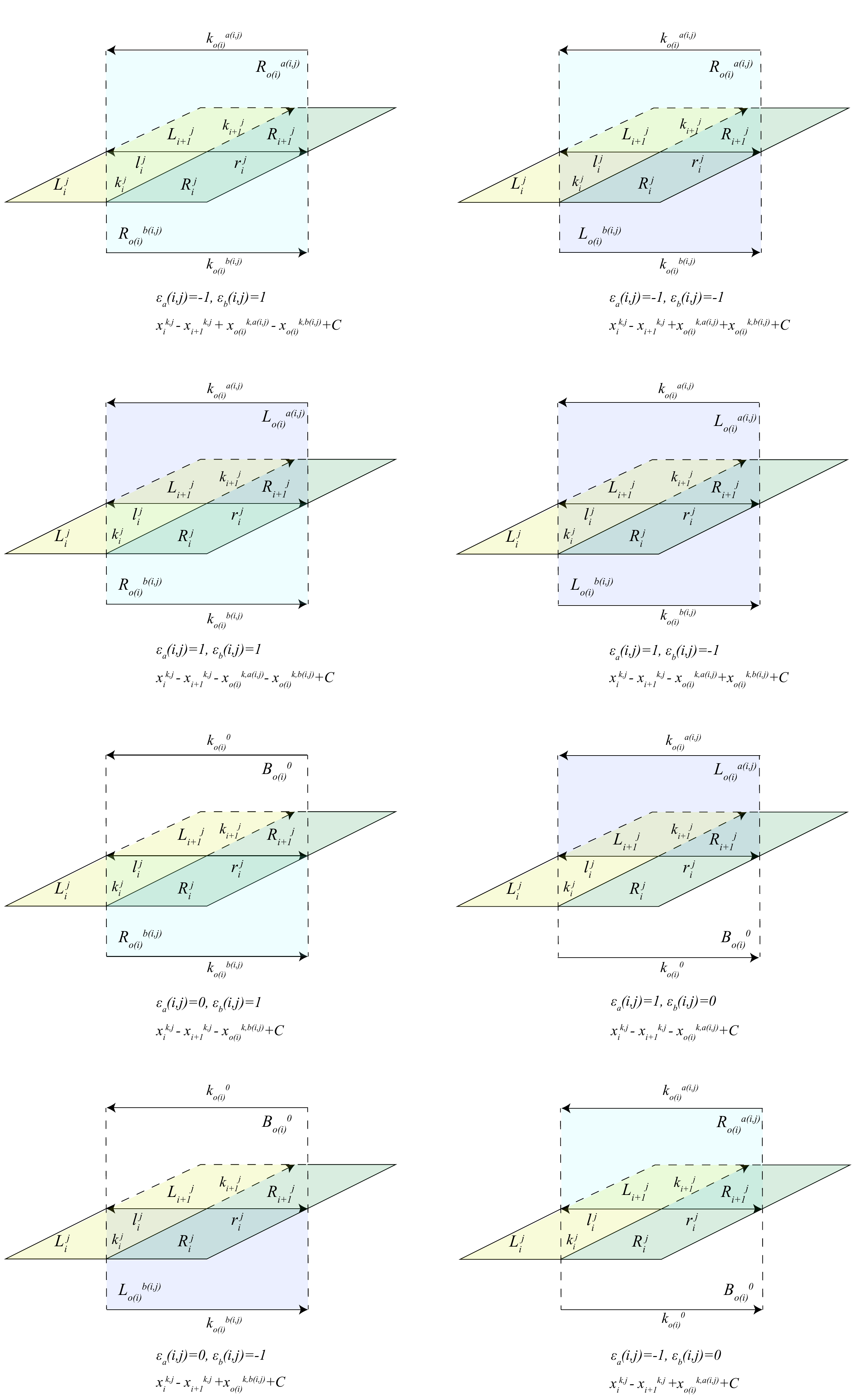}
\caption{Configurations of cells incident to $r_i^j-l_i^j$ at a negative crossing. }
\label{rijlijcasesneg.fig}
\end{figure}

\section{Computing Linking Numbers}\label{intersectionnumber.sec}

The linking number of $K^j$ and $K^k$ is the algebraic intersection number of $K^j$ with a rational 2-chain bounding $K^k$.  Theorem ~\ref{chains.thm} determines when such a 2-chain exists and, when it does, describes one such 2-chain $\Sigma^k$ explicitly.  Now in Theorem ~\ref{linking.thm}, we give a formula for the intersection number of $K^j$ with $\Sigma^k$.  By iterating over all $j,k\in \{0,\dots,q\}$, Theorem ~\ref{linking.thm} allows us to compute the dihedral linking invariant of $(K,\rho)$.  In the statement of the theorem, we allow $j=k$, in which case the result can be interpreted as a self-linking number of $K^j$ with respect to a lift of the blackboard framing of the diagram $D_K$ of $K$ that determined the cell structure on $S^3$.

\begin{thm} Let $q=\dfrac{p-1}{2}$ with $p$ odd.  Let $\rho:\pi_1(S^3\rightarrow K)$ be a $p$-coloring of $K$, with corresponding irregular $p$-fold dihedral cover $f:M_\rho\rightarrow S^3$. Suppose that $K^j$ and $K^k$, $j, k\in\{0,\dots, q\}$, are rationally null-homologous connected components of the singular set $f^{-1}(K)$.  Let $\Sigma^k$ be the rational 2-chain given in Theorem ~\ref{chains.thm}.  Then the linking number of $K^j$ with $K^k$ is 
	$$\text{lk}(K^j,K^k)=\sum_{i=0}^{n-1} \text{Int}(i,j,k) $$
	where $\text{Int}(i,j,k)=\epsilon_a(i,j)x_{o(i)}^{k,a(i,j)}-C_a(i,j,k).$
	\label{linking.thm}
\end{thm}
Note that $\epsilon_a(i,j)=0$ when $a(i,j)=0$, so we treat the first term as 0 in this case even though $x_i^{k,0}$ is not defined.

\begin{proof} Let $(K^j)'$ be a push-off of $K^j$, obtained by pushing $K^j$ into the 3-cells $E^{h_i^j}$ such that $(K^j)'$ intersects the 2-skeleton transversely. Let $(k_i^j)'$ denote the corresponding push-off of $k_i^j$.   Let $\text{Int}(i,j,k)$ denote the signed intersection number of $\overline{(k_i^j)'\cup (k_{i+1}^j)'}$ with $\Sigma^k$. 
	
	We now show $\text{Int}(i,j,k)=\epsilon_a(i,j)x_{o(i)}^{k,a(i,j)}-C_a(i,j,k)$.
	
	 Depending on the values of $\epsilon(i)$ and $\epsilon_a(i,j)$, $(K^j)'$ meets one of the 2-cells $R_{o(i)}^{a(i,j)}$, $L_{o(i)}^{a(i,j)}$, or $B_{o(i)}$ transversely, as shown in Figure ~\ref{intersectionnumbers.fig} when $j\neq 0$, and in Figure ~\ref{intersectionnumbersj0.fig} when $j\neq 0$. 
	
	Recall that 
	$$\Sigma^0= \sum_{i=0}^{n-1} B_i + \sum_{j=1}^q \sum_{i=0}^{n-1} x_i^{0,j}(R_i^j-L_i^j),$$
	
	and if $k\neq 0$, 
    
   	$$\Sigma^k=  \sum_{i=0}^{n-1}x_i^{k,k}R_i^k+(1-x_i^{k,k})L_i^k+ \sum_{j\in\{0,\dots,q\}\setminus\{k\} }\quad \sum_{i=0}^{n-1}x_i^{k,j}(R_i^j-L_i^j),$$
	
	where the $x_i^{k,j}$ are a solution to the system of equations in Theorem ~\ref{chains.thm}.

	{\bf Case 1: $k=0$.} 
	
	At crossing $i$, $(K^j)'$ meets $R_{o(i)}^{a(i,j)}$, $L_{o(i)}^{a(i,j)}$, or $B_{o(i)}$, as shown in Figure ~\ref{intersectionnumbers.fig} if $j\neq 0$ or Figure ~\ref{intersectionnumbersj0.fig} if $j=0$.
	
{\bf Case 1a:}	 Suppose $(K^j)'$ meets $R_{o(i)}^{a(i,j)}$.   If $\epsilon(i)=1$, the intersection number of $(K^j)'$ with $R_{o(i)}^{a(i,j)}$ is positive, and the coefficient of $R_{o(i)}^{a(i,j)}$ in $\Sigma^k$ is $x_{o(i)}^{k,a(i,j)}$. Then $\text{Int}(i,j,k)=x_{o(i)}^{k,a(i,j)}$. In this case, $\epsilon_a(i,j)=1$ and $C_a(i,j,k)=0$.  Therefore $\text{Int}(i,j,k)=\epsilon_a(i,j)x_{o(i)}^{k,a(i,j)}-C_a(i,j,k)$ as desired.	 If instead $\epsilon(i)=-1$, the intersection number of $(K^j)'$ with $R_{o(i)}^{a(i,j)}$ is negative, but the coefficient of $R_{o(i)}^{a(i,j)}$ in $\Sigma^k$ is still $x_{o(i)}^{k,a(i,j)}$.  Therefore $\text{Int}(i,j,k)=-x_{o(i)}^{k,a(i,j)}$.  In this case, $\epsilon(i,j)=-1$ and $C_a(i,j,k)=0$. Therefore $\text{Int}(i,j,k)=\epsilon_a(i,j)x_{o(i)}^{k,a(i,j)}-C_a(i,j,k)$ as well.
	
{\bf Case 1b:}	Suppose instead that $(K^j)'$ meets $L_{o(i)}^{a(i,j)}$.  If $\epsilon(i)=1$, the intersection number of $(K^j)'$ with $L_{o(i)}^{a(i,j)}$ is positive, and the coefficient of $L_{o(i)}^{a(i,j)}$ in $\Sigma^k$ is $-x_{o(i)}^{k,a(i,j)}$. Then $\text{Int}(i,j,k)=-x_{o(i)}^{k,a(i,j)}$. In this case, $\epsilon_a(i,j)=-1$ and $C_a(i,j,k)=0$.  Therefore $\text{Int}(i,j,k)=\epsilon_a(i,j)x_{o(i)}^{k,a(i,j)}-C_a(i,j,k)$ as desired.	 If instead $\epsilon(i)=-1$, the intersection number of $(K^j)'$ with $L_{o(i)}^{a(i,j)}$ is negative, but the coefficient of $L_{o(i)}^{a(i,j)}$ in $\Sigma^k$ is still $-x_{o(i)}^{k,a(i,j)}$.  Therefore $\text{Int}(i,j,k)=x_{o(i)}^{k,a(i,j)}$.  In this case, $\epsilon(i,j)=1$ and $C_a(i,j,k)=0$. Therefore $\text{Int}(i,j,k)=\epsilon_a(i,j)x_{o(i)}^{k,a(i,j)}-C_a(i,j,k)$ as well.
	
{\bf Case 1c:}	Finally,  suppose that $(K^j)'$ meets $B_{o(i)}$.  Since $k=0$, the coefficient of $B_{o(i)}$ in $\Sigma^k$ is 1.  The intersection number of $(K^j)'$ with $B_{o(i)}$ is simply $\text{Int}(i,j,k)=\epsilon(i)$.  In this case, $\epsilon_a(i,j)=0$ and $C_a(i,j,k)=-\epsilon(i)$.  Therefore $\text{Int}(i,j,k)=\epsilon_a(i,j)x_{o(i)}^{k,a(i,j)}-C_a(i,j,k)$.

	{\bf Case 2: $k\neq 0$.}
	
	At crossing $i$, $(K^j)'$ meets $R_{o(i)}^{a(i,j)}$, $L_{o(i)}^{a(i,j)}$, or $B_{o(i)}$, as shown in Figure ~\ref{intersectionnumbers.fig} if $j\neq 0$ or Figure ~\ref{intersectionnumbersj0.fig} if $j=0$.

	{\bf Case 2a:} Suppose $(K^j)'$ meets $R_{o(i)}^{a(i,j)}$.  This is identical to Case 1a.
	
	{\bf Case 2b:} Suppose $(K^j)'$ meets $L_{o(i)}^{a(i,j)}$ and $a(i,j)\neq k$.  This is identical to Case 1b.
	
	{\bf Case 2c:} Suppose $(K^j)'$ meets $L_{o(i)}^{a(i,j)}$ and $a(i,j)= k$. If $\epsilon(i)=1$, the intersection number of $(K^j)'$ with $L_{o(i)}^{a(i,j)}$ is positive, and the coefficient of $L_{o(i)}^{a(i,j)}$ in $\Sigma^k$ is $1-x_{o(i)}^{k,a(i,j)}$. Then $\text{Int}(i,j,k)=1-x_{o(i)}^{k,a(i,j)}$. In this case, $\epsilon_a(i,j)=-1$ and $C_a(i,j,k)=-1$.  Therefore $\text{Int}(i,j,k)=\epsilon_a(i,j)x_{o(i)}^{k,a(i,j)}-C_a(i,j,k)$ as desired.	 If instead $\epsilon(i)=-1$, the intersection number of $(K^j)'$ with $L_{o(i)}^{a(i,j)}$ is negative, but the coefficient of $L_{o(i)}^{a(i,j)}$ in $\Sigma^k$ is still $1-x_{o(i)}^{k,a(i,j)}$.  Therefore $\text{Int}(i,j,k)=x_{o(i)}^{k,a(i,j)}-1$.  In this case, $\epsilon_a(i,j)=1$ and $C_a(i,j,k)=1$. Therefore $\text{Int}(i,j,k)=\epsilon_a(i,j)x_{o(i)}^{k,a(i,j)}-C_a(i,j,k)$.

	{\bf Case 2d:}	Suppose that $(K^j)'$ meets $B_{o(i)}$. Since $k\neq 0$, the coefficient of $B_{o(i)}$ in $\Sigma^k$ is 0. Therefore $\text{Int}(i,j,k)=0$.  In this case, $\epsilon_a(i,j)=0$ and $C_a(i,j,k)=0$.  Therefore $\text{Int}(i,j,k)=\epsilon_a(i,j)x_{o(i)}^{k,a(i,j)}-C_a(i,j,k)$.
	\newpage

	\begin{figure}[htbp]
	\includegraphics[width=\textwidth]{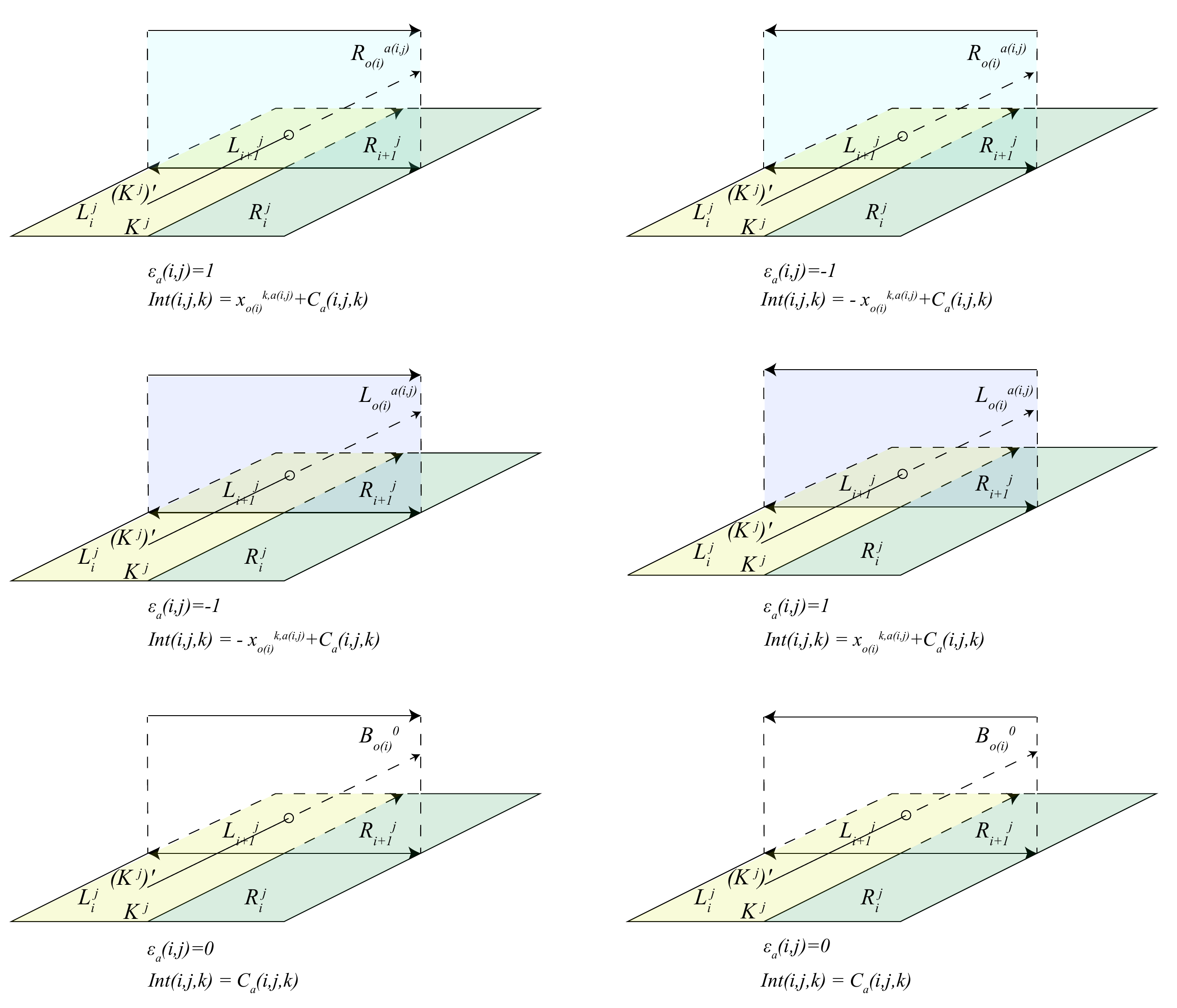}
	
	\caption{}
	\label{intersectionnumbers.fig}
	\caption{The push-off $(K_j)'$ of $K_j$ meets one of the 2-cells in $\Sigma^k$.}
	\end{figure}
	
	\begin{figure}[htbp]
	\includegraphics[width=\textwidth]{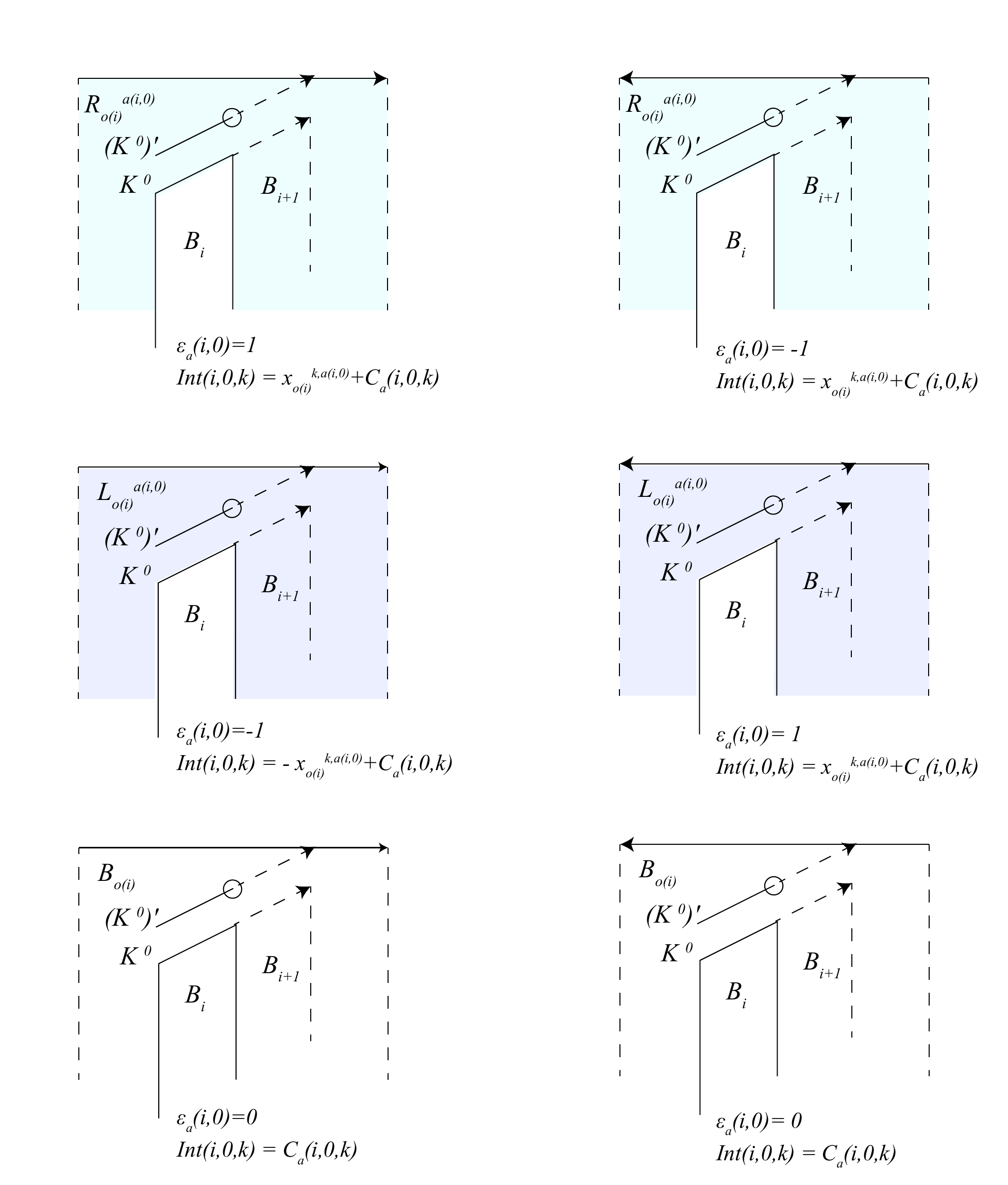}
	
	\caption{}
	\label{intersectionnumbersj0.fig}
	\caption{The push-off $(K_0)'$ of $K_0$ meets one of the 2-cells in $\Sigma^k$.}
	\end{figure}
\end{proof}

\newpage

\section{Additional Examples}\label{examples.sec}

In this section, we carry out our algorithm in full detail on the trefoil.  We then include a few additional examples to illustrate the use of the code at ~\cite{cahngithubdihedrallinking}.

\subsection{The trefoil}

Consider the trefoil knot with 3-coloring and configuration diagram shown in Figure ~\ref{configtrefoil.fig}.

From the configuration diagram, we compute the values of $a(i,j)$ and $b(i,j)$, shown in the tables below.
\begin{center}
\begin{tabular}{|c|c|c|}
\hline
$a(i,j)$&$j=0$&$j=1$\\ \hline
$i=0$&$1$&$1$ \\ \hline
$i=1$&$1$&$0$ \\ \hline
$i=2$&$1$&$1$ \\ \hline
$i=3$&$0$&$1$ \\ \hline
\end{tabular}
\hspace{1in}
\begin{tabular}{|c|c|c|}
\hline
$b(i,j)$&$j=0$&$j=1$\\ \hline
$i=0$&$1$&$0$ \\ \hline
$i=1$&$1$&$1$ \\ \hline
$i=2$&$1$&$0$ \\ \hline
$i=3$&$0$&$1$ \\ \hline
\end{tabular}
\end{center}
\begin{figure}
\includegraphics[width=6in]{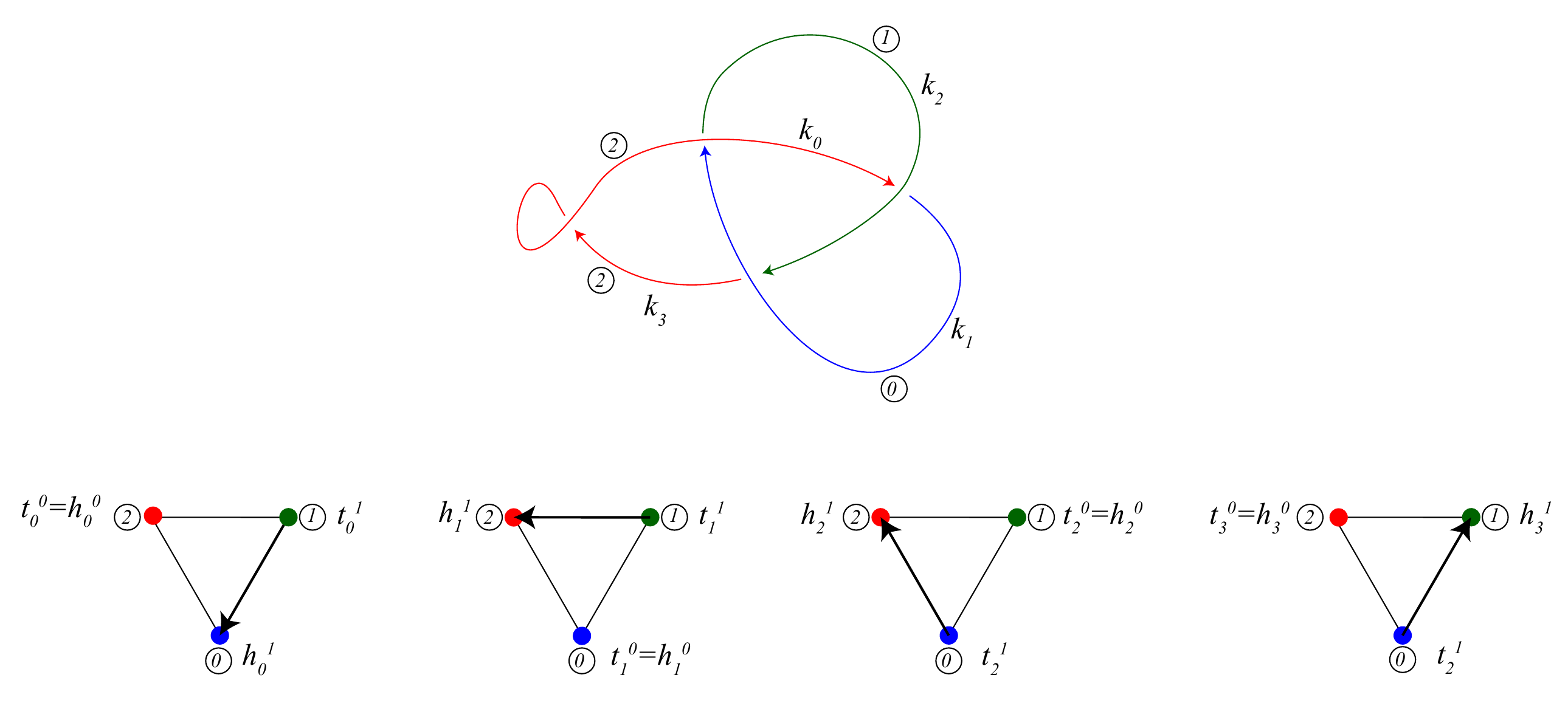}
\caption{A 3-colored diagram of the trefoil, and the corresponding configuration diagrams.}
\label{configtrefoil.fig}
\end{figure}

We then compute the values of $\epsilon_a(i,j)$ and $\epsilon_b(i,j)$, shown in the tables below.
\begin{center}
\begin{tabular}{|c|c|c|}
\hline
$\epsilon_a(i,j)$&$j=0$&$j=1$\\ \hline
$i=0$&$1$&$-1$ \\ \hline
$i=1$&$1$&$0$ \\ \hline
$i=2$&$-1$&$1$ \\ \hline
$i=3$&$0$&$1$ \\ \hline
\end{tabular}
\hspace{1in}
\begin{tabular}{|c|c|c|}
\hline
$\epsilon_b(i,j)$&$j=0$&$j=1$\\ \hline
$i=0$&$-1$&$0$ \\ \hline
$i=1$&$-1$&$1$ \\ \hline
$i=2$&$1$&$0$ \\ \hline
$i=3$&$0$&$1$ \\ \hline
\end{tabular}
\end{center}

Last, we compute the values of $C_a(i,j,k)$ and $C_b(i,j,k)$.
\begin{center}
\begin{tabular}{|c|c|c|}
\hline
$C_a(i,j,0)$&$j=0$&$j=1$\\ \hline
$i=0$&$0$&$0$ \\ \hline
$i=1$&$0$&$-1$ \\ \hline
$i=2$&$0$&$0$ \\ \hline
$i=3$&$-1$&$0$ \\ \hline
\end{tabular}
\hspace{1in}
\begin{tabular}{|c|c|c|}
\hline
$C_b(i,j,0)$&$j=0$&$j=1$\\ \hline
$i=0$&$0$&$1$ \\ \hline
$i=1$&$0$&$0$ \\ \hline
$i=2$&$0$&$1$ \\ \hline
$i=3$&$1$&$0$ \\ \hline
\end{tabular}
\end{center}

\begin{center}
\begin{tabular}{|c|c|c|}
\hline
$C_a(i,j,1)$&$j=0$&$j=1$\\ \hline
$i=0$&$0$&$-1$ \\ \hline
$i=1$&$0$&$0$ \\ \hline
$i=2$&$-1$&$0$ \\ \hline
$i=3$&$0$&$0$ \\ \hline
\end{tabular}
\hspace{1in}
\begin{tabular}{|c|c|c|}
\hline
$C_b(i,j,1)$&$j=0$&$j=1$\\ \hline
$i=0$&$0$&$0$ \\ \hline
$i=1$&$0$&$1$ \\ \hline
$i=2$&$1$&$0$ \\ \hline
$i=3$&$0$&$1$ \\ \hline
\end{tabular}
\end{center}
Now we determine the $\alpha_i^{k,j}$ using Theorem ~\ref{chains.thm}.  Each table gives the augmented matrix corresponding to the system of equations $\alpha_i^{k,j}=0$ for $i\in\{0,1,2,3\}$, $j\in\{0,1\}$, and fixed $k\in \{0,1\}$.
\begin{center}
\begin{tabular}{|c|c|c|c|c||c|}
\hline
$\Sigma^0$&$x_0^{0,1}$&$x_1^{0,1}$&$x_2^{0,1}$&$x_3^{0,1}$&
\\ \hline
$\alpha_0^{0,1}$&1&-1&1&0&-1
\\ \hline
$\alpha_1^{0,1}$&-1&1&-1&0&1
\\ \hline
$\alpha_2^{0,1}$&0&-1&1&-1&-1
\\ \hline
$\alpha_3^{0,1}$&-1&0&0&-1&0
\\ \hline
\end{tabular}
\hspace{1in}
\begin{tabular}{|c|c|c|c|c||c|}
\hline
$\Sigma^1$&$x_0^{1,1}$&$x_1^{1,1}$&$x_2^{1,1}$&$x_3^{1,1}$&
\\ \hline
$\alpha_0^{1,1}$&1&-1&1&0&1
\\ \hline
$\alpha_1^{1,1}$&-1&1&-1&0&-1
\\ \hline
$\alpha_2^{1,1}$&0&-1&1&-1&0
\\ \hline
$\alpha_3^{1,1}$&-1&0&0&-1&-1
\\ \hline
\end{tabular}
\end{center}
One choice of solution for each system is:
\begin{center}
\begin{tabular}{|c|c|c|c|}
\hline
$x_0^{0,1}$&$x_1^{0,1}$&$x_2^{0,1}$&$x_3^{0,1}$
\\ \hline
0&1&0&0
\\ \hline
\end{tabular}
\hspace{1in}
\begin{tabular}{|c|c|c|c|}
\hline
$x_0^{1,1}$&$x_1^{1,1}$&$x_2^{1,1}$&$x_3^{1,1}$
\\ \hline
1&0&0&0
\\ \hline
\end{tabular}
\end{center}

Finally, we determine the linking numbers of the $K^j$ using Theorem~\ref{linking.thm}.

\begin{tabular}{|c|c|c|c||c|}
\hline
Int(0,0,0)&Int(1,0,0)&Int(2,0,0)&Int(3,0,0)&$\text{lk}(K^0,K^0)$
\\ \hline
0&0&-1&1&0
\\ \hline
\end{tabular}

\begin{tabular}{|c|c|c|c||c|}
\hline
Int(0,0,1)&Int(1,0,1)&Int(2,0,1)&Int(3,0,1)&$\text{lk}(K^0,K^1)$
\\ \hline
0&1&1&0&2
\\ \hline
\end{tabular}

\begin{tabular}{|c|c|c|c||c|}
\hline
Int(0,1,0)&Int(1,1,0)&Int(2,1,0)&Int(3,1,0)&$\text{lk}(K^1,K^0)$
\\ \hline
0&1&1&0&2
\\ \hline
\end{tabular}

\begin{tabular}{|c|c|c|c||c|}
\hline
Int(0,1,1)&Int(1,1,1)&Int(2,1,1)&Int(3,1,1)&$\text{lk}(K^1,K^1)$
\\ \hline
1&0&0&0&1
\\ \hline
\end{tabular}

Therefore, the dihedral linking invariant of the trefoil is $\{\text{lk}(K^0,K^1)\}=\{2\}$.
\subsection{The knot $8_{16}$}.  The knot $8_{16}$ has determinant 35, so is both 5- and 7-colorable. We now illustrate how to use the code available at ~\cite{cahngithubdihedrallinking} to determine the 5- and 7-dihedral linking invariants.  

The list of over-arc subscripts $[o(0),\dots,o(7)]$ is

overstrands=[6, 4, 0, 7, 2, 3, 1, 5].

The list of local writhe numbers $[\epsilon(0),\dots,\epsilon(7)]$ is 

signs=[1, 1, 1, -1, 1, -1, 1, -1].

The unique 5-coloring $\rho_5$ $[c(0),\dots,c(7)]$ of $8_{16}$ up to equivalence is

coloring5=[2, 3, 2, 2, 0, 4, 0, 1].

The function DLNmatrix(5, overstrands, signs, coloring5) returns the 3x3 matrix below, whose $jk$-entry is the linking number of $K^j$ and $K^k$.  Note that the diagonal entries are self-linking numbers and are not part of the dihedral linking invariant.

$$\begin{pmatrix}
-22& 18& -6\\ 
18& -14& 6\\
-6 & 6& -2
\end{pmatrix}$$

Therefore $DLN(8_{16},\rho_5)=\{18,6,-6\}.$

The unique 7-coloring $\rho_7$ of $8_{16}$ up to equivalence is

coloring7=[3, 4, 5, 1, 1, 2, 0, 1].

 The function DLNmatrix(7, overstrands, signs, coloring7) returns the 4x4 matrix below, whose $jk$-entry is the linking number of $K^j$ and $K^k$.  Again note that the diagonal entries are self-linking numbers and are not part of the dihedral linking invariant.
 
 $$\begin{pmatrix}
 22& -6& -22& 18\\
 -6& 0& 6& -2\\
 -22& 6& 20& -14\\
 18& -2& -14& 8
 \end{pmatrix}
$$

Therefore $DLN(8_{16},\rho_7)=\{-6,-22,18,6,-2,-14\}.$

\section{Application: Pseudo-Branch Curves and Computing Ribbon Obstructions}\label{xi.sec}

In ~\cite{kjuchukova2018dihedral}, Kjuchukova introduced an invariant $\Xi_p(K,\rho)$, of a knot $K$ equipped with a $p$-coloring $\rho$.  This invariant gives rise to a homotopy-ribbon obstruction~\cite{cahnkjuchukova2017singbranchedcovers,geske2021signatures}, as well as bounds on 4-genera associated to $p$-colorable knots~\cite{cahnkju2018genus}. In~\cite{cahnkjuchukova2018computing}, the first author and Kjuchukova gave a diagrammatic algorithm for computing $\Xi_p$. This algorithm requires one to compute linking numbers of lifts of curves in $S^3- K$ to the $p$-fold dihedral cover of $S^3$ along $K$; that is, one must compute the linking number of curves other than the branch curves.  The first author and Kjuchukova gave an algorithm for computing these linking numbers when $p=3$ in~\cite{cahnkjuchukova2016linking}. In addition, all of the examples of the $\Xi_p$ computation in \cite{cahnkjuchukova2018computing} are carried out for $p=3$.  In the following sections, we give an overview of the linking number computation for a general odd number $p$, and then carry out the $\Xi_p$ computation for an example with $p>3$. This is the first such computation of the $\Xi_p$ invariant.  Another explicit method for computing $\Xi_p$ is given in~\cite{cahn2023algorithms}, but that method requires additional data---a $p$-colorable surface in $B^4$ over which the $p$-coloring of $K$ extends---and such surfaces need not exist in general~\cite{kjorr2020admissible}.

\subsection{Computing linking numbers of pseudo-branch curves}\label{pseudobranch.sec} The algorithm for computing the linking numbers of the branch curves in a $p$-fold dihedral cover of $S^3$ along $K$ extends to an algorithm for computing linking numbers of arbitrary curves in the $p$-fold dihedral branched cover, presented as lifts of curves $\gamma$ and $\delta$, called {\it pseudo-branch curves}, in $S^3- K$.  Such an algorithm was given in the case $p=3$ by the first author and Kjuchukova in ~\cite{cahnkjuchukova2016linking}. Here we give an overview of the algorithm for general odd $p$. This algorithm is implemented in the file \verb|dihedrallinking_pseudobranch.py| at~\cite{cahngithubdihedrallinking} and is used to perform the computation of $\Xi_p$ in the next section.

We begin with a diagram of $K\cup \gamma\cup \delta$. The arcs of the knot $K$ are equipped with a $p$-coloring. Note that the pseudo-branch curves $\gamma$ and $\delta$ are not colored. In the corresponding map $\pi_1(S^3- K\cup\gamma\cup \delta)\twoheadrightarrow D_p$, meridians of $\gamma$ and $\delta$ map to the identity.  Each of $\gamma$ and $\delta$ has $p$ path-lifts to the irregular $p$-fold dihedral cover of $S^3$ branched along $K$. For simplicity we will assume that each of these lifts is closed, as is the case in our application. Denote these lifts by $\gamma^j$ and $\delta^k$ for $j,k\in\mathbb{Z}_p$. The arcs of the diagram of $K\cup \gamma\cup \delta$ are labeled as described in \cite{cahnkjuchukova2016linking}. Briefly, we denote the arcs of $\gamma$ by $q_0,\dots, q_m$ and the arcs of $\delta$ by $p_0,\dots, p_m$. The lift of $q_i$ that lies on the lift $\gamma^j$ is denoted $q_i^j$. Similarly, $p_i^k$ denotes the lift of $p_i$ that lies on $\delta^k.$

The steps of the algorithm are as follows:
\begin{enumerate}
\item Equip $S^3$ with a cell structure determined by the cone on the link $K\cup \gamma\cup \delta$ (presented as a link diagram).  This structure consists of ``horizontal'' 1-cells determined by arcs of $K$, $\gamma$, and $\delta$;``vertical'' 2-cells below arcs of $K$, $\gamma$, and $\delta$;  ``vertical'' 1-cells below all crossings of the diagram; one 3-cell in the complement of the cone; and 0-cells as needed.
\item Lift this cell-structure to the $p$-fold irregular dihedral cover of $S^3$ branched along $K.$ 
\item Determine the local configuration of the lift of the cell structure in a neighborhood of each crossing.  For self-crossings of $K$, the lifts are as described in Figures~\ref{rijlijcases.fig} and ~\ref{rijlijcasesneg.fig}.  For positive crossings of $K$ over or under $\gamma$ or $\delta$, the lifts are as described in Figures ~\ref{pboverk.fig} and ~\ref{koverpb.fig}. One can draw analogous figures for negative crossings.  
\item Set up a system of linear equations whose solution determines a 2-chain $\Sigma^k$ bounding each rationally null-homologous closed lift $\delta^k$ of $\delta$. Above crossing $i$, the 2-chain will include one copy of the 2-cell incident to $p_{o(i)}^k$, as well as some number of copies of each $R_i^l-L_i^l$. The system can be set up by analyzing Figure ~\ref{pboverk.fig}, and the analogous figure for negative crossings.
\item Compute the intersection number of each rationally null-homologous closed lift $\gamma^j$ of $\gamma$ with each 2-chain $\Sigma^k$ above.   The intersection numbers can be determined by analyzing Figure ~\ref{koverpb.fig}, and the analogous figure for negative crossings.  In particular, each arc $q_i^j$ of $\gamma^j$ terminates at some 2-cell $C$; the local intersection number at crossing $i$ can be computed by multiplying the coefficient of $C$ in $\Sigma^k$ by the sign of crossing $i$.
\end{enumerate}

\begin{figure}[htbp]
	\includegraphics[width=6in]{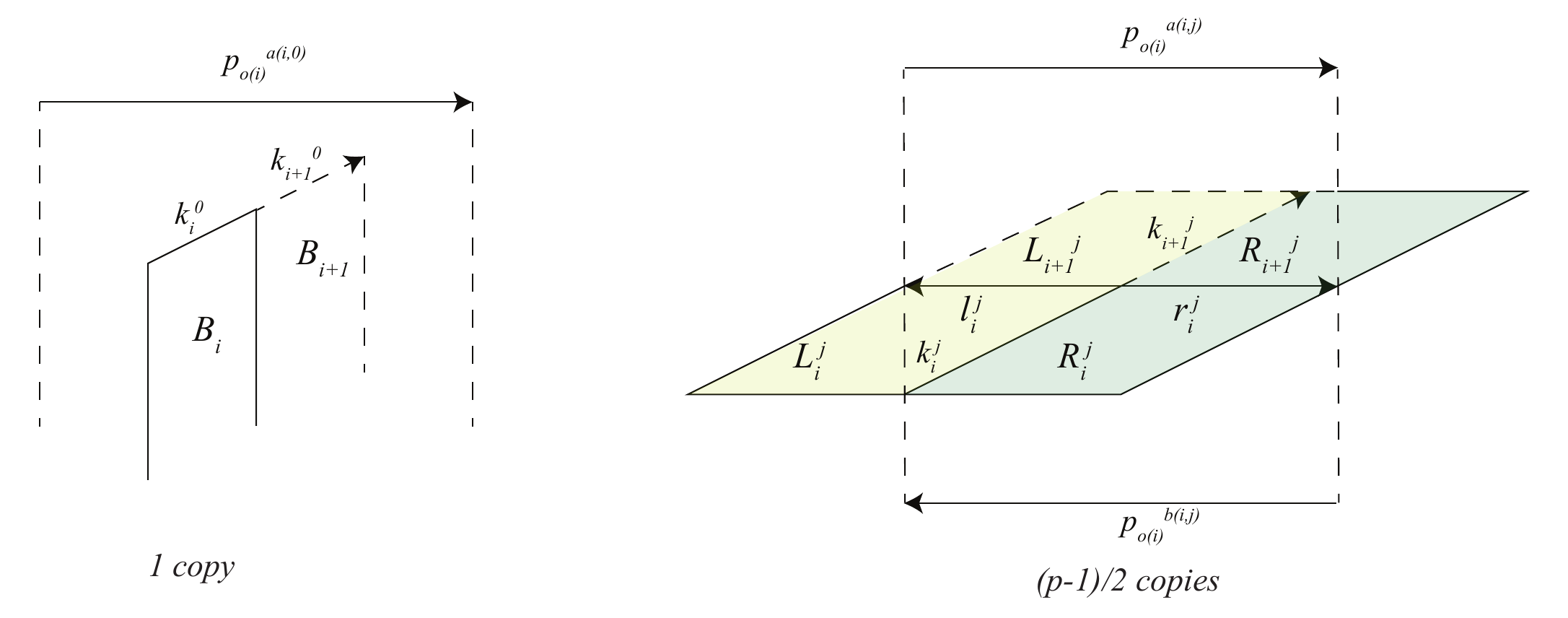}
	\caption{Configuration of 2-cells above a crossing of arc $k_i$ of the knot $K$ under an arc $p_{o(i)}$ of the pseudo-branch curve $\delta$. }\label{pboverk.fig}
\end{figure}

\begin{figure}[htbp]
	\includegraphics[width=6in]{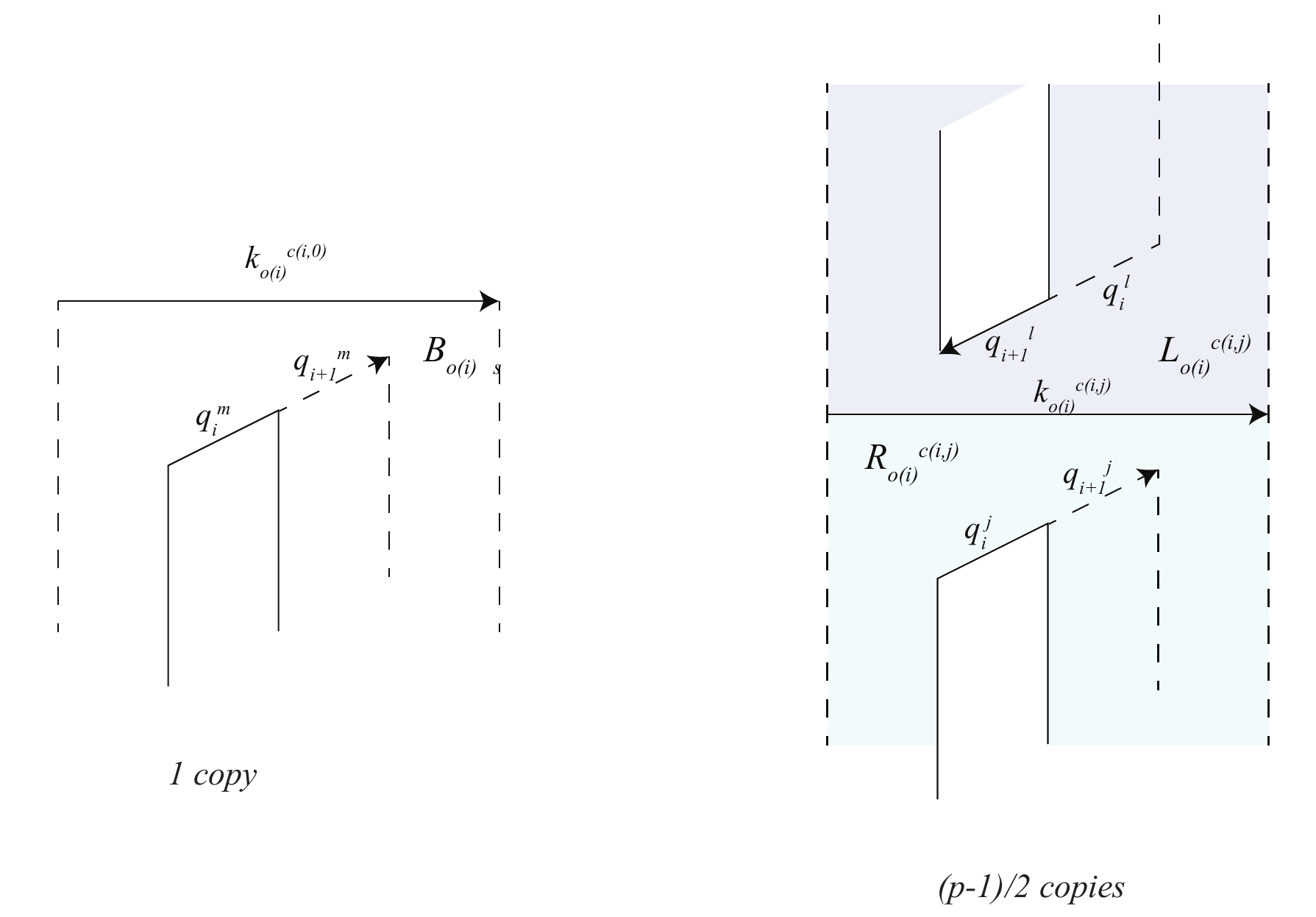}
	\caption{Configuration of 2-cells above a crossing of arc $q_i$ of the pseudo-branch curve $\gamma$ under arc $k_{o(i)}$ of the knot $K$. }\label{koverpb.fig}
\end{figure}
\subsection{Computing ribbon obstructions}\label{ribbonobs.sec}
Let $K$ be a $p$-colored knot in $S^3$ with Seifert surface $\Sigma$, and corresponding Seifert matrix $V$.  A mod $p$ {\it characteristic knot} for $K$ is an embedded curve $\beta\subset \Sigma$ such that $(V+V^T)\beta \equiv 0 \mod p$ \cite{CS1984linking}. Homology classes of mod $p$ characteristic knots in $H_1(\Sigma)$ are in 1-1 correspondence with $p$-colorings of $K$.

Cappell and Shaneson \cite{CS1984linking} constructed a cobordism $W(K,\beta)$ from the $p$-fold irregular dihedral cover of $S^3$ along $K$, corresponding to a coloring $\rho$ of $K$, to the $p$-fold cyclic cover of $S^3$ along a characteristic knot $\beta$ corresponding to $\rho$. The following formula for $\Xi_p(K,\rho)$ from \cite{kjuchukova2018dihedral} depends on the signature $\sigma(W(K,\beta))$ of this cobordism, as well as the Tristram-Levine signatures $\sigma_{\zeta_i}$ of $\beta$, where the $\zeta_i$ denote primitive $p^{\text{th}}$ roots of unity, and the self-linking number of $\beta$ with respect to the symmetrized Seifert form $L_V$ for $\Sigma$:

$$\Xi_p(K,\rho)=\dfrac{p^2-1}{6p}L_V(\beta,\beta)+\sum_{i=1}^{p-1}\sigma_{\zeta_i}(\beta)+\sigma(W(K,\beta)).$$

The computation of the first two terms is straightforward. The term $\sigma(W(K,\beta))$ is equal to the signature of a matrix of linking numbers of certain lifts of a basis for $H_1(\Sigma- \beta)$ to the irregular $p$-fold dihedral cover of $K$. One can determine which lifts to choose directly from diagrammatic data, as laid out in \cite{cahnkjuchukova2018computing}.

We first give an overview of the method of constructing the $p$-fold dihedral cover of $S^3$ along $K$ from the $p$-fold cyclic cover $M_\beta$ of $S^3$ along $\beta$, and the construction of the cobordism $W(K,\beta)$, from \cite{CS1984linking}:
\begin{itemize}
	\item Let $\Sigma$ be a Seifert surface for $K$ and $\beta\subset \Sigma$ a mod $p$ characteristic knot for $K$.
	\item Thicken the Seifert surface to obtain $\Sigma\times [-1,1]$ and let $J$ denote its lift to $M_\beta$. See Figure ~\ref{star.fig}, top. Note that $\partial J$ is a union of $2p$ copies of $\Sigma- \beta$, identified along lifts of the right and left push-offs $\beta_r$ and $\beta_l$ of $\beta$ in $\Sigma$.
	\item Let $h:\Sigma\times[-1,1]\rightarrow \Sigma\times [-1,1]$ be the involution mapping $x\times t$ to $x\times -t$, and $\bar{h}$ a lift of $h$ to $J$.  
	\item The $p$-fold irregular dihedral cover of $S^3$ along $K$ is homeomorphic to the manifold obtained from $M- \text{int}(J)$ by identifying points on its boundary $\partial J$ via $\bar{h}$. The image of $\partial J$ under this involution is shown in Figure ~\ref{star.fig}, bottom.
	\item The cobordism $W(K,\beta)$ is the mapping cone on $\bar{h}$, namely $M_\beta\times [0,1]/\bar{h}\times 1$. 
\end{itemize}

Let $\{\omega_1,\omega_2,\dots,\omega_{g-1}\}$ denote a basis for $H_1(\Sigma- \beta)$, and let $\omega_i^+$ and $\omega_i^-$ denote the push-offs of these curves to $\Sigma\times 1$ and $\Sigma\times -1$ respectively.  Each of $\omega_i^+$ and $\omega_i^-$ has $p$ lifts to $M_\beta$, which lie on $\partial J$. The same holds for the push-offs $\beta_r^{\pm}$ and $\beta_l^{\pm}$.  These lifts bound cylinders in $J$ in pairs, and for each curve in $\{\omega_1,\omega_2,\dots,\omega_{g-1},\beta\}$, $(p-1)/2$ of these pairs bound cylinders in $\bar{h}(J)$. In the schematic diagrams in Figure ~\ref{star.fig}, these pairs lie directly opposite each other on the ``spokes'' of $J$ and $\bar{h}(J)$ respectively.  The corresponding cylinders can be capped off to form closed classes in $H_2(W(K,\beta))$ as described in \cite{kjuchukova2018dihedral,cahnkjuchukova2018computing}. Therefore the desired signature $\sigma(W(K,\beta))$ is (after a sign correction) the signature of a matrix $M$ of linking numbers of these pairs and their push-offs from $\partial J$ into $M_\beta- \text{int}(J)$; that is, $\sigma(W(K,\beta))=-\sigma(M)$~\cite{kjuchukova2018dihedral}, \cite[Theorem 1]{cahnkjuchukova2018computing}.

\subsection{Example of $\Xi_p$ computation} \label{xiexample.sec} We will show that $\Xi_5(K)=6$, where $K$ is the 5-colorable knot in Figure~\ref{K35.fig}.   We chose this knot because, in addition to being 5-colorable, it has a diagram in which the characteristic knot $\beta$ and the basis for $H_1(\Sigma-\beta)$ are relatively simple, but the Seifert surface still has high enough genus to demonstrate the full complexity of the computation. Another reason for choosing this particular knot is that we can confirm the value of $\Xi_p(K)$ by an alternative method, as follows. The knot $K$ is the knot $K_{3,5}$ in the family of knots in Figure 21 of \cite{cahn2023algorithms}. Section 9 of \cite{cahn2023algorithms} presents a method for computing $\Xi_p(K,\rho)$ using colored tri-plane diagram representations of surfaces in the 4-ball.  The method in \cite{cahn2023algorithms} only applies when one has a $p$-colored surface in $B^4$ with boundary $K$ over which the coloring $\rho$ extends.  This is the case for $K_{3,5}$, but there exist other $p$-colored knots for which no such surface exists \cite{kjuchukova2026extending}. In contrast, the method presented in \cite{cahnkjuchukova2018computing} and implemented below is fully general.

A Seifert surface for the knot $K$, cut along a choice $\beta$ of characteristic knot, is shown in Figure ~\ref{K35.fig}, together with a basis $\{A,B\}$ for $H_1(\Sigma- \beta)$. The characteristic knot $\beta$ is not drawn, but is parallel to its right and left push-offs $\beta_r$ and $\beta_l.$ We pick basepoints on each of $K$, $A$, $B$, $\beta_r$, and $\beta_l$. These basepoints are marked with dots in Figure ~\ref{K35.fig}. Let $\gamma_r$ and $\gamma_l$ be paths connecting a basepoint of $K$ to the basepoints of $\beta_r$ and $\beta_l$ respectively, and let $\gamma$ be the closed curve on $\Sigma$ obtained by concatenating $\gamma_r$ and $\gamma_l$ and identifying their endpoints on $\beta_r$ and $\beta_l$. 

We first compute the symmetrized Seifert matrix of $K$ with respect to the basis $\{A,B,\gamma,\beta\}$, from which one can confirm that $\beta$ is a mod 5 characteristic knot for $K$:

$$[L_V]=\begin{pmatrix}2&-1&0&0\\
-1&2&-1&0\\
0&-1&-4&-10\\
0&0&-10&0
\end{pmatrix}.$$

In addition, note that the self-linking term $L_V(\beta,\beta)$ of $\Xi_5(K)$ is equal to zero.
 
We equip $S^3$ with the cell structure determined by the cone on $K=K_{3,5}$ as described in Section ~\ref{setup.sec}.  

For each curve $C\in\{A,B,\beta,\beta_r,\beta_l\}$, let $C^+$ denote the push-off of $C$ to $\Sigma\times 1$ and let $C^-$ denote the push-off of $C$ to $\Sigma\times -1$. For $i\in \{0,\dots,4\}$, let $C^i$, $C^{i,+}$, and $C^{i,-}$ denote the lifts of $C$, $C^+$, and $C^-$ such that the lift of the basepoint lies in the 3-cell $E^i.$ Now let $C\in \{A,B,\beta_r\}$. We choose paths $\delta_A, \delta_B$ and $\gamma_r$ (called {\it anchor paths} in \cite{cahnkjuchukova2018computing}) connecting the basepoint of each $C$ to the basepoint of $K$; for simplicity we choose these paths such that their interiors are disjoint from $K$. To determine which pairs $C^{i,+}$ and $C^{j,-}$ bound cylinders in $J$, observe that traveling from the basepoint of $C^+$  to the basepoint of $K^+$ along a positive push-off of $\delta_A$, $\delta_B$, or $\gamma_r$, and then to the basepoint of $C^-$ along the negative push-off of $\delta_A$, $\delta_B$, or $\gamma_r$, requires passing through the vertical 2-cell below the basepoint of $K$ one time. If $K$ is colored $c$, this means one will pass from the 3-cell $E^i$ to the 3-cell $E^{\text{Ref}_c(i)}$, where $j=\text{Ref}_c(i)$ if vertices $j$ and $i$ of the regular $p$-gon are related by reflection over the line through vertex $c$.  Therefore for $C\in \{A,B,\beta_r\}$, $C^{i,+}$ and $C^{\text{Ref}_c(i),-}$ bound a cylinder in $J.$ See Figure ~\ref{cylinder_boundaries.fig}.  The superscripts of the lifts of $\beta_l$ can be determined by the monodromy associated to the anchor path $\gamma_l$, as discussed in \cite{cahnkjuchukova2018computing}; this information can be used to determine the superscripts of all lifts of positive and negative push-offs of $\{A,B,\beta,\beta_r,\beta_l\}$ as shown in Figure~\ref{star.fig}.

For our 5-colored diagram of $K=K_{3,5}$, note that the basepoint of $K$ lies on an arc colored 1.  When the vertices of a regular 5-gon are labeled 0,...,4 in counter-clockwise order, the vertices 0 and 2 are related by reflection over the line through 1, as are the vertices 3 and 4.  Therefore for $C\in\{A,B,\beta_r\}$, $C^{0,+}$ and $C^{2,-}$ bound a cylinder in $J$, as do $C^{3,+}$ and $C^{4,-}$.  After applying the involution $\bar{h}$, $C^{i,\pm}$ is identified with $C^{i,\mp}$.  Therefore $C^0$ and $C^2$ bound a cylinder in $\bar{h}(J)$, as do $C^3$ and $C^4$.

Recall that $\sigma(W(K,\beta))=-\sigma(M)$, where $M$ is a matrix of linking numbers of curves that bound cylinders in $\bar{h}(J)$.  In this case, $M$ is a $6\times 6$ matrix, whose entries are, for each $C\in\{A,B,\beta\}$, linking numbers of $C^0-C^2$ and $C^3-C^4$ with their push-offs into the interior of $M_\beta- \text{int}(J)$, namely $C^{0,+}-C^{2,-}$ and $C^{3,+}-C^{4,-}$. (In order to land in the interior of $M_\beta- \text{int}(J)$, we assume these push-offs lie on lifts of $\Sigma \times \pm (1+\epsilon)$.)

We begin by applying the algorithm in Section~\ref{pseudobranch.sec} to compute the linking numbers of all lifts of $A, B$, and $\beta$ with all lifts of each of their positive push-offs with respect to $\Sigma$.  These linking numbers are computed in the file \verb|Xi_computation.ipynb| at \cite{cahngithubdihedrallinking}, and are displayed in Table ~\ref{alllinkingnumbers.tab}.

We use the linking numbers in Table ~\ref{alllinkingnumbers.tab} to determine the matrix $M$, which is shown in Table ~\ref{int_form.tab}.  All eigenvalues of this matrix $M$ are negative.  Therefore $\sigma(M)=-6$, and $\sigma(W(K,\beta))=6.$

Finally we compute $\Xi_5(K_{3,5})$ using the formula in Section~\ref{ribbonobs.sec}. As the characteristic knot is an unknot, its Tristram-Levine signatures vanish. In addition, the self-linking number $L_V(\beta,\beta)$ is zero. Hence $\Xi_5(K_{3,5})=6$. This agrees with the value computed in \cite{cahn2023algorithms} using trisections. 

In~\cite{cahnkju2018genus}, the first author and Kjuchukova introduced the notion of the {\it $p$-dihedral genus} $\frak{g}_p(K,\rho)$ of a knot $K$ with respect to a $p$-coloring $\rho$, which is the minimum genus of a locally flat, orientable surface in $B^4$ over which the coloring $\rho$ extends.  In Theorem 1 of \cite{cahnkju2018genus}, we proved
$$\frak{g}_p(K,\rho)\geq\dfrac{|\Xi_p(K,\rho)|- \text{rk } H_1(M_\rho;\mathbb{Z})}{p-1}-\dfrac{1}{2}.$$
The knot $K$ above is 2-bridge, from which it follows that 1) the coloring $\rho$ in Figure~\ref{K35.fig} is the unique 5-coloring of $K$ up to equivalence, and 2) the corresponding 5-fold dihedral cover $M_\rho$ is homeomorphic to $S^3$.  Substituting $\Xi_5(K,\rho)=6$ into this bound, we conclude $\frak{g}_5(K,\rho)\geq 1.$  Furthermore, as $K$ has a unique 5-coloring, if $K$ were homotopy-ribbon, $\rho$ would have to extend over any choice of homotopy-ribbon disk for $K$ \cite{cahnkjuchukova2017singbranchedcovers,geske2021signatures}. As $\frak{g}_p(K,\rho)\geq 1$, we conclude $K$ is not homotopy-ribbon.  Observe from the symmetrized Seifert form $L_V$ above that the Murasugi signature $\sigma(K)$ is equal to 2, so we can also conclude that the topological 4-genus of $K$ is at least 1 via other methods. A necessary condition for the topological 4-genus and dihedral 4-genus to coincide was given in Theorem 1 of \cite{cahnkju2018genus}. A more thorough comparison of these genera is a worthy topic for further study.

\begin{figure}[htbp]
	\includegraphics[width=4in]{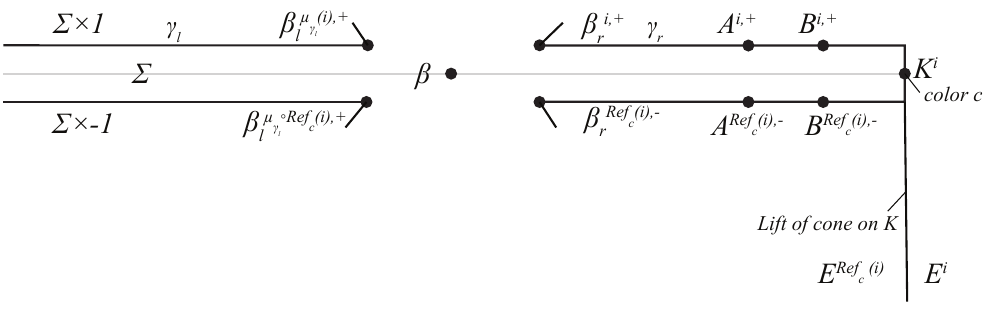}
	\caption{When the basepoint of $K$ lies on an arc of $K$ colored $c$, and the interiours of the anchor paths $\delta_A$, $\delta_B$, and $\gamma_r$ are disjoint from $K$, the superscripts of lifts of $A^\pm$, $B^\pm$, and $\beta^\pm$ that cobound cylinders in $J$ are related by reflection over vertex $c$ of a regular $p$-gon.}
	\label{cylinder_boundaries.fig}
\end{figure}

\begin{figure}[htbp]
	\includegraphics[width=\textwidth]{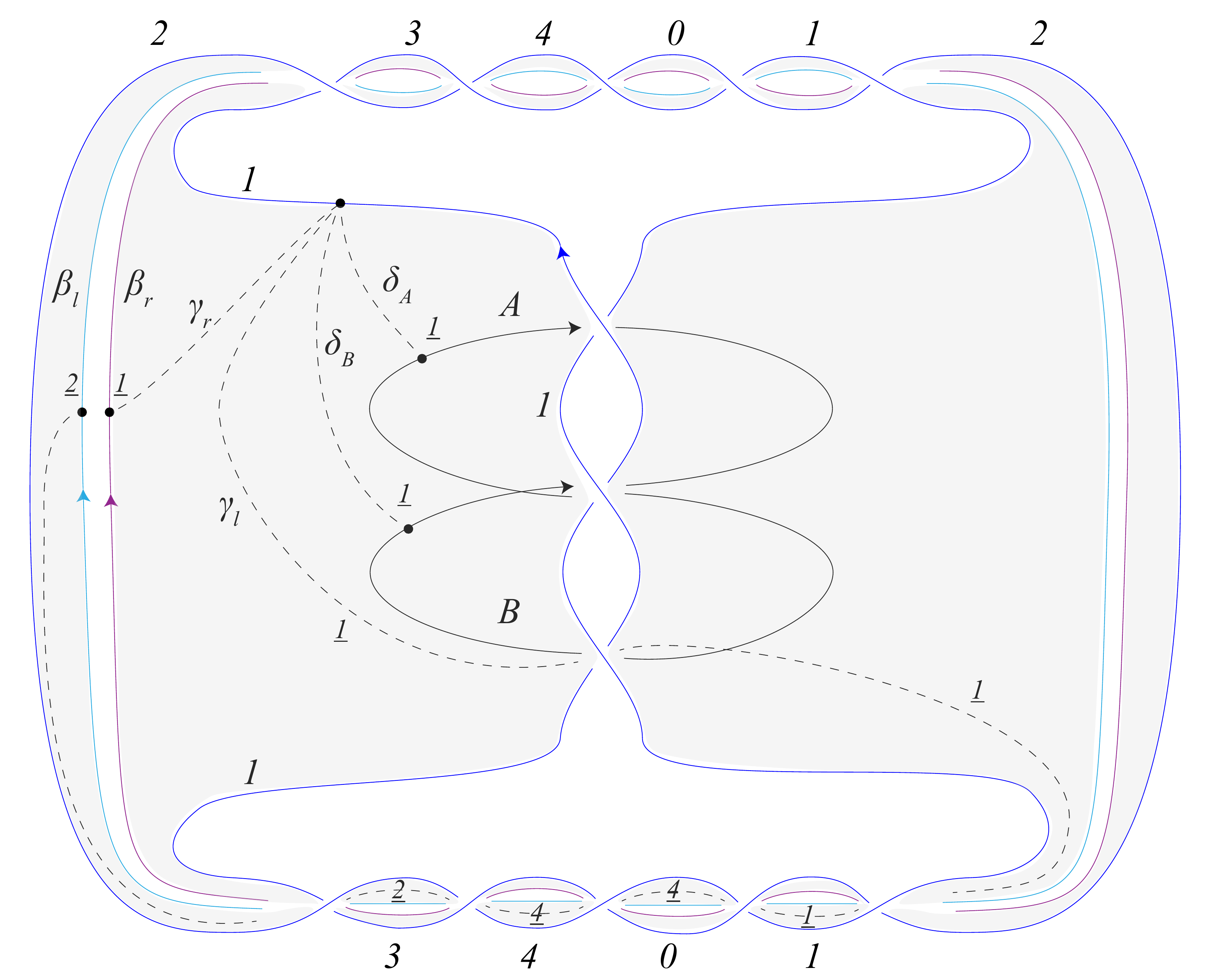}
	\caption{The knot $K_{3,5}$, shown together with a Seifert surface $\Sigma$ cut along a regular neighborhood of a characteristic knot $\beta\subset \Sigma$.  The boundary components of the resulting surface are $K_{3,5}$, and the right and left push-offs $\beta_r$ and $\beta_l$ of $\beta$. The curves $\{A,B\}$ form a basis for $H_1(\Sigma-\beta$).}
	\label{K35.fig}
\end{figure}

\begin{table}
\begin{tabular}{|c|c|c|c|}
	\hline
	& $A+$& $B^+$&$\beta^+$\\
	\hline
	
$A$ &$\begin{pmatrix}0&0&1&0&0\\0&1&0&0&0\\1&0&0&0&0\\0&0&0&0&1\\0&0&0&1&0\end{pmatrix}$&
$\begin{pmatrix}0&0&0&0&0\\0&0&0&0&0\\0&0&0&0&0\\0&0&0&0&0\\0&0&0&0&0\end{pmatrix}$&
$\begin{pmatrix}0&0&0&1&-1\\-1&0&1&-1&1\\1&0&-1&1&-1\\0&0&0&-1&1\\0&0&0&0&0\end{pmatrix}$\\
\hline
$B$ &$\begin{pmatrix}-1&0&0&0&0\\0&-1&0&0&0\\0&0&-1&0&0\\0&0&0&-1&0\\0&0&0&0&-1\end{pmatrix}$&
$\begin{pmatrix}0&0&1&0&0\\0&1&0&0&0\\1&0&0&0&0\\0&0&0&0&1\\0&0&0&1&0\end{pmatrix}$&
$\begin{pmatrix}0&0&0&1&-1\\-1&0&1&-1&1\\1&0&-1&1&-1\\0&0&0&-1&1\\0&0&0&0&0\end{pmatrix}$\\
\hline
$\beta$&$\begin{pmatrix}0&-1&1&0&0\\0&0&0&0&0\\0&1&-1&0&0\\1&-1&1&-1&0\\-1&1&-1&1&0\end{pmatrix}$&
$\begin{pmatrix}0&-1&1&0&0\\0&0&0&0&0\\0&1&-1&0&0\\1&-1&1&-1&0\\-1&1&-1&1&0\end{pmatrix}$&
$\begin{pmatrix}-3&3&-3&3&0\\3&-6&6&-3&0\\-3&6&-6&3&0\\3&-3&3&-3&0\\0&0&0&0&0\end{pmatrix}$\\
\hline	
\end{tabular}
\caption{Linking numbers of the lifts of the curves $A$, $B$, and $\beta$ and their positive push-offs, in the $p$-fold irregular dihedral cover of $K$.}
\label{alllinkingnumbers.tab}
\end{table}
\begin{figure}[htbp]
	\includegraphics[width=4in]{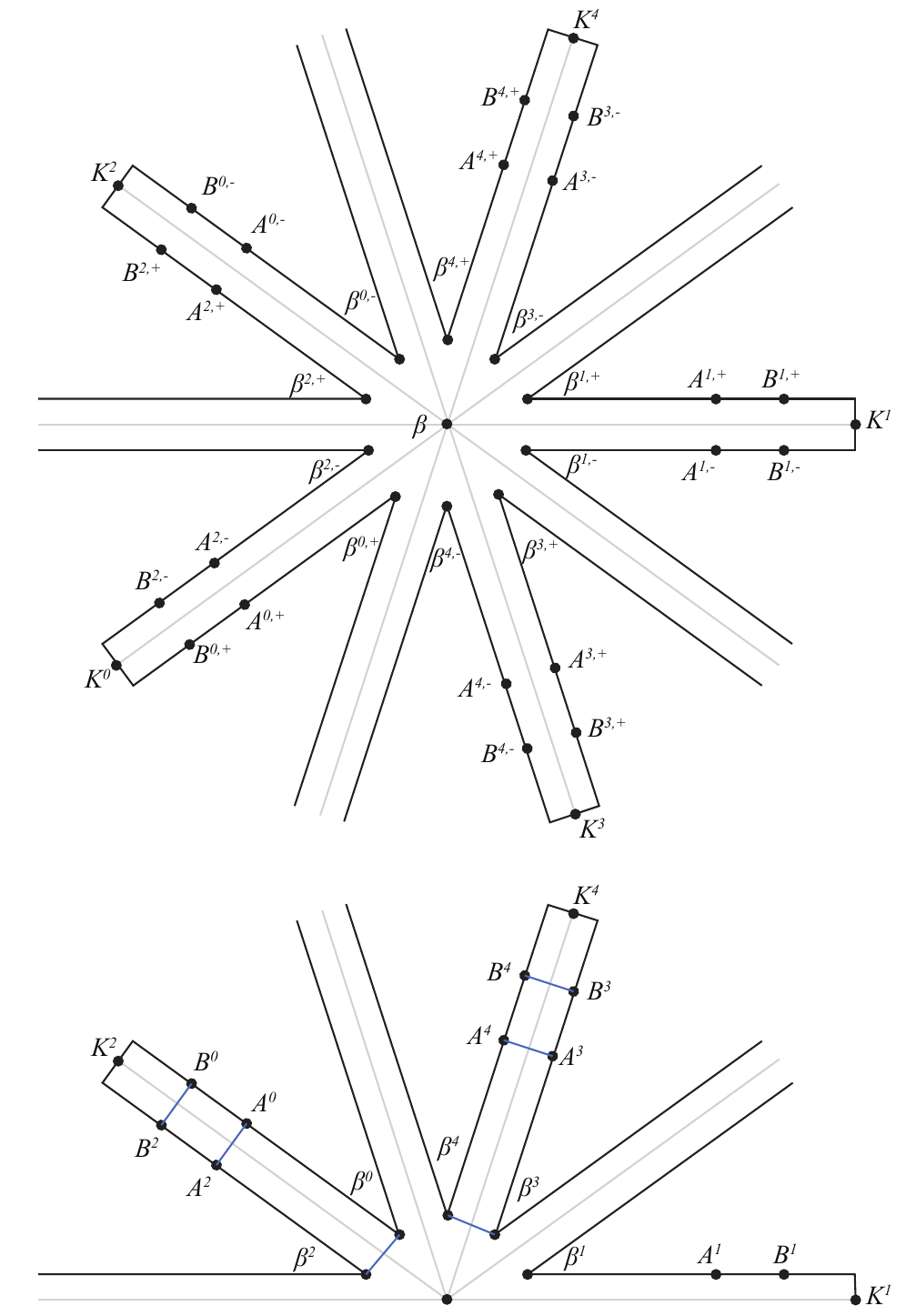}
	
	\caption{The lift of a thickened Seifert surface $\Sigma\times[0,1]$ for $K$ to the $5$-fold cyclic cover of a mod 5 characteristic knot $\beta$ for $K$ (top), and its quotient under an involution (bottom).}
	\label{star.fig}
\end{figure} 

\begin{table}
\begin{tabular}{|c|c|c|c|c|c|c|}
	\hline
	&$A^{0,+}-A^{2,-}$&$A^{3,+}-A^{4,-}$&$B^{0,+}-B^{2,-}$&$B^{3,+}-B^{4,-}$&$\beta^{0,+}-\beta^{2,-}$&$\beta^{3,+}-\beta^{4,-}$\\
	\hline
$A^{0}-A^{2}$	&-2&0&-1&0&-2&0\\
	\hline
$A^{3}-A^{4}$   &0&-2&0&-1&0&-2\\
	\hline
$B^{0}-B^{2}$	&-1&0&-2&0&-2&0\\
	\hline
$B^{3}-B^{4}$	&0&-1&0&-2&0&-2\\
	\hline
$\beta^{0}-\beta^{2}$	&-2&0&-2&0&-3&0\\
	\hline
$\beta^{3}-\beta^{4}$	&0&-2&0&-2&0&-3\\
	\hline
	\end{tabular}
	\caption{The intersection form of $W(K,\beta)$ is represented by the matrix $-M$, where $M$ is the matrix of linking numbers above. }
	\label{int_form.tab}
\end{table}
	
\section{Tabulation of the Dihedral Linking Invariant}\label{tabulation.sec}

Throughout this section all data on knot polynomials is obtained from Knot Info~\cite{knotinfo}.  All data on mutants, specifically the tabulation of groups of mutant knots through 13 crossings, is due to Stoimenov \cite{stoimenovmutants}.
Our computations and tabulations are available at \cite{cahngithubdihedrallinking}.
\subsection{Tabulation}

Let $K$ be a knot with $\text{det}(K)\neq 1$ and $p$ a prime dividing $\text{det}(K)$.  Let $\Sigma_2(K)$ denote the double branched cover of $S^3$ along $K$.  Each $p$-coloring of $K$ corresponds to a nontrivial homomorphism $H_1(\Sigma_2(K);\mathbb{Z})\twoheadrightarrow \mathbb{Z}_p$. We call two such homomorphisms equivalent if they are related by multiplication by a unit in $\mathbb{Z}_p$.  Let $C_p$ denote the set of equivalence classes of $p$-colorings of $K$.  We define the {\it number of $p$-colorings of $K$} to be the number of equivalence classes of such homomorphisms $|C_p|$.  Observe that $|C_p|=\dfrac{p^n-1}{p-1}$, where $n$ is the number of $p^k$-cyclic summands in the elementary divisor decomposition of $H_1(\Sigma_2(K);\mathbb{Z})$.

Recall that given a $p$-coloring $\rho:\pi_1(S^3-K)\twoheadrightarrow \mathbb{Z}_p$ we define the corresponding dihedral linking invariant to be

 $$DLN(K,\rho) = \left\{ \text{lk}(K^i,K^j) | i\neq j \in \{0,1,\dots,(p-1)/2\}\right\}.$$
 
 The linking numbers between the various index-2 curves can be recovered from linking numbers with the index-1 curve $K^0$ (see \cite[p. 168]{CS1984linking}, remark above Theorem II). Therefore it suffices to study 
 
  $$DLN_0(K,\rho) = \left\{ \text{lk}(K^0,K^j) | j \in \{0,1,\dots,(p-1)/2\}\right\}.$$

   Next we define the {\it total dihedral linking invariant} of $K$ to be 

  $$DLN(K)=\bigcup_{p\geq 3 \text{ prime }}\bigcup_{[\rho]\in C_p} \{(p,DLN_0(K,\rho))\}.$$

Our tabulation of $DLN(K)$ for prime knots of crossing number $\leq 13$ is available in the file \verb|data/tabulation_with_colorings.csv| at~\cite{cahngithubdihedrallinking}.  To carry out the tabulation of $DLN$, we extract a list of knots, their braid words, and determinants from Knot Info \cite{knotinfo}. For each prime $p$ with $3\leq p\leq \text{det}(K)$, we check whether $p$ divides the determinant $\text{det}(K)$ of $K$. If so, we solve a system of linear equations (determined from the braid word) mod $p$ to determine all $p$-colorings of $K$, and then choose one representative of each equivalence class of the set of all $p$-colorings of $K$.  We then run \verb|dihedrallinking.py| to compute the values of $DLN_0(K,\rho)$ for each $p$-coloring $\rho$, extracting the overstrand and sign lists from the braid word. The tabulation of the linking invariant for 3-colorable knots up to 11 crossings is due to Perko ~\cite{perko1964thesis}.

We note that the $DLN(K)$ invariant can take on $\dfrac{p^n-1}{p-1}$ distinct values on the $\dfrac{p^n-1}{p-1}$  equivalence classes of $p$-colorings of $K$.  For $p=3$ and $n=2$, i.e., knots with four equivalence classes of 3-colorings, there are exactly 57 prime knots of crossing number $\leq 13$ whose four 3-dihedral linking invariants (each of which is a single element of $\mathbb{Q}\cup\{\infty\}$) are distinct.  The first such knot in the tables is 12\_n388, whose four linking numbers are 0, 2, 4, and 6.

For $p=5$ and $n=2$, i.e., knots with six equivalence classes of 5-colorings, there is just one prime knot of crossing number $\leq 13$, namely 13a\_3773, whose six 5-dihedral linking invariants  (each of which is a two-element subset of $\mathbb{Q}\cup\{\infty\}$) are distinct. These six values of the 5-dihedral linking invariant of 13a\_3773 are:

$$\{-54/5, 86/5\}, \{-34/5, 66/5\}, \{-24/5, 16/5\}, \{-14/5, 46/5\}, \{-4/5, -4/5\}, \{6/55, 166/55\}.$$

Through 13 crossings, there are five additional prime knots such that five of the six 5-dihedral linking invariants are distinct: 13n\_1026, 13n\_1102, 13n\_4173, 13a\_3581, and 13a\_3760. 

For $p=7$ and $n=2$, i.e., knots with eight equivalence classes of 7-colorings, there are no prime knots of crossing number $\leq 13$ whose eight 7-dihedral linking invariants (each of which is a three-element subset of $\mathbb{Q}\cup\{\infty\}$) are all distinct.  There are three knots with eight distinct 7-colorings through 13 crossings that have seven distinct 7-dihedral linking invariants: 13a\_1067, 13a\_2330, 13n\_2694.

\subsection{Comparing the total dihedral linking invariant to other invariants}

The {\it coloring invariant} $\text{col}(K)$ of $K$ is the set of pairs $(p,|C_p|)$, where $p$ is a prime dividing $\text{det}(K)$ and $|C_p|$ is the number of  $p$-colorings of $K$.  We first focus on distinguishing knots with the same coloring invariant, since we include the data of the coloring invariant in the total dihedral linking invariant.

Through 13 crossings, there are 673,562 pairs of knots with the same coloring invariant.  The total dihedral linking invariant distinguishes 663,344, or 98.48\%, of those pairs.

\begin{table}[htbp]
	\begin{tabular}{|l|l|l|l|}
		\hline
		Polynomial& Pairs with same $\text{col}(K)$  & Pairs distinguished &Percent of pairs \\
		Invariant (PI)&and same PI&by $\text{DLN}(K)$&distinguished by $\text{DLN}(K)$\\
		\hline
		Alexander& 17885&15886&88.8\%\\
		\hline
		Jones&4268&3179&74.5\%\\
		\hline
		HOMFLY-PT&2207&1163&52.7\%\\
		\hline
		Kauffman &1136&109&9.56\%\\
		\hline
		Khovanov unreduced $\mathbb{Z}$&4658&3568&76.6\%\\
		\hline
		\end{tabular}
		\caption{Counting pairs of knots of crossing number $\leq 13$ with equal coloring invariant and equal polynomial invariant that are distinguished by the total dihedral linking invariant.}\label{polynomial_all.tab}
\end{table}

\begin{table}[htbp]
	\begin{tabular}{|l|l|l|l|}
		\hline
		Polynomial&  {\it Non-mutant} pairs with   & Pairs distinguished &Percent of pairs \\
		Invariant (PI)& same $\text{col}(K)$ and same PI& by $\text{DLN}(K)$&distinguished by $\text{DLN}(K)$\\
		\hline
		Alexander&16760&15811&94.3\%\\
		\hline
		Jones&3244&3179&98.0\%\\
		\hline
		HOMFLY-PT&1183&1163&98.3\%\\
		\hline
		Kauffman&112&109& 97.3\%\\
		\hline
		Khovanov unreduced $\mathbb{Z}$&3634&3568&98.2\%\\
		\hline
		\end{tabular}
		\caption{Counting pairs of {\it non-mutant} knots of crossing number $\leq 13$ with equal coloring invariant and equal polynomial invariant that are distinguished by the total dihedral linking invariant.}
		\label{polynomial_non_mutant.tab}
\end{table}

We now compare the dihedral linking invariant to various polynomial invariants.  Table~\ref{polynomial_all.tab} demonstrates the ability of the dihedral linking invariant to distinguish knots that have the same coloring invariant, and the same Alexander, Jones, HOMFLY-PT, Kauffman, or unreduced Khovanov $\mathbb{Z}$-polynomial. 

Upon further investigation, many of the above pairs that the dihedral linking invariant does not distinguish are in fact mutants.   Therefore, we also test the ability of the dihedral linking invariant to distinguish {\it non-mutant} knots with the same coloring invariant and polynomial invariant.  The results of this experiment are shown in Table~\ref{polynomial_non_mutant.tab}.  The dihedral invariant is remarkably strong in this context, distinguishing between 94\% and 98\% of such pairs, depending on the polynomial chosen.  

While we have not found any mutant pairs distinguished by the dihedral linking invariant up to crossing number 13, we do not know whether the dihedral linking invariant is mutation-invariant in general. Perko has shown that linking invariants in other branched covers can distinguish mutants \cite{perko2016historical}.

In Table~\ref{same_homfly_and_dln.tab}, we list the 20 pairs of non-mutant knots of crossing number $\leq 13$ that have the same HOMFLY-PT polynomial and the same dihedral linking invariant (and hence the same coloring invariant).  Similarly, there are three pairs of knots of crossing number $\leq 13$ that have the same Kauffman polynomial and the same dihedral linking invariant, and are not mutants: 12a\_325 and 12a\_711; 12a\_390 and 12a\_672; and 13a\_1312 and 13a\_1592. These three pairs also have the same HOMFLY-PT polynomial, and so also appear in Table~\ref{same_homfly_and_dln.tab}.

We display selected pairs of knots with the same HOMFLY-PT polynomial and coloring invariant that are distinguished by the dihedral linking invariant in Tables~\ref{same_homfly_diff_dln1.tab} and \ref{same_homfly_diff_dln2.tab}. Namely, we display data for knots with crossing number $\leq 12$ such that the largest prime dividing $\text{det}(K)$ is 7. We do the same for the Kauffman polynomial in Table~\ref{same_kauffman_diff_dln.tab}.

\renewcommand{\arraystretch}{1.5}
\begin{table}[htbp]
\begin{tabular}{|c|c|}
	\hline
	$(12a\_325, 12a\_711)$&$\{3: [[-6/5]]\}$\\
	\hline
	$(12a\_390, 12a\_672)$&	$\{3: [[0]]\}$\\
	\hline
	$(12a\_606, 12a\_715)$&	$\{13: [[-2, -2, -2, 2, 2, 2]]\}$\\
	\hline
	$(12n\_144, 12n\_507)$&	$\{3: [[-2]], 5: [[0, 0]]\}$\\
	\hline
	$(12n\_210, 13n\_2051)$&$\{\}$\\
	\hline
	$(12n\_214, 13n\_2051)$&$\{\}$\\
	\hline
	$(12n\_399, 13n\_1922)$&	$\{3: [[2]]\}$\\
	\hline
	$(13a\_4342, 13a\_659)$&
	$\{3: [[-6]]\}$\\
	\hline
	$(13a\_3013, 13a\_737)$&	$\{5: [[-2, 2]], 7: [[0, 0, 0]]\}$\\
	\hline
	$(13a\_2825, 13a\_812)$&	$\{3: [[6]]\}$\\
	\hline
	$(13a\_1312, 13a\_1592)$&	$\{3: [[0]], 5: [[-2, 2]]\}$\\
	\hline
	$(13n\_4188, 13n\_436)$&	$\{3: [[0]]\}$\\
	\hline
	$(13n\_4188, 13n\_447)$&	$\{3: [[0]]\}$\\
	\hline
	$(13n\_4188, 13n\_636)$&	$\{3: [[0]]\}$\\
	\hline
	$(13n\_2124, 13n\_561)$&	$\{3: [[2]], 5: [[0, 0]]\}$\\
	\hline
	$(13n\_1132, 13n\_3246)$&	$\{3: [[-2]]\}$\\
	\hline
	$(13n\_1357, 13n\_1645)$&	$\{3: [[2]]\}$\\
	\hline
	$(13n\_1995, 13n\_4472)$&	$\{\}$\\
	\hline
	$(13n\_2656, 13n\_3980)$&	$\{\}$\\
	\hline
	$(13n\_3074, 13n\_4384)$&	$\{3: [[-2/3], [-2/3], [2/3], [\infty]]\}$\\
	\hline
\end{tabular}
\caption{The total dihedral linking invariant distinguishes all but the above 20 of the 1134 pairs of non-mutant knots with the same HOMFLY-PT polynomial and the same coloring invariant, through 13 crossings. Four such pairs have determinant 1 and hence no dihedral linking invariant. }
\label{same_homfly_and_dln.tab}
\end{table}

\newpage

\begin{table}\tiny
\begin{tabular}{|c|c|}
\hline \multicolumn{2}{|l|}{ \tiny $ 2 \, v^{6} z^{8} + v^{10} - 3 \, v^{8} - {\left(3 \, v^{8} - 10 \, v^{6} + 2 \, v^{4}\right)} z^{6} + v^{6} + {\left(v^{10} - 10 \, v^{8} + 18 \, v^{6} - 5 \, v^{4}\right)} z^{4} + 2 \, v^{4} + {\left(2 \, v^{10} - 10 \, v^{8} + 12 \, v^{6} - v^{4}\right)} z^{2} $} \\ 
\hline $12n\_749$&$7\_1$\\ \hline 
7-coloring 1: -6/29, 14/29, 26/29&7-coloring 1: 2, 2, 2\\ 
\hline \multicolumn{2}{|l|}{ \tiny $ 2 \, v^{6} z^{8} + v^{10} - 3 \, v^{8} - {\left(3 \, v^{8} - 10 \, v^{6} + 2 \, v^{4}\right)} z^{6} + v^{6} + {\left(v^{10} - 10 \, v^{8} + 18 \, v^{6} - 5 \, v^{4}\right)} z^{4} + 2 \, v^{4} + {\left(2 \, v^{10} - 10 \, v^{8} + 12 \, v^{6} - v^{4}\right)} z^{2} $} \\ 
\hline $10\_156$&$8\_16$\\ \hline 
5-coloring 1: -10, 18&5-coloring 1: -18, 6\\ 
7-coloring 1: -46/13, -38/13, 54/13&7-coloring 1: -18, 6, 22\\ 
\hline \multicolumn{2}{|l|}{ \tiny $ 2 \, v^{6} z^{8} + v^{10} - 3 \, v^{8} - {\left(3 \, v^{8} - 10 \, v^{6} + 2 \, v^{4}\right)} z^{6} + v^{6} + {\left(v^{10} - 10 \, v^{8} + 18 \, v^{6} - 5 \, v^{4}\right)} z^{4} + 2 \, v^{4} + {\left(2 \, v^{10} - 10 \, v^{8} + 12 \, v^{6} - v^{4}\right)} z^{2} $} \\ 
\hline $10\_35$&$12n\_48$\\ \hline 
7-coloring 1: -6, 2, 10&7-coloring 1: -2, 2, 2\\ 
\hline \multicolumn{2}{|l|}{ \tiny $ 2 \, v^{6} z^{8} + v^{10} - 3 \, v^{8} - {\left(3 \, v^{8} - 10 \, v^{6} + 2 \, v^{4}\right)} z^{6} + v^{6} + {\left(v^{10} - 10 \, v^{8} + 18 \, v^{6} - 5 \, v^{4}\right)} z^{4} + 2 \, v^{4} + {\left(2 \, v^{10} - 10 \, v^{8} + 12 \, v^{6} - v^{4}\right)} z^{2} $} \\ 
\hline $10\_40$&$12n\_412$\\ \hline 
3-coloring 1: -2&3-coloring 1: -2\\ 
5-coloring 1: -6, 18&5-coloring 1: -18/11, 30/11\\ 
\hline \multicolumn{2}{|l|}{ \tiny $ 2 \, v^{6} z^{8} + v^{10} - 3 \, v^{8} - {\left(3 \, v^{8} - 10 \, v^{6} + 2 \, v^{4}\right)} z^{6} + v^{6} + {\left(v^{10} - 10 \, v^{8} + 18 \, v^{6} - 5 \, v^{4}\right)} z^{4} + 2 \, v^{4} + {\left(2 \, v^{10} - 10 \, v^{8} + 12 \, v^{6} - v^{4}\right)} z^{2} $} \\ 
\hline $10\_106$&$12n\_369$\\ \hline 
3-coloring 1: 2&3-coloring 1: 2\\ 
5-coloring 1: -6, 18&5-coloring 1: -30/11, 18/11\\ 
\hline \multicolumn{2}{|l|}{ \tiny $ 2 \, v^{6} z^{8} + v^{10} - 3 \, v^{8} - {\left(3 \, v^{8} - 10 \, v^{6} + 2 \, v^{4}\right)} z^{6} + v^{6} + {\left(v^{10} - 10 \, v^{8} + 18 \, v^{6} - 5 \, v^{4}\right)} z^{4} + 2 \, v^{4} + {\left(2 \, v^{10} - 10 \, v^{8} + 12 \, v^{6} - v^{4}\right)} z^{2} $} \\ 
\hline $10\_141$&$12n\_438$\\ \hline 
3-coloring 1: 4&3-coloring 1: 14/5\\ 
7-coloring 1: -14, -2, 10&7-coloring 1: -50/13, 18/13, 74/13\\ 
\hline \multicolumn{2}{|l|}{ \tiny $ 2 \, v^{6} z^{8} + v^{10} - 3 \, v^{8} - {\left(3 \, v^{8} - 10 \, v^{6} + 2 \, v^{4}\right)} z^{6} + v^{6} + {\left(v^{10} - 10 \, v^{8} + 18 \, v^{6} - 5 \, v^{4}\right)} z^{4} + 2 \, v^{4} + {\left(2 \, v^{10} - 10 \, v^{8} + 12 \, v^{6} - v^{4}\right)} z^{2} $} \\ 
\hline $11a\_104$&$11a\_168$\\ \hline 
5-coloring 1: -10, 10/3&5-coloring 1: -50, 130\\ 
\hline \multicolumn{2}{|l|}{ \tiny $ 2 \, v^{6} z^{8} + v^{10} - 3 \, v^{8} - {\left(3 \, v^{8} - 10 \, v^{6} + 2 \, v^{4}\right)} z^{6} + v^{6} + {\left(v^{10} - 10 \, v^{8} + 18 \, v^{6} - 5 \, v^{4}\right)} z^{4} + 2 \, v^{4} + {\left(2 \, v^{10} - 10 \, v^{8} + 12 \, v^{6} - v^{4}\right)} z^{2} $} \\ 
\hline $11n\_21$&$11n\_4$\\ \hline 
7-coloring 1: -10, -2, 6&7-coloring 1: -18, 6, 26\\ 
\hline \multicolumn{2}{|l|}{ \tiny $ 2 \, v^{6} z^{8} + v^{10} - 3 \, v^{8} - {\left(3 \, v^{8} - 10 \, v^{6} + 2 \, v^{4}\right)} z^{6} + v^{6} + {\left(v^{10} - 10 \, v^{8} + 18 \, v^{6} - 5 \, v^{4}\right)} z^{4} + 2 \, v^{4} + {\left(2 \, v^{10} - 10 \, v^{8} + 12 \, v^{6} - v^{4}\right)} z^{2} $} \\ 
\hline $11n\_4$&$12n\_24$\\ \hline 
7-coloring 1: -18, 6, 26&7-coloring 1: -10, 2, 6\\ 
\hline \multicolumn{2}{|l|}{ \tiny $ 2 \, v^{6} z^{8} + v^{10} - 3 \, v^{8} - {\left(3 \, v^{8} - 10 \, v^{6} + 2 \, v^{4}\right)} z^{6} + v^{6} + {\left(v^{10} - 10 \, v^{8} + 18 \, v^{6} - 5 \, v^{4}\right)} z^{4} + 2 \, v^{4} + {\left(2 \, v^{10} - 10 \, v^{8} + 12 \, v^{6} - v^{4}\right)} z^{2} $} \\ 
\hline $11n\_21$&$12n\_24$\\ \hline 
7-coloring 1: -10, -2, 6&7-coloring 1: -10, 2, 6\\ 
\hline \multicolumn{2}{|l|}{ \tiny $ 2 \, v^{6} z^{8} + v^{10} - 3 \, v^{8} - {\left(3 \, v^{8} - 10 \, v^{6} + 2 \, v^{4}\right)} z^{6} + v^{6} + {\left(v^{10} - 10 \, v^{8} + 18 \, v^{6} - 5 \, v^{4}\right)} z^{4} + 2 \, v^{4} + {\left(2 \, v^{10} - 10 \, v^{8} + 12 \, v^{6} - v^{4}\right)} z^{2} $} \\ 
\hline $11n\_132$&$11n\_50$\\ \hline 
5-coloring 1: -2/11, 10/11&5-coloring 1: -2/3, 2\\ 
\hline \multicolumn{2}{|l|}{ \tiny $ 2 \, v^{6} z^{8} + v^{10} - 3 \, v^{8} - {\left(3 \, v^{8} - 10 \, v^{6} + 2 \, v^{4}\right)} z^{6} + v^{6} + {\left(v^{10} - 10 \, v^{8} + 18 \, v^{6} - 5 \, v^{4}\right)} z^{4} + 2 \, v^{4} + {\left(2 \, v^{10} - 10 \, v^{8} + 12 \, v^{6} - v^{4}\right)} z^{2} $} \\ 
\hline $11n\_58$&$12n\_449$\\ \hline 
5-coloring 1: -6, 10&5-coloring 1: -36, 16\\ 
7-coloring 1: -22, -6, 14&7-coloring 1: -70, -14, 46\\ 
\hline \multicolumn{2}{|l|}{ \tiny $ 2 \, v^{6} z^{8} + v^{10} - 3 \, v^{8} - {\left(3 \, v^{8} - 10 \, v^{6} + 2 \, v^{4}\right)} z^{6} + v^{6} + {\left(v^{10} - 10 \, v^{8} + 18 \, v^{6} - 5 \, v^{4}\right)} z^{4} + 2 \, v^{4} + {\left(2 \, v^{10} - 10 \, v^{8} + 12 \, v^{6} - v^{4}\right)} z^{2} $} \\ 
\hline $12a\_28$&$12a\_799$\\ \hline 
5-coloring 1: 0, 0&5-coloring 1: -6, 14\\ 
7-coloring 1: -2, 2, 2&7-coloring 1: -1186/503, -686/503, 2502/503\\ 
\hline \multicolumn{2}{|l|}{ \tiny $ 2 \, v^{6} z^{8} + v^{10} - 3 \, v^{8} - {\left(3 \, v^{8} - 10 \, v^{6} + 2 \, v^{4}\right)} z^{6} + v^{6} + {\left(v^{10} - 10 \, v^{8} + 18 \, v^{6} - 5 \, v^{4}\right)} z^{4} + 2 \, v^{4} + {\left(2 \, v^{10} - 10 \, v^{8} + 12 \, v^{6} - v^{4}\right)} z^{2} $} \\ 
\hline $12a\_210$&$12a\_623$\\ \hline 
3-coloring 1: 6&3-coloring 1: -42\\ 
7-coloring 1: -4, 4, 4&7-coloring 1: -2, 2, 2\\ 
\hline \multicolumn{2}{|l|}{ \tiny $ 2 \, v^{6} z^{8} + v^{10} - 3 \, v^{8} - {\left(3 \, v^{8} - 10 \, v^{6} + 2 \, v^{4}\right)} z^{6} + v^{6} + {\left(v^{10} - 10 \, v^{8} + 18 \, v^{6} - 5 \, v^{4}\right)} z^{4} + 2 \, v^{4} + {\left(2 \, v^{10} - 10 \, v^{8} + 12 \, v^{6} - v^{4}\right)} z^{2} $} \\ 
\hline $12a\_310$&$12a\_388$\\ \hline 
5-coloring 1: -2, 2&5-coloring 1: -2, 2\\ 
7-coloring 1: -350/13, -110/13, 250/13&7-coloring 1: -210/43, -30/43, 70/43\\ 
\hline
\end{tabular}
\label{same_homfly_diff_dln1.tab}
\caption{Knots with the same HOMFLY-PT polynomial and coloring invariant that are distinguished by the dihedral linking invariant. We restrict to knots of crossing number $\leq 12$ with determinant divisible only by primes $\leq 7$.}
\end{table}
\newpage
\begin{table}\tiny
\begin{tabular}{|c|c|}
\hline \multicolumn{2}{|l|}{ \tiny $ 2 \, v^{6} z^{8} + v^{10} - 3 \, v^{8} - {\left(3 \, v^{8} - 10 \, v^{6} + 2 \, v^{4}\right)} z^{6} + v^{6} + {\left(v^{10} - 10 \, v^{8} + 18 \, v^{6} - 5 \, v^{4}\right)} z^{4} + 2 \, v^{4} + {\left(2 \, v^{10} - 10 \, v^{8} + 12 \, v^{6} - v^{4}\right)} z^{2} $} \\ 
\hline $12a\_435$&$12a\_990$\\ \hline 
3-coloring 1: -4&3-coloring 1: -2\\ 
3-coloring 2: 0&3-coloring 2: 0\\ 
3-coloring 3: 0&3-coloring 3: 0\\ 
3-coloring 4: 4&3-coloring 4: 2\\ 
5-coloring 1: -10, 10&5-coloring 1: 0, 0\\ 
\hline \multicolumn{2}{|l|}{ \tiny $ 2 \, v^{6} z^{8} + v^{10} - 3 \, v^{8} - {\left(3 \, v^{8} - 10 \, v^{6} + 2 \, v^{4}\right)} z^{6} + v^{6} + {\left(v^{10} - 10 \, v^{8} + 18 \, v^{6} - 5 \, v^{4}\right)} z^{4} + 2 \, v^{4} + {\left(2 \, v^{10} - 10 \, v^{8} + 12 \, v^{6} - v^{4}\right)} z^{2} $} \\ 
\hline $12a\_511$&$12a\_988$\\ \hline 
5-coloring 1: -10, 30&5-coloring 1: -110/19, 50/19\\ 
\hline \multicolumn{2}{|l|}{ \tiny $ 2 \, v^{6} z^{8} + v^{10} - 3 \, v^{8} - {\left(3 \, v^{8} - 10 \, v^{6} + 2 \, v^{4}\right)} z^{6} + v^{6} + {\left(v^{10} - 10 \, v^{8} + 18 \, v^{6} - 5 \, v^{4}\right)} z^{4} + 2 \, v^{4} + {\left(2 \, v^{10} - 10 \, v^{8} + 12 \, v^{6} - v^{4}\right)} z^{2} $} \\ 
\hline $12a\_513$&$12a\_989$\\ \hline 
5-coloring 1: -14/3, 14&5-coloring 1: -238/19, 98/19\\ 
7-coloring 1: -2, -2, 2&7-coloring 1: -2, -2, 2\\ 
\hline \multicolumn{2}{|l|}{ \tiny $ 2 \, v^{6} z^{8} + v^{10} - 3 \, v^{8} - {\left(3 \, v^{8} - 10 \, v^{6} + 2 \, v^{4}\right)} z^{6} + v^{6} + {\left(v^{10} - 10 \, v^{8} + 18 \, v^{6} - 5 \, v^{4}\right)} z^{4} + 2 \, v^{4} + {\left(2 \, v^{10} - 10 \, v^{8} + 12 \, v^{6} - v^{4}\right)} z^{2} $} \\ 
\hline $12a\_587$&$12a\_977$\\ \hline 
5-coloring 1: 10/3, 6&5-coloring 1: 42/11, 54/11\\ 
7-coloring 1: -2, 2, 6&7-coloring 1: -1374/251, 322/251, 1622/251\\ 
\hline \multicolumn{2}{|l|}{ \tiny $ 2 \, v^{6} z^{8} + v^{10} - 3 \, v^{8} - {\left(3 \, v^{8} - 10 \, v^{6} + 2 \, v^{4}\right)} z^{6} + v^{6} + {\left(v^{10} - 10 \, v^{8} + 18 \, v^{6} - 5 \, v^{4}\right)} z^{4} + 2 \, v^{4} + {\left(2 \, v^{10} - 10 \, v^{8} + 12 \, v^{6} - v^{4}\right)} z^{2} $} \\ 
\hline $12a\_701$&$12a\_987$\\ \hline 
3-coloring 1: -4&3-coloring 1: -4\\ 
3-coloring 2: -4&3-coloring 2: -4\\ 
3-coloring 3: -2&3-coloring 3: -2\\ 
3-coloring 4: -2&3-coloring 4: -2\\ 
5-coloring 1: -18, 6&5-coloring 1: -6, 6\\ 
\hline \multicolumn{2}{|l|}{ \tiny $ 2 \, v^{6} z^{8} + v^{10} - 3 \, v^{8} - {\left(3 \, v^{8} - 10 \, v^{6} + 2 \, v^{4}\right)} z^{6} + v^{6} + {\left(v^{10} - 10 \, v^{8} + 18 \, v^{6} - 5 \, v^{4}\right)} z^{4} + 2 \, v^{4} + {\left(2 \, v^{10} - 10 \, v^{8} + 12 \, v^{6} - v^{4}\right)} z^{2} $} \\ 
\hline $12a\_1226$&$12a\_916$\\ \hline 
3-coloring 1: -2&3-coloring 1: -2\\ 
7-coloring 1: -54, 18, 78&7-coloring 1: -30/13, -18/13, 42/13\\ 
\hline \multicolumn{2}{|l|}{ \tiny $ 2 \, v^{6} z^{8} + v^{10} - 3 \, v^{8} - {\left(3 \, v^{8} - 10 \, v^{6} + 2 \, v^{4}\right)} z^{6} + v^{6} + {\left(v^{10} - 10 \, v^{8} + 18 \, v^{6} - 5 \, v^{4}\right)} z^{4} + 2 \, v^{4} + {\left(2 \, v^{10} - 10 \, v^{8} + 12 \, v^{6} - v^{4}\right)} z^{2} $} \\ 
\hline $12a\_1136$&$12a\_1224$\\ \hline 
3-coloring 1: 2&3-coloring 1: 2\\ 
7-coloring 1: -30, -6, 18&7-coloring 1: -6, 6, 6\\ 
\hline \multicolumn{2}{|l|}{ \tiny $ 2 \, v^{6} z^{8} + v^{10} - 3 \, v^{8} - {\left(3 \, v^{8} - 10 \, v^{6} + 2 \, v^{4}\right)} z^{6} + v^{6} + {\left(v^{10} - 10 \, v^{8} + 18 \, v^{6} - 5 \, v^{4}\right)} z^{4} + 2 \, v^{4} + {\left(2 \, v^{10} - 10 \, v^{8} + 12 \, v^{6} - v^{4}\right)} z^{2} $} \\ 
\hline $12n\_413$&$12n\_597$\\ \hline 
3-coloring 1: 2&3-coloring 1: 2\\ 
5-coloring 1: -18/31, 66/31&5-coloring 1: 6/19, 18/19\\ 
\hline \multicolumn{2}{|l|}{ \tiny $ 2 \, v^{6} z^{8} + v^{10} - 3 \, v^{8} - {\left(3 \, v^{8} - 10 \, v^{6} + 2 \, v^{4}\right)} z^{6} + v^{6} + {\left(v^{10} - 10 \, v^{8} + 18 \, v^{6} - 5 \, v^{4}\right)} z^{4} + 2 \, v^{4} + {\left(2 \, v^{10} - 10 \, v^{8} + 12 \, v^{6} - v^{4}\right)} z^{2} $} \\ 
\hline $12n\_201$&$12n\_551$\\ \hline 
3-coloring 1: 0&3-coloring 1: 0\\ 
5-coloring 1: -10/41, 162/41&5-coloring 1: -50/79, 286/79\\ 
7-coloring 1: -6, 2, 14&7-coloring 1: -54/13, 46/13, 90/13\\ 
\hline \multicolumn{2}{|l|}{ \tiny $ 2 \, v^{6} z^{8} + v^{10} - 3 \, v^{8} - {\left(3 \, v^{8} - 10 \, v^{6} + 2 \, v^{4}\right)} z^{6} + v^{6} + {\left(v^{10} - 10 \, v^{8} + 18 \, v^{6} - 5 \, v^{4}\right)} z^{4} + 2 \, v^{4} + {\left(2 \, v^{10} - 10 \, v^{8} + 12 \, v^{6} - v^{4}\right)} z^{2} $} \\ 
\hline $12n\_244$&$12n\_338$\\ \hline 
5-coloring 1: 2, 6&5-coloring 1: -2, 6\\ 
\hline \multicolumn{2}{|l|}{ \tiny $ 2 \, v^{6} z^{8} + v^{10} - 3 \, v^{8} - {\left(3 \, v^{8} - 10 \, v^{6} + 2 \, v^{4}\right)} z^{6} + v^{6} + {\left(v^{10} - 10 \, v^{8} + 18 \, v^{6} - 5 \, v^{4}\right)} z^{4} + 2 \, v^{4} + {\left(2 \, v^{10} - 10 \, v^{8} + 12 \, v^{6} - v^{4}\right)} z^{2} $} \\ 
\hline $12n\_417$&$12n\_694$\\ \hline 
5-coloring 1: 0, 8&5-coloring 1: -34/19, 50/19\\ 
7-coloring 1: 18/43, 66/43, 118/43&7-coloring 1: -2/13, 38/13, 46/13\\ 
\hline \multicolumn{2}{|l|}{ \tiny $ 2 \, v^{6} z^{8} + v^{10} - 3 \, v^{8} - {\left(3 \, v^{8} - 10 \, v^{6} + 2 \, v^{4}\right)} z^{6} + v^{6} + {\left(v^{10} - 10 \, v^{8} + 18 \, v^{6} - 5 \, v^{4}\right)} z^{4} + 2 \, v^{4} + {\left(2 \, v^{10} - 10 \, v^{8} + 12 \, v^{6} - v^{4}\right)} z^{2} $} \\ 
\hline $12n\_420$&$12n\_636$\\ \hline 
3-coloring 1: -22/9&3-coloring 1: -2/3\\ 
3-coloring 2: -2/3&3-coloring 2: -2/3\\ 
3-coloring 3: $\infty$&3-coloring 3: -2/3\\ 
3-coloring 4: $\infty$&3-coloring 4: -2/3\\ 
\hline
	\end{tabular}
	\label{same_homfly_diff_dln2.tab}
	\caption{Knots with the same HOMFLY-PT polynomial and coloring invariant that are distinguished by the total dihedral linking invariant (continued). We restrict to knots of crossing number $\leq 12$ with determinant divisible only by primes $\leq 7$.}
	\end{table}

\begin{table}\tiny
	\begin{tabular}{|c|c|}
	\hline \multicolumn{2}{|l|}{ \tiny $ 3 \, {\left(a + \frac{1}{a}\right)} z^{11} + 3 \, {\left(3 \, a^{2} + \frac{2}{a^{2}} + 5\right)} z^{10} + {\left(10 \, a^{3} + 7 \, a + \frac{1}{a} + \frac{4}{a^{3}}\right)} z^{9} + {\left(5 \, a^{4} - 27 \, a^{2} - \frac{23}{a^{2}} + \frac{1}{a^{4}} - 56\right)} z^{8} + {\left(a^{5} - 31 \, a^{3} - 56 \, a - \frac{41}{a} - \frac{17}{a^{3}}\right)} z^{7}$} \\
	\multicolumn{2}{|l|}{$- {\left(6 \, a^{4} - 26 \, a^{2} - \frac{23}{a^{2}} + \frac{4}{a^{4}} - 59\right)} z^{6} + {\left(6 \, a^{5} + 38 \, a^{3} + 73 \, a + \frac{64}{a} + \frac{23}{a^{3}}\right)} z^{5} + {\left(4 \, a^{6} + a^{4} - 20 \, a^{2} - \frac{5}{a^{2}} + \frac{5}{a^{4}} - 27\right)} z^{4} $}\\
\multicolumn{2}{|l|}{$	- 2 \, a^{4} + {\left(a^{7} - 6 \, a^{5} - 25 \, a^{3} - 38 \, a - \frac{32}{a} - \frac{12}{a^{3}}\right)} z^{3} - {\left(a^{6} - 3 \, a^{4} - 12 \, a^{2} - \frac{1}{a^{2}} + \frac{2}{a^{4}} - 11\right)} z^{2} - 5 \, a^{2} + {\left(3 \, a^{5} + 7 \, a^{3} + 8 \, a + \frac{6}{a} + \frac{2}{a^{3}}\right)} z - \frac{1}{a^{2}} - 3 $} \\ 
	\hline $11n\_21$&$11n\_4$\\ \hline 
	7-coloring 1: -10, -2, 6&7-coloring 1: -18, 6, 26\\ 
	\hline \multicolumn{2}{|l|}{ \tiny $ 3 \, {\left(a + \frac{1}{a}\right)} z^{11} + 3 \, {\left(3 \, a^{2} + \frac{2}{a^{2}} + 5\right)} z^{10} + {\left(10 \, a^{3} + 7 \, a + \frac{1}{a} + \frac{4}{a^{3}}\right)} z^{9} + {\left(5 \, a^{4} - 27 \, a^{2} - \frac{23}{a^{2}} + \frac{1}{a^{4}} - 56\right)} z^{8} + {\left(a^{5} - 31 \, a^{3} - 56 \, a - \frac{41}{a} - \frac{17}{a^{3}}\right)} z^{7} $}\\
	\multicolumn{2}{|l|}{$- {\left(6 \, a^{4} - 26 \, a^{2} - \frac{23}{a^{2}} + \frac{4}{a^{4}} - 59\right)} z^{6} + {\left(6 \, a^{5} + 38 \, a^{3} + 73 \, a + \frac{64}{a} + \frac{23}{a^{3}}\right)} z^{5} + {\left(4 \, a^{6} + a^{4} - 20 \, a^{2} - \frac{5}{a^{2}} + \frac{5}{a^{4}} - 27\right)} z^{4} - 2 \, a^{4}$}\\
	\multicolumn{2}{|l|}{$ + {\left(a^{7} - 6 \, a^{5} - 25 \, a^{3} - 38 \, a - \frac{32}{a} - \frac{12}{a^{3}}\right)} z^{3} - {\left(a^{6} - 3 \, a^{4} - 12 \, a^{2} - \frac{1}{a^{2}} + \frac{2}{a^{4}} - 11\right)} z^{2} - 5 \, a^{2} + {\left(3 \, a^{5} + 7 \, a^{3} + 8 \, a + \frac{6}{a} + \frac{2}{a^{3}}\right)} z - \frac{1}{a^{2}} - 3 $} \\ 
	\hline $12a\_310$&$12a\_388$\\ \hline 
	5-coloring 1: -2, 2&5-coloring 1: -2, 2\\ 
	7-coloring 1: -350/13, -110/13, 250/13&7-coloring 1: -210/43, -30/43, 70/43\\ 
	\hline \multicolumn{2}{|l|}{ \tiny $ 3 \, {\left(a + \frac{1}{a}\right)} z^{11} + 3 \, {\left(3 \, a^{2} + \frac{2}{a^{2}} + 5\right)} z^{10} + {\left(10 \, a^{3} + 7 \, a + \frac{1}{a} + \frac{4}{a^{3}}\right)} z^{9} + {\left(5 \, a^{4} - 27 \, a^{2} - \frac{23}{a^{2}} + \frac{1}{a^{4}} - 56\right)} z^{8} + {\left(a^{5} - 31 \, a^{3} - 56 \, a - \frac{41}{a} - \frac{17}{a^{3}}\right)} z^{7}$}\\
	\multicolumn{2}{|l|}{$ - {\left(6 \, a^{4} - 26 \, a^{2} - \frac{23}{a^{2}} + \frac{4}{a^{4}} - 59\right)} z^{6} + {\left(6 \, a^{5} + 38 \, a^{3} + 73 \, a + \frac{64}{a} + \frac{23}{a^{3}}\right)} z^{5} + {\left(4 \, a^{6} + a^{4} - 20 \, a^{2} - \frac{5}{a^{2}} + \frac{5}{a^{4}} - 27\right)} z^{4} - 2 $}\\
	\multicolumn{2}{|l|}{$ \, a^{4} + {\left(a^{7} - 6 \, a^{5} - 25 \, a^{3} - 38 \, a - \frac{32}{a} - \frac{12}{a^{3}}\right)} z^{3} - {\left(a^{6} - 3 \, a^{4} - 12 \, a^{2} - \frac{1}{a^{2}} + \frac{2}{a^{4}} - 11\right)} z^{2} - 5 \, a^{2} + {\left(3 \, a^{5} + 7 \, a^{3} + 8 \, a + \frac{6}{a} + \frac{2}{a^{3}}\right)} z - \frac{1}{a^{2}} - 3 $} \\ 
	\hline $12n\_420$&$12n\_636$\\ \hline 
	3-coloring 1: -22/9&3-coloring 1: -2/3\\ 
	3-coloring 2: -2/3&3-coloring 2: -2/3\\ 
	3-coloring 3: $\infty$&3-coloring 3: -2/3\\ 
	3-coloring 4: $\infty$&3-coloring 4: -2/3\\
\hline
\end{tabular}
\caption{Knots with the same Kauffman polynomial and coloring invariant that are distinguished by the total dihedral linking invariant.  We restrict to knots of crossing number $\leq 12$ with determinant divisible only by primes $\leq 7$. Note that these knots also have the same HOMFLY-PT polynomial.}
\label{same_kauffman_diff_dln.tab}
\end{table}
\providecommand{\bysame}{\leavevmode\hbox to3em{\hrulefill}\thinspace}
\providecommand{\MR}{\relax\ifhmode\unskip\space\fi MR }
\providecommand{\MRhref}[2]{%
  \href{http://www.ams.org/mathscinet-getitem?mr=#1}{#2}
}
\providecommand{\href}[2]{#2}


\begin{thebibliography}{10}

\bibitem{bankwitz1934viergeflechte}
Carl Bankwitz and Hans~Georg Schumann, \emph{{\"U}ber viergeflechte},
  Abhandlungen aus dem Mathematischen Seminar der Universit{\"a}t Hamburg,
  vol.~10, Springer, 1934, pp.~263--284.

\bibitem{birman1980seifert}
Joan~S Birman and Julian Eisner, \emph{Seifert and {T}hrelfall, a {T}extbook of
  {T}opology}, vol.~89, Academic Press, 1980.

\bibitem{cahngithubdihedrallinking}
Patricia Cahn, Elise Catania, Sarangoo Chimgee, Olivia Del~Guercio, and Jack
  Kendrick, \emph{Dihedral linking invariants},
  https://github.com/patriciacahn/DihedralLinkingInvariants, 2021.

\bibitem{cahnkjuchukova2017singbranchedcovers}
Patricia Cahn and Alexandra Kjuchukova, \emph{Singular branched covers of
  four-manifolds}, arXiv preprint arXiv:1710.11562.

\bibitem{cahnkju2018genus}
\bysame, \emph{The dihedral genus of a knot}, Algebraic and Geometric Topology
  \textbf{20} (2020), 1939--1963.

\bibitem{cahnkjuchukova2018computing}
\bysame, \emph{Computing ribbon obstructions for colored knots}, Fundamenta
  Mathematicae \textbf{253} (2021), 155--173.

\bibitem{cahnkjuchukova2016linking}
\bysame, \emph{Linking numbers in three-manifolds}, Discrete and Computational
  Geometry \textbf{66} (2021), 435--463.

\bibitem{cahn2023linking}
\bysame, \emph{Linking in cyclic branched covers and satellite
  (non)-homomorphisms}, arXiv preprint arXiv:2308.05856 (2023).

\bibitem{cahn2023algorithms}
Patricia Cahn, Gordana Matic, and Benjamin Ruppik, \emph{Algorithms for
  computing invariants of trisected branched covers}, Algebraic and Geometric
  Topology, to appear. arXiv preprint arXiv:2308.11689 (2023).

\bibitem{CS1975invariants}
Sylvain Cappell and Julius Shaneson, \emph{Invariants of 3-manifolds}, Bulletin
  of the American Mathematical Society \textbf{81} (1975), no.~3, 559--562.

\bibitem{CS1984linking}
\bysame, \emph{Linking numbers in branched covers}, Contemporary Mathematics
  \textbf{35} (1984), 165--179.

\bibitem{carter2014three}
J~Scott Carter, Daniel~S Silver, and Susan G.~Williams, \emph{Three dimensions
  of knot coloring}, The American Mathematical Monthly \textbf{121} (2014),
  no.~6, 506--514.

\bibitem{geske2021signatures}
Christian Geske, Alexandra Kjuchukova, and Julius~L Shaneson, \emph{Signatures
  of topological branched covers}, International Mathematics Research Notices
  \textbf{2021} (2021), no.~6, 4605--4624.

\bibitem{hartley1983identifying}
Richard Hartley, \emph{Identifying non-invertible knots}, Topology \textbf{22}
  (1983), no.~2, 137--145.

\bibitem{hartley1977covering}
Richard Hartley and Kunio Murasugi, \emph{Covering linkage invariants}, Canad.
  J. Math \textbf{29} (1977), 1312--1339.

\bibitem{kjuchukova2018dihedral}
Alexandra Kjuchukova, \emph{Dihedral branched covers of four-manifolds},
  Advances in Mathematics \textbf{332} (2018), 1--33.

\bibitem{kjorr2020admissible}
Alexandra Kjuchukova and Kent Orr, \emph{Knots arising as singularities on
  branched covers between four-manifolds}, In preparation.

\bibitem{kjuchukova2026extending}
Alexandra Kjuchukova and Kent~E Orr, \emph{Extending quotients of knot groups
  over surfaces in $ {B}^4$}, arXiv preprint arXiv:2604.00460 (2026).

\bibitem{litherland1980formula}
R~Litherland, \emph{A formula for the {C}asson--{G}ordon invariant of a knot},
  preprint (1980).

\bibitem{knotinfo}
Charles Livingston and Allison~H. Moore, \emph{Knotinfo: Table of knot
  invariants}, URL: knotinfo.org, January 2026.

\bibitem{neuwirth2016chapter}
Lee~Paul Neuwirth, \emph{Chapter {III.} {C}ombinatorial covering space theory
  for 3-manifolds}, Knot Groups. Annals of Mathematics Studies.(AM-56), Volume
  56, Princeton University Press, 2016, pp.~9--24.

\bibitem{perko1964thesis}
Kenneth Perko, \emph{An invariant of certain knots}, Undergraduate Thesis
  (1964).

\bibitem{perko1976dihedral}
\bysame, \emph{On dihedral covering spaces of knots}, Inventiones mathematicae
  \textbf{34} (1976), no.~2, 77--82.

\bibitem{perko2016historical}
Kenneth~A Perko, \emph{Historical highlights of non-cyclic knot theory},
  Journal of Knot Theory and Its Ramifications \textbf{25} (2016), no.~03,
  1640010.

\bibitem{reidemeister1938knot}
K~Reidemeister, \emph{Knot theory, translated by {L.} {B}oron, {C.}
  {C}hristiansen and {B.} {S}mith, {BCS} {A}ssociates, 1983. {O}riginally
  published as {K}notentheorie}, Ergebnisse der Mathematik und Ihrer
  Grenzgebiete,(Alte Folge) \textbf{1} (1938), no.~1.

\bibitem{KnotJob}
Dirk Sch{\"u}tz, \emph{Knotjob}, 2025, Software available at
  https://www.maths.dur.ac.uk/users/dirk.schuetz/knotjob.html.

\bibitem{stoimenovmutants}
A.~Stoimenov, \emph{Knot data tables}, URL:
  http://stoimenov.net/stoimeno/homepage/ptab/, January 2026.

\end{thebibliography}
\end{document}